\documentclass[11pt]{article}

\usepackage[margin=1in]{geometry}
\usepackage{amsmath,amssymb,amsthm,mathtools}
\usepackage{microtype}
\usepackage{enumitem}
\usepackage{xspace}
\usepackage{aliascnt}
\usepackage{hyperref}
\usepackage[dvipsnames]{xcolor}
\usepackage[capitalize,nameinlink]{cleveref}
\hypersetup{
	colorlinks=true,
	pdfpagemode=UseNone,
    citecolor=OliveGreen,
    linkcolor=NavyBlue,
    urlcolor=Magenta,
	pdfstartview=FitW
}

\theoremstyle{plain}
\newtheorem{theorem}{Theorem}[section]
\newaliascnt{lemma}{theorem}
\newtheorem{lemma}[lemma]{Lemma}
\aliascntresetthe{lemma}
\newaliascnt{proposition}{theorem}
\newtheorem{proposition}[proposition]{Proposition}
\aliascntresetthe{proposition}
\newaliascnt{corollary}{theorem}
\newtheorem{corollary}[corollary]{Corollary}
\aliascntresetthe{corollary}

\theoremstyle{definition}
\newaliascnt{definition}{theorem}
\newtheorem{definition}[definition]{Definition}
\aliascntresetthe{definition}
\newaliascnt{example}{theorem}
\newtheorem{example}[example]{Example}
\aliascntresetthe{example}

\theoremstyle{remark}
\newaliascnt{remark}{theorem}
\newtheorem{remark}[remark]{Remark}
\aliascntresetthe{remark}
\newaliascnt{note}{theorem}

\aliascntresetthe{note}

\crefname{theorem}{theorem}{theorems}
\Crefname{theorem}{Theorem}{Theorems}
\crefname{lemma}{lemma}{lemmas}
\Crefname{lemma}{Lemma}{Lemmas}
\crefname{proposition}{proposition}{propositions}
\Crefname{proposition}{Proposition}{Propositions}
\crefname{corollary}{corollary}{corollaries}
\Crefname{corollary}{Corollary}{Corollaries}
\crefname{definition}{definition}{definitions}
\Crefname{definition}{Definition}{Definitions}
\crefname{example}{example}{examples}
\Crefname{example}{Example}{Examples}
\crefname{exercise}{exercise}{exercises}
\Crefname{exercise}{Exercise}{Exercises}
\crefname{remark}{remark}{remarks}
\Crefname{remark}{Remark}{Remarks}
\crefname{note}{note}{notes}
\Crefname{note}{Note}{Notes}

\def\*#1{\mathbf{#1}}
\def\+#1{\mathcal{#1}}
\def\-#1{\mathrm{#1}}
\newcommand{\bb}{\mathbb}

\def\!#1{\mathsf{#1}}
\def\@#1{\mathscr{#1}}

\newcommand{\tp}[1]{\left(#1\right)}
\newcommand{\stp}[1]{\left[#1\right]}
\newcommand{\set}[1]{\left\{#1\right\}}
\newcommand{\abs}[1]{\left|#1\right|}
\newcommand{\norm}[1]{\left\|#1\right\|}
\newcommand{\ceil}[1]{\left\lceil#1\right\rceil}

\newcommand{\ang}[1]{\left\langle#1\right\rangle}
\newcommand{\inner}[2]{\left\langle#1,\,#2\right\rangle}

\newcommand{\cmid}{\,:\,}

\newcommand{\wt}{\widetilde}

\newcommand{\mean}[1]{\!{mean}\tp{#1}}
\newcommand{\var}[1]{\!{var}\tp{#1}}
\newcommand{\cov}[1]{\!{cov}\tp{#1}}

\newcommand{\dd}{\,\-d}
\newcommand{\eps}{\varepsilon}
\newcommand{\defeq}{:=}

\DeclareMathOperator{\conv}{conv}
\DeclareMathOperator{\diag}{diag}
\DeclareMathOperator{\dom}{dom}
\DeclareMathOperator{\supp}{supp}

\def\Emery{\'Emery\xspace}

\def\Follmer{F\"ollmer\xspace}

\def\Ito{It\^o\xspace}

\def\Poincare{Poincar\'e\xspace}

\def\Schrodinger{Schr\"odinger\xspace}

\title{High-Dimensional Ultra-Log-Concave Distributions}
\author{Zongchen Chen\thanks{School of Computer Science, Georgia Institute of Technology, Atlanta, GA, USA. \href{mailto:chenzongchen@gatech.edu}{chenzongchen@gatech.edu}} 
\and Sihan Wang\thanks{Zhiyuan College, Shanghai Jiao Tong University, Shanghai, China. \href{mailto:wangsihan_leo@sjtu.edu.cn}{wangsihan\_leo@sjtu.edu.cn}}}
\date{\today}

\begin{document}
\pagenumbering{gobble}

\maketitle

\begin{abstract}
    Ultra-log-concave distributions are ubiquitous in probability, combinatorics, and statistical mechanics and have been studied extensively. In this paper, we introduce a quantitative high-dimensional extension of this notion, called $\delta$-ultra-log-concavity, for probability measures on $\mathbb{N}^d$ with downward closed support. When \(\delta=1\), this notion coincides with the class studied by Gurvits (2009) via strongly log-concave generating functions, and with the class defined by Anari, Oveis Gharan, and Vinzant (2021) via completely log-concave generating functions; in one dimension, it reduces to classical ultra-log-concavity. We establish several functional inequalities, including a weighted Poincar\'e inequality, a discrete Brascamp--Lieb inequality, and a weighted Wu-type modified log-Sobolev inequality. Our approach combines integrated Bakry--\'Emery calculus for a canonical birth-death chain with Poisson stochastic localization, which arises as the time reversal of coordinatewise binomial thinning. We further establish concentration of measure, maximum-entropy principles, and several closure properties for ultra-log-concave measures, and develop applications to queueing models, polymatroids, antiferromagnetic Potts models, and hardcore models. Finally, a lattice scaling limit of the discrete theory yields Poincar\'e and Brascamp--Lieb inequalities for Laguerre diffusions.
\end{abstract}

\newpage

\tableofcontents

\newpage
\pagenumbering{arabic}

\section{Introduction}
\label{sec:introduction}

Sampling from a high-dimensional distribution specified by local constraints and unnormalized weights is a central task in probability, theoretical computer science, and statistical physics. Among the most practical sampling methods are local Markov chains, which modify only a small part of the current configuration at each step. While the locality makes such chains easy to implement, analyzing their global convergence is often challenging.

A standard approach to proving rapid mixing, i.e., polynomial-time convergence of a Markov chain, is to establish a functional inequality, such as a \Poincare or modified log-Sobolev inequality. In continuous spaces, Bakry--\Emery theory provides a powerful local-to-global paradigm and has found broad applications in log-concave sampling \cite{Che26}. If a probability measure on \(\bb R^d\) has density proportional to \(e^{-V}\) and satisfies \(\nabla^2 V \succeq \rho \*I\), then the Bakry--\Emery criterion yields \Poincare and logarithmic Sobolev inequalities with constant \(\rho\) \cite{BE85,BGL14}. Retaining the full Hessian of the potential \(V\), rather than only its uniform lower bound, leads to the sharper Brascamp--Lieb variance inequality \cite{BL76}. These inequalities turn local curvature information into global control of fluctuations and entropy decay, and they underlie quantitative convergence guarantees for Langevin-type dynamics and the associated sampling algorithms.

The local-to-global paradigm has also been remarkably successful on discrete fixed-size spaces. For a homogeneous distribution on \(\binom{[d]}{k}\), the natural Markov chain associated with it is the down-up walk, which, in each step, removes a random element and then adds back a random feasible element.
Strong log-concavity (equivalently, the Lorentzian property) of the generating polynomial implies sharp functional inequalities and rapid mixing of the down-up walk \cite{ALOVii,BH20,CGM21}. This theory is closely connected to high-dimensional expanders, where global expansion is controlled through local links \cite{Opp18,KO20,AL20}. In particular, after conditioning on \(k-2\) selected elements, one obtains a rank-two link that captures the relevant local second-order information, playing a role analogous to the Hessian in the continuous theory.

This framework extends naturally from fixed-size subsets to integer-valued points on the fixed-size simplex \(\sum_{i=1}^d x_i=k\) \cite{ALOVii,BH20}.  
For a finitely supported inhomogeneous distribution, one can sometimes reduce to the homogeneous setting by introducing a slack coordinate \(x_0 \defeq k - \sum_{i=1}^d x_i\), corresponding to a homogenization of the generating polynomial; see, e.g., \cite{ALGV24,RSJ19,ALOVV21,CGZZ24}. Such reductions are effective when the support is finite and yield Markov chain algorithms for sampling. However, they generally do not directly establish rapid mixing of the canonical coordinatewise birth-death process associated with the measure.

Distributions with unbounded support on \(\bb N^d\), however, arise naturally in applications such as queueing theory and population dynamics. This raises the question of what the appropriate notion of log-concavity should be in this setting. The canonical local dynamics is a coordinatewise birth-death chain, in which each transition adds or removes one particle from a single coordinate. One could truncate the state space and homogenize the resulting finite distribution, but this approach still does not directly analyze the birth-death chain and becomes cumbersome even in one dimension. More importantly, such a black-box reduction obscures the local curvature matrix---the discrete analogue of the Hessian---that should directly govern the behavior of the original birth-death process.

Meanwhile, in the one-dimensional case, the discrete notion of log-concavity is well understood. A probability measure \(\mu\) on \(\bb N\) is \emph{log-concave} if
\[
\mu\tp{n}^2 \ge \mu\tp{n - 1} \mu\tp{n + 1}.
\]
For example, geometric distributions and uniform distributions on an interval \(\set{0, 1, \dots, L}\) are log-concave. The stronger notion of \emph{ultra-log-concavity} (ULC) asks that the factorially reweighted sequence
\[
a\tp{n} \defeq n! \mu\tp{n}
\]
is log-concave. Poisson and binomial distributions are basic examples of ULC measures.

Log-concave and ultra-log-concave sequences arise throughout probability, combinatorics, and statistical mechanics, notably in the enumeration of matchings and the analysis of monomer--dimer systems, the study of independent-set sequences of matroids, and rank statistics of determinantal measures and exclusion processes \cite{HL72,Lig97,BBL09,ALGV24}. For a classical survey of log-concavity in combinatorics, see \cite{Sta89,Bre89}. The factorial weight in the definition of ultra-log-concavity is structural rather than cosmetic: up to normalization, \(a\) is the density of \(\mu\) relative to the Poisson distribution with mean \(1\). From this perspective, ULC can be viewed as a discrete, Poisson-relative analogue of strong log-concavity in continuous space, where the corresponding notion could be interpreted in terms of log-concavity relative to a Gaussian reference measure.

This Poisson geometry also singles out a natural reversible local dynamics. Particles die independently, giving death rate \(n\) at state \(n\), while reversibility determines the birth rate
\[
b\tp{n} \defeq \frac{\tp{n + 1} \mu\tp{n + 1}}{\mu\tp{n}} = \frac{a\tp{n + 1}}{a\tp{n}}.
\]
In this formulation, ULC is precisely the monotonicity of the birth rates. A constant birth rate yields a Poisson stationary distribution and recovers the classical \(M/M/\infty\) queue. More broadly, one-dimensional ULC has long been connected to \Poincare and modified log-Sobolev inequalities, Poisson concentration, and Poisson maximum-entropy principles \cite{DJ13,Joh17,Wu00,ALS25,LRS25,AMM21,Joh07}.

This leads to our guiding question: \emph{Is there a high-dimensional notion of ultra-log-concavity on \(\bb N^d\) that is local, compatible with Poisson geometry, and strong enough to control the natural birth-death dynamics?} Recent work on downward closed subsets of the hypercube identifies the appropriate local-matrix condition in the binary setting \cite{CCCYZ25,GJMPPS26}. We extend this condition from \(\set{0, 1}^d\) to general downward closed subsets of \(\bb N^d\), leading to a quantitative notion of \(\delta\)-ULC. This notion retains a concrete local curvature matrix and reduces to classical ULC when \(d = 1\) and \(\delta = 1\). We show that local curvature lower bounds imply a weighted \Poincare inequality, a discrete Brascamp--Lieb inequality, and a weighted Wu-type modified log-Sobolev inequality. We further characterize \(1\)-ULC in terms of spectral independence and log-concavity of the probability-generating function, and derive Poisson domination, concentration inequalities, maximum-entropy principles, and several closure properties for this class. Applications include queueing models, discrete polymatroids, antiferromagnetic Potts models, and hardcore models. Finally, we relate \(\delta\)-ULC distributions on lattice points to the \emph{Laguerre diffusion} on \(\bb R^d\) through a scaling limit.

\subsection{High-Dimensional Ultra-Log-Concave Distributions}

Let \(\mu\) be a probability measure on \(\bb N^d\) with downward closed support: if \(\mu\tp{\*n}>0\) and \(\*m\le \*n\) coordinatewise, then \(\mu\tp{\*m}>0\). Set
\[
a\tp{\*n} \defeq \*n!\mu\tp{\*n},
\]
where \(\*n! \defeq\prod_{i=1}^d n_i!\), and define the subset of active coordinates, i.e., those that can be increased without leaving the support, by
\[
\+I_{\*n} \defeq \set{i \in \stp{d} \cmid \*n + \*e_i \in \supp\tp{\mu}}.
\]
For \(\*n\in\supp\tp{\mu}\), define matrices indexed by the active coordinates \(\+I_{\*n}\) by
\[
\*C_{\*n} \defeq \tp{1 - \frac{a\tp{\*n} a\tp{\*n + \*e_i + \*e_j}}{a\tp{\*n + \*e_i} a\tp{\*n + \*e_j}}}_{i,j\in\+I_{\*n}},\qquad \*D_{\*n} \defeq \diag \tp{\frac{a\tp{\*n}}{a\tp{\*n + \*e_i}}}_{i\in\+I_{\*n}}.
\]
For \(\delta > 0\), we say that \(\mu\) is \emph{\(\delta\)-ultra-log-concave}, or \(\delta\)-ULC, if
\begin{equation}
    \*C_{\*n} + \tp{1 - \delta} \*D_{\*n} \succeq \*O, \quad \forall \*n\in\supp\tp{\mu}.
    \label{eq:intro-ulc-condition}
\end{equation}
We call \(1\)-ULC measures simply ULC. 
Note that \(\delta > 1\) is permissible.

The measure \(\mu\) is associated with a natural birth-death process. Its reversible generator is
\[
\+L f\tp{\*n} \defeq \sum_{i=1}^d b_i\tp{\*n} \nabla_i^+f\tp{\*n} + \sum_{i=1}^d n_i \nabla_i^-f\tp{\*n}, \qquad b_i\tp{\*n} \defeq \frac{\tp{n_i + 1} \mu\tp{\*n + \*e_i}}{\mu\tp{\*n}} = \frac{a\tp{\*n + \*e_i}}{a\tp{\*n}},
\]
where \(\nabla_i^+ f\tp{\*n} = f\tp{\*n + \*e_i} - f\tp{\*n}\) and \(\nabla_i^- f\tp{\*n} = f\tp{\*n - \*e_i} - f\tp{\*n}\), with unavailable moves assigned rate zero. The entry \(\tp{\*C_{\*n}}_{ij}\) measures how a birth in coordinate \(i\) changes the neighboring birth ratio in coordinate \(j\), while \(\*D_{\*n}\) fixes the scale set by the Poisson reference dynamics. Consequently,
\[
\*K_{\*n} \defeq \*C_{\*n} + \*D_{\*n} \succeq \delta \*D_{\*n}
\]
is the discrete analogue of a Hessian lower bound for the negative Poisson-relative potential.

Several consistency properties support this interpretation. When \(d = 1\), the ULC condition is precisely the log-concavity of \(a\tp{n} = n! \mu\tp{n}\). On a downward closed subset of \(\set{0,1}^d\), it reduces to the local matrix condition of \cite{CCCYZ25,GJMPPS26}. Moreover, ULC is preserved under several natural operations, including convolution (\Cref{thm:ulc-convolution-closure}), marginalization (\Cref{prop:ulc-marginalization-closure}), coordinatewise binomial thinning (\Cref{prop:ulc-thinning-closure}), and particlewise stochastic projection (\Cref{thm:ulc-stochastic-projection-closure}).

Most importantly, \Cref{thm:ulc-generating-function-characterization} shows that \(\mu\) is ULC if and only if its probability-generating function
\[
g_\mu\tp{\boldsymbol{\theta}} \defeq \sum_{\*n\in\bb N^d} \mu\tp{\*n} \boldsymbol{\theta}^{\*n}, \qquad \boldsymbol{\theta}^{\*n} \defeq \theta_1^{n_1} \cdots \theta_d^{n_d}
\]
is strongly log-concave, and also if and only if it is completely log-concave. In addition, \(\mu\) admits an equivalent characterization in terms of \(1\)-\emph{spectral independence} under every tilted weighted pinning. This extends \cite[Proposition~3.8]{CCCYZ25}, which establishes the corresponding equivalence for distributions on \(\set{0, 1}^d\) with downward closed support. It also yields, in the present probabilistic setting, a local curvature characterization of the strongly log-concave generating functions introduced by Gurvits \cite{Gur09a} and the completely log-concave generating functions introduced by Anari, Oveis Gharan, and Vinzant \cite{AOGV21}. Thus, the same condition admits three complementary descriptions: locally through birth ratios and curvature matrices, globally through covariances of tilted pinnings, and analytically through the probability-generating function.

\subsection{Functional Inequalities}
\label{subsec:func-ineq}

We first recall some basics of functional inequalities. For a reversible Markov generator \(\+L\) with stationary measure \(\mu\), the associated Dirichlet form is \(\+E\tp{f, g} \defeq -\*E_{\mu}\stp{f \+L g}\), where \(\*E_{\mu}\) denotes expectation under \(\mu\). A \Poincare inequality with constant \(\gamma\) is
\[
\gamma \, \*{Var}_{\mu}\stp{f} \le \+E\tp{f, f}, \quad \forall f,
\]
where \(\*{Var}_{\mu}\stp{f} \defeq \*E_{\mu}\stp{\tp{f - \*E_{\mu}\stp{f}}^2}\). A modified log-Sobolev inequality with constant \(\rho\) is
\[
\rho \, \*{Ent}_{\mu}\stp{f} \le \+E\tp{\log f, f}, \quad \forall f > 0,
\]
where \(\*{Ent}_{\mu}\stp{f} \defeq \*E_{\mu}\stp{f \log f} - \*E_{\mu}\stp{f} \log \*E_{\mu}\stp{f}\).

We establish several functional inequalities and structural properties for high-dimensional ULC measures.
Our first result converts the local curvature lower bound into a global spectral-gap estimate.

\begin{theorem}[Weighted \Poincare inequality; informal version of \Cref{thm:ulc-poincare-inequality}]
    Every \(\delta\)-ULC probability measure \(\mu\) satisfies
    \begin{equation}
        \delta\,\*{Var}_\mu\stp{f} \le \*E_{\*n\sim\mu}\stp{\sum_{i=1}^d b_i\tp{\*n} \tp{\nabla_i^+ f\tp{\*n}}^2}, \quad \forall f
        \label{eq:intro-weighted-poincare}
    \end{equation}
    Equivalently, the canonical birth-death generator \(\+L\) has spectral gap at least \(\delta\); the spectral gap is the smallest nonzero eigenvalue of \(-\+L\) in \(L^2\tp{\mu}\).
    \label{thm:intro-weighted-poincare}
\end{theorem}

The constant \(\delta\) in \eqref{eq:intro-weighted-poincare} is sharp for exact \(\delta\)-ULC distributions, namely, those for which all the inequalities in \eqref{eq:intro-ulc-condition} hold with equality. Examples include products of negative binomial distributions when \(\delta < 1\), Poisson distributions when \(\delta = 1\), and binomial distributions when \(\delta > 1\) (\Cref{ex:ulc-exact-1d,ex:ulc-product-measure}).

The weighted \Poincare inequality is dimension-free and is expressed intrinsically in terms of the birth rates of the measure. It also yields the covariance bound
\[
\cov{\mu} \preceq \frac{1}{\delta} \diag\tp{\mean{\mu}}, 
\]
whenever the relevant moments are finite (\Cref{cor:ulc-cov-bound}). Here, \(\mean{\mu}\) and \(\cov{\mu}\) denote the mean vector and covariance matrix, respectively, and \(\diag\tp{\mean{\mu}}\) is the diagonal matrix of coordinate means. From the sampling perspective, \(\delta\)-ULC therefore implies \((1 / \delta)\)-spectral independence of the measure with respect to Poisson stochastic localization.

The weighted inequality also immediately implies an unweighted \Poincare inequality by uniformly bounding the birth rates:
\[ 
\delta \, \*{Var}_\mu\stp{f} \le \sup_{\*n \in \supp\tp{\mu}, \, i \in \stp{d}} b_i\tp{\*n} \, \*E_{\mu}\stp{\norm{\nabla^+ f}_2^2}, \quad \forall f,
\]
whenever the supremum is finite. In one dimension, the monotonicity of the birth rates for ULC distributions gives a particularly simple bound:
\[  
\*{Var}_\mu\stp{f} \le \frac{\mu\tp{1}}{\mu\tp{0}} \*E_{\mu}\stp{\tp{\nabla^+ f}^2}, \quad \forall f.
\]
This recovers \cite[Theorem~1.5]{Joh17} specialized to ULC measures (\Cref{ex:one-dim-ulc-poincare}).

The preceding results use the \(\delta\)-ULC condition, or equivalently the uniform curvature lower bound \(\*K_{\*n} \succeq \delta\*D_{\*n}\). Retaining the full curvature matrix \(\*K_{\*n}\) leads to a sharper functional inequality.

\begin{theorem}[Discrete Brascamp--Lieb inequality; informal version of \Cref{thm:discrete-brascamp-lieb}]
Suppose that \(\*K_{\*n}\) is positive definite for every \(\*n \in \supp\tp{\mu}\). Then
\[
\*{Var}_\mu\stp{f} \le \*E_{\*n \sim \mu}\stp{\tp{\nabla^+ f\tp{\*n}}_{\+I_{\*n}}^{\top} \*K_{\*n}^{-1} \tp{\nabla^+ f\tp{\*n}}_{\+I_{\*n}}}, \quad \forall f.
\]
\label{thm:intro-brascamp-lieb}
\end{theorem}

This is an exact discrete-space counterpart of the continuous Brascamp--Lieb inequality: local directions with larger curvature make a smaller contribution to the variance. Matrix domination by \(\delta \*D_{\*n}\) recovers the preceding weighted \Poincare inequality for \(\delta\)-ULC measures, while applying the theorem to linear functions gives a matrix covariance bound (\Cref{cor:discrete-brascamp-lieb-cov-bound}).

At the entropy level, we prove a stronger, Wu-type modified log-Sobolev inequality under the ULC condition itself (i.e., \(\delta = 1\)). This distinction is essential: when \(\delta < 1\), the additional diagonal term appearing in the \(\delta\)-ULC condition is not preserved under arbitrary exponential tilts, whereas ULC is.

\begin{theorem}[Weighted Wu-type modified log-Sobolev inequality; informal version of \Cref{thm:ulc-modified-log-sobolev-inequality}]
Every ULC probability measure \(\mu\) with finite first moments satisfies the weighted Wu-type modified log-Sobolev inequality with constant \(1\):
\begin{equation}
    \*{Ent}_{\mu}\stp{f} \le \*E_{\*n \sim \mu}\stp{\sum_{i=1}^d b_i\tp{\*n} \Psi\tp{f\tp{\*n},\nabla_i^+f\tp{\*n}}} \le \+E\tp{\log f,f}, \quad \forall f > 0,
    \label{eq:intro-weighted-wu-inequality}
\end{equation}
where
\[
\Psi\tp{u,v} \defeq \tp{u + v} \log\tp{u + v} - u \log u - v \tp{\log u + 1},
\]
and \(\+E\) denotes the Dirichlet form associated with \(\+L\). In particular, \(\+L\) satisfies the modified log-Sobolev inequality with constant \(1\).
\label{thm:intro-wu-mlsi}
\end{theorem}

The constant \(1\) in \eqref{eq:intro-weighted-wu-inequality} is sharp for product Poisson measures. 

The inequality \eqref{eq:intro-weighted-wu-inequality} is of Wu type: its integrand contains the Bregman divergence term \(\Psi\tp{f, \nabla_i^+f}\) associated with the entropy function \(u \mapsto u \log u\). This is stronger than the usual modified log-Sobolev inequality involving \(\+E\tp{\log f, f}\), which in turn governs entropy contraction along the semigroup generated by \(\+L\). Inequalities of this form were introduced by Wu \cite{Wu00} for Poisson point processes.

As with the weighted \Poincare inequality, the weighted Wu-type inequality yields an unweighted bound by uniformly controlling the birth rates. In one dimension, monotonicity of the birth rates for ULC measures gives the following result:
\[  
\*{Ent}_\mu\stp{f} \le \frac{\mu\tp{1}}{\mu\tp{0}} \*E_{\mu}\stp{\Psi\tp{f, \nabla^+ f}}, \quad \forall f \ge 0.
\]
This recovers \cite[Theorem~1.3]{Joh17} specialized to ULC measures (\Cref{ex:one-dim-ulc-mlsi}).

The weighted Wu-type modified log-Sobolev inequality with constant \(1\) also yields multivariate analogues of several classical Poisson extremality and concentration results for ULC measures on \(\bb N^d\). Write \(\*m \defeq \mean{\mu}\) and \(\pi_{\*m} \defeq \bigotimes_{i = 1}^d \-{Pois}\tp{m_i}\).

\begin{theorem}[Informal version of \Cref{thm:poisson-laplace-transform-domination,thm:poisson-orthant-concentration,thm:poisson-maximum-entropy}]
If \(\mu\) satisfies the weighted Wu-type modified log-Sobolev inequality with constant \(1\), i.e., \eqref{eq:intro-weighted-wu-inequality} holds, and has finite positive mean \(\*m\), then the following statements hold.
\begin{enumerate}
    \item \textbf{Poisson domination:} For every \(\boldsymbol{\lambda}\in\bb R^d\),
    \begin{equation}
        \*E_{\*X \sim \mu}\stp{e^{\inner{\boldsymbol{\lambda}}{\*X}}} \le \exp\tp{\sum_{i = 1}^d m_i \tp{e^{\lambda_i} - 1}} = \*E_{\*Z \sim \pi_{\*m}}\stp{e^{\inner{\boldsymbol{\lambda}}{\*Z}}}.
    \label{eq:intro-mgf-comparison}
    \end{equation}

    \item \textbf{Concentration of measure:} For \(\*t \in \bb R^d\) satisfying \(t_i > -m_i\) for every \(i \in \stp{d}\), define the orthant event
    \[
    \+O_{\*t} \defeq \bigcap_{\substack{i \in \stp{d} \\ t_i > 0}} \set{X_i - m_i \ge t_i} \cap \bigcap_{\substack{i \in \stp{d} \\ t_i < 0}} \set{X_i - m_i \le t_i}.
    \]
    Then
    \begin{equation}
        \*{Pr}_{\mu}\stp{\+O_{\*t}} \le \exp\tp{-\sum_{i = 1}^d \frac{t_i^2}{2 m_i} h\tp{\frac{t_i}{m_i}}},
        \label{eq:intro-concentration}
    \end{equation}
    where
    \[  
    h\tp{x} \defeq 2 \frac{\tp{1 + x} \log\tp{1 + x} - x}{x^2}, \quad x \ge -1.
    \]

    \item \textbf{Maximum-entropy principle:} For a distribution \(\nu\) on \(\bb N^d\), define its Shannon entropy as \(H\tp{\nu} \defeq -\sum_{\*n \in \supp\tp{\nu}} \nu\tp{\*n} \log \nu\tp{\*n}\). Then
    \begin{equation}
        H\tp{\pi_{\*m}} - H\tp{\mu} \ge \-{KL}\tp{\mu \| \pi_{\*m}} \ge 0.
        \label{eq:intro-maximum-entropy}
    \end{equation}
    Here, \(\-{KL}\tp{\mu \| \pi_{\*m}}\) denotes the relative entropy of \(\mu\) with respect to \(\pi_{\*m}\). Consequently, among all ULC measures with mean \(\*m\), the product Poisson measure \(\pi_{\*m}\) is the unique entropy maximizer.
\end{enumerate}
\label{thm:intro-concentration}
\end{theorem}

The moment-generating-function comparison \eqref{eq:intro-mgf-comparison} follows from a Herbst-type argument. Applying generic Chernoff bounds to this comparison yields the Poisson-type concentration inequalities in \eqref{eq:intro-concentration}. The comparison \eqref{eq:intro-mgf-comparison} also leads to the maximum-entropy principle of product Poisson measures \eqref{eq:intro-maximum-entropy}.

\subsection{Examples and Applications}

Here, we present some examples of \(\delta\)-ULC measures and applications to these models.
The common feature of the following examples is that a model-specific local matrix calculation can be separated from the general functional-inequality machinery. Once the local curvature matrix is controlled, the global conclusions follow from the general theory in \Cref{subsec:func-ineq}.

\paragraph{Two interacting queues.}

We consider a simple interacting variant of the classical \(M/M/\infty\) queue; see \Cref{ex:two-queue-model} for the detailed calculation. The \(M/M/\infty\) model is a birth-death process describing a queue with infinitely many servers, so that each customer begins service immediately upon arrival. In this model, the arrival rate is constant, while the departure rate is proportional to the current queue length \cite{Asm03}. The stationary measure of the \(M/M/\infty\) queue is Poisson. 

We consider two infinite-server queues, with the service rate normalized so that a queue of length \(n\) has total departure rate \(n\). The arrival rates depend on the relative queue lengths: customers join the longer queue at rate \(b_{\-{long}}\), reflecting a preference for popularity, and the shorter queue at rate \(b_{\-{short}}\), reflecting a preference for balance. When the two queues have equal length, we assign birth rate \(b_{\-{long}}\) to both. 
The stationary distribution is
\[
\mu\tp{n_1, n_2} \propto \frac{b_{\-{long}}^{\max\set{n_1, n_2}} b_{\-{short}}^{\min\set{n_1, n_2}}}{n_1! n_2!}.
\]

We show that, when \(\abs{b_{\-{long}} - b_{\-{short}}} < 1\), this distribution is \(\delta\)-ULC with \(\delta \defeq 1 - \abs{b_{\-{long}} - b_{\-{short}}}\). Consequently, the corresponding birth-death process described above has spectral gap at least \(\delta\). In the special case \(b_{\-{long}} = b_{\-{short}} = b\), the interaction disappears: the system reduces to two independent \(M/M/\infty\) queues with birth rate \(b\), and its stationary distribution is the product Poisson measure \(\-{Pois}\tp{b} \otimes \-{Pois}\tp{b}\), which is ULC.

\paragraph{Discrete polymatroids.}

Discrete polymatroids generalize matroids from subset configurations in \(\set{0, 1}^d\) to vectors in \(\bb N^d\). They play a central role in combinatorics and discrete convex analysis \cite{HH02,BHKKL25}. Concretely, let \(r:2^{\stp{d}}\to\bb N\) be a normalized, monotone, and submodular rank function. The associated discrete polymatroid consists of the integer points:
\[
P_r \defeq \set{\*n \in \bb N^d \cmid \forall S \subseteq \stp{d}, \, \sum_{i \in S} n_i \le r\tp{S}}.
\]
Thus, discrete polymatroids allow arbitrary integral multiplicities subject to submodular rank constraints; see \Cref{subsec:polymatroid} for the detailed treatment.

For activities \(\boldsymbol{\lambda}\in\bb R_{>0}^d\), consider the product-Poisson measure conditioned on \(P_r\):
\[
\mu_{r,\boldsymbol{\lambda}}\tp{\*n} \propto \frac{\boldsymbol{\lambda}^{\*n}}{\*n!}, \quad \*n \in P_r.
\]
We prove that \(\mu_{r,\boldsymbol{\lambda}}\) is ULC (\Cref{thm:polymatroid-independent-vectors-ulc}). Consequently, its birth-death chain satisfies the modified log-Sobolev inequality with constant \(1\). Given a membership oracle for \(P_r\), a lazy uniformization of the birth-death chain can be implemented efficiently and, for fixed \(\boldsymbol{\lambda}\), has mixing time
\[
O_{\boldsymbol{\lambda}}\tp{\tp{d + R}\tp{\log R + \log\log\tp{d + R} + \log\tp{1 / \eps}}},
\]
where \(R \defeq r\tp{\stp{d}}\) is the rank of the polymatroid (\Cref{cor:polymatroid-birth-death-mixing}). This extends the local-walk perspective developed for matroids from set-valued independent sets to integer-valued independent vectors \cite{ALOVii,CGM21,CGZZ24}.

\paragraph{Antiferromagnetic Potts models.}

The \(q\)-state antiferromagnetic Potts measure on a graph \(G = \tp{V, E}\) with interaction parameter \(0 < B < 1\) is defined by
\[
\mu_{G, B}\tp{\sigma} \propto B^{m_G\tp{\sigma}}, \quad \sigma \in \stp{q}^V,
\]
where
\[
m_G\tp{\sigma} \defeq \abs{\set{vw \in E \cmid \sigma_v = \sigma_w}}
\]
is the number of monochromatic edges in \(\sigma\); see \Cref{subsec:q-potts} for a detailed treatment.

Fix \(q\ge3\). Let \(G\) have maximum degree \(\Delta\), and set \(\lambda^{\star} \defeq -\lambda_{\min}\tp{\*A_G}\). Under the downward closed encoding in \eqref{eq:downward-closed-potts-encoding}, the embedded Potts measure is \(\delta\)-ULC whenever
\[
\lambda^{\star} \tp{1 - B} B^{1 - \Delta} < 1, \qquad \delta \defeq 1 - \lambda^{\star} \tp{1 - B} B^{1 - \Delta}.
\]
Comparison with the local birth-death chain gives a spectral gap of at least \(\delta B^\Delta / \tp{q \abs{V}}\) for the heat-bath Glauber dynamics (\Cref{thm:antiferromagnetic-potts-ulc} and \Cref{cor:q-potts-rapid-mixing}). For random \(\Delta\)-regular graphs and fixed \(q\), the resulting polynomial mixing regime includes, for large \(\Delta\), parameters beyond the tree-uniqueness threshold (\Cref{cor:q-potts-random-regular-mixing}); see \cite{GGY18,BBR23,BBBR23} for relevant tree thresholds.

\paragraph{Hardcore model.}

The hardcore model on a graph \(G = \tp{V, E}\) with fugacity \(\lambda > 0\) is defined by
\[
\mu_{G, \lambda}\tp{I} \propto \lambda^{\abs{I}}, \quad I \in \+I\tp{G},
\]
where \(\+I\tp{G}\) is the collection of all independent sets of \(G\). If \(G\) is \(\Delta\)-regular and
\[
\lambda \le \frac{1}{-\lambda_{\min}\tp{\*A_G} - 1},
\]
then the Brascamp--Lieb inequality gives the explicit variance bound on the size of the random independent set \(I \sim \mu_{G, \lambda}\):
\[
\frac{1}{\abs{V}} \*{Var}\stp{\abs{I}} \le \frac{\lambda}{1+\tp{\Delta+1}\lambda};
\]
see \Cref{thm:hardcore-cardinality-variance} and \Cref{subsec:variance-hardcore} for the full analysis. For recent complementary work on variances in the hardcore model, see \cite{DST25,ZX26}.

\subsection{Our Approach}

Our proof has two main components. First, the integrated Bakry--\Emery calculus yields the weighted \Poincare and Brascamp--Lieb inequalities (\Cref{thm:intro-weighted-poincare,thm:intro-brascamp-lieb}). Second, Poisson stochastic localization yields the weighted Wu-type modified log-Sobolev inequality (\Cref{thm:intro-wu-mlsi}). Poisson domination, concentration of measure, and the maximum-entropy principle in \Cref{thm:intro-concentration} then follow from this modified log-Sobolev inequality. We also highlight that a scaling limit transfers the results for linear-death chains to the Laguerre diffusion in continuous space.

\paragraph{Integrated Bakry--\Emery calculus.}

The Bakry--\Emery method originates in the continuous diffusion setting \cite{BE85}, but its classical pointwise curvature criterion is often too restrictive for discrete jump processes. We instead use the integrated Bakry--\Emery approach, which is well suited to proving sharp \Poincare inequalities for birth-death chains; see, e.g., \cite{KKO13,Joh17,GJMPPS26}.

For the canonical birth-death generator with linear death rates, we derive an exact integrated discrete Bochner identity (\Cref{thm:integrated-gamma-2-form}). After integration against the stationary law, the first-order contribution is governed by \(\*K_{\*n}\), while the remaining second-difference term is nonnegative. Discarding that term and using \(\*K_{\*n} \succeq \delta \*D_{\*n}\) proves the weighted \Poincare inequality. This is an integrated argument: it exploits cancellations under \(\mu\) and does not assert a pointwise \(\-{CD}\tp{\delta, \infty}\) condition as in the continuous Bakry--\Emery theory. To obtain the Brascamp--Lieb inequality, we instead retain the full matrix, solve a Poisson equation, and use a matrix Cauchy--Schwarz inequality. The method extends the discrete Bochner program of \cite{BCDPP06,CDPP09,KKO13,Joh17,GJMPPS26} beyond uniform curvature bounds.

\paragraph{Poisson stochastic localization.}

Stochastic localization provides a general mechanism for establishing functional inequalities by interpolating between a target measure and progressively simpler conditional laws. Within the localization-scheme framework, this approach has proved particularly effective in the analysis of Glauber dynamics for Ising models and the proximal sampler for strongly log-concave measures \cite{CE25}. On the Boolean hypercube \(\set{0, 1}^d\), negative-fields localization introduces noise through Bernoulli thinning, whereby each \(1\) is independently erased to \(0\); the resulting noising--denoising Markov chain is precisely the field dynamics \cite{CFYZ24}.

We extend this mechanism from the hypercube to the count space \(\bb N^d\) by replacing coordinatewise Bernoulli thinning with binomial thinning. The time reversal of this noising process, equivalently its denoising process, is exactly the Poisson--\Follmer process. This noising--denoising formulation has already appeared in the context of generative modeling \cite{SBP26}. Conditional on the thinned observation \(\*X_t=\*x\), the posterior law is
\[
\nu_{t,\*x}\tp{\*n} \propto \mu\tp{\*n} \prod_{i=1}^d\binom{n_i}{x_i} t^{x_i}\tp{1-t}^{n_i-x_i}, \quad \*n\ge\*x.
\]
These posterior measures form a linear-tilt localization process and satisfy exact dissipation identities for conditional \(\Phi\)-entropies (\Cref{thm:poisson-linear-tilt,thm:phi-entropy-evolution}). The reverse-time dynamics is a denoising birth process, while the accelerated infinitesimal noising--denoising chain converges to the birth-death generator \(\+L\) (\Cref{thm:associated-limiting-process}). ULC is preserved under the tilts and weighted pinnings arising along the flow, and the resulting covariance bounds yield approximate entropy conservation and, ultimately, the weighted Wu-type inequality (\Cref{prop:sufficient-condition-approximate-conservation}). This construction connects stochastic localization \cite{Eld13,EAM22,CE25,AKV24,STZ25} with Poisson--\Follmer processes and binomial-thinning flows \cite{KL18,ALS25,LRS25,SBP26}.

\paragraph{Laguerre diffusion.}

Finally, a lattice scaling of the linear-death chain gives the Laguerre diffusion. Under this scaling, the generators, Dirichlet forms, and local curvature matrices converge to their continuous Gamma-reference counterparts. Passing the discrete functional inequalities to the limit produces the Laguerre \Poincare and Brascamp--Lieb inequalities in \Cref{thm:laguerre-functional-inequalities}.

\subsection{Related Work}
\label{subsec:related-work}

\paragraph{Ultra-log-concave measures.}

The classical theory of ULC sequences, including their connections to convolution, negative dependence, and combinatorial probability, was developed in \cite{Walk76,Lig97,Pem00}; see also the survey \cite{SW14}. \Poincare inequalities for discrete log-concave and ULC measures were studied in \cite{DJ13,Joh17}, while Wu's modified log-Sobolev inequality on Poisson space initiated the corresponding entropy theory \cite{Wu00}. More recent work establishes and sharpens Wu-type modified log-Sobolev inequalities, Poisson-type concentration bounds, and stability results for one-dimensional ULC measures \cite{ALS25,LRS25,AMM21}. The Poisson maximum-entropy property for sums of Bernoulli random variables and related ULC classes was studied in \cite{Joh07}. One-dimensional ULC measures also satisfy several additional properties, including a monotone entropic law of thin numbers \cite{Yu09}, concavity of entropy under thinning \cite{YJ09}, and domination by the Poisson law in convex order \cite{Yu10}. In \Cref{sec:concentration-maximum-entropy,sec:ulc-closure-properties}, we establish multivariate analogues of several of these results for high-dimensional ULC measures.

\paragraph{Log-concave polynomials.}

Log-concave and Lorentzian polynomials provide a powerful framework for studying homogeneous discrete measures, matroids, and polymatroids \cite{ALOVii,BH20,CGM21}. Gurvits introduced strongly log-concave entire functions with nonnegative coefficients and developed multivariate Newton-like inequalities for this class \cite{Gur09a}. Anari, Oveis Gharan, and Vinzant introduced completely log-concave polynomials and applied them to obtain a deterministic approximation algorithm for counting matroid bases \cite{AOGV21}. In the homogeneous setting, both notions are equivalent to the Lorentzian property \cite{BH20}. More recently, \cite{CCCYZ25} established a local matrix characterization---corresponding to ULC in our terminology---of strong log-concavity for probability measures supported on finite downward closed set systems. We extend this characterization to \(\bb N^d\) and prove that ULC is equivalent to both strong log-concavity and complete log-concavity of the probability-generating function (\Cref{thm:ulc-generating-function-characterization}).

\paragraph{Integrated Bakry--\Emery criterion.}

The integrated Bakry--\Emery criterion provides a useful characterization of the \Poincare inequality and a flexible route to sharp spectral-gap estimates. A general framework based on this criterion was developed in \cite{BCDPP06} and applied to interacting particle systems. It was subsequently extended to infinite-volume systems by \cite{KKO13}, with applications to Glauber dynamics for point processes. More recently, \cite{GJMPPS26} adapted the argument of \cite{KKO13} to probability measures supported on finite downward closed set systems, yielding a short proof of the \Poincare criterion of \cite{CCCYZ25}. Our local matrix condition extends this criterion from binary to count-valued configurations and, more importantly, endows the local matrix with a curvature interpretation. Through the point-process lift developed in \Cref{sec:point-processes}, the resulting criterion also recovers the earlier criterion for Glauber dynamics of point processes in \cite{KKO13}.

Separately, \cite{Joh17} used an argument similar to \cite{KKO13} for one-dimensional birth-death chains with unit birth rates and introduced the notion of \(c\)-log-concavity. The integrated Bakry--\Emery criterion has also recently been used to establish rapid mixing in several other settings, including spin systems \cite{LOG25,GZ26}, the SK model \cite{Wan26}, and colorings \cite{CL26}.

\paragraph{Localization and binomial thinning flows.}

Stochastic localization, introduced by Eldan \cite{Eld13}, has become a central tool for converting local covariance or entropy estimates into global functional inequalities \cite{EAM22,CE25,AKV24}. Negative-fields localization and the associated field dynamics have likewise proved effective for analyzing Markov-chain mixing on the hypercube \cite{CFYZ24,AKJVP22,CE25,CCYZ25,CCCYZ25}. Our process is instead adapted to the Poisson reference measure and unbounded count spaces. It is related to Poisson--F\"ollmer processes, Poisson transport, and binomial thinning interpolations \cite{KL18,ALS25,LRS25,SBP26}; \Cref{sec:poisson-bridge} makes the Poisson \Schrodinger-bridge connection explicit,\footnote{Upon completing the manuscript, we noticed that \cite[Section~2.4]{SBP26} contains the same observations. We include the explicit computation for the \Schrodinger bridge in \Cref{sec:poisson-bridge} for completeness.} while \Cref{sec:trickle-down} gives an alternative localization proof of the weighted \Poincare inequality via the trickle-down theorem \cite{AKV24}.

\subsection{Organization}
\label{subsec:organization}

\Cref{sec:preliminaries} collects notation, background on Markov chains and Markov processes, and preliminaries on functional inequalities. \Cref{sec:bakry-emery-birth-death} develops the integrated Bakry--\Emery calculus, introduces high-dimensional ultra-log-concavity, proves the weighted \Poincare and discrete Brascamp--Lieb inequalities, and presents an application to the hardcore model. \Cref{sec:poisson-stochastic-localization} constructs Poisson stochastic localization, identifies its limiting birth-death process, proves the weighted Wu-type modified log-Sobolev inequality, and relates ULC measures to strongly log-concave generating functions. \Cref{sec:concentration-maximum-entropy} derives the Poisson comparison, concentration, and maximum-entropy principle. \Cref{sec:applications} establishes ultra-log-concavity for polymatroids and Potts models and derives the resulting algorithmic consequences. \Cref{sec:laguerre} identifies Laguerre dynamics as the scaling limit of birth-death chains and extends the discrete results on \(\bb N^d\) to \(\bb R_{> 0}^d\).

\section{Preliminaries}
\label{sec:preliminaries}

\subsection{Notation}
\label{subsec:notations}

We write \(\bb N \defeq \set{0, 1, \dots}\), \(\bb N_{> 0} \defeq \set{1, 2, \dots}\), and \(\stp{d} \defeq \set{1, \dots, d}\). Vectors are written in bold. We denote by \(\*e_i\) the \(i\)-th standard basis vector and by \(\*0\) and \(\*1\) the all-zero and all-one vectors, with dimensions determined by context. Vector inequalities are interpreted coordinatewise. For \(\*n \in \bb N^d\) and \(\boldsymbol{\theta} \in \bb R^d\), we write
\[
\*n! \defeq \prod_{i = 1}^d n_i!, \qquad \boldsymbol{\theta}^{\*n} \defeq \prod_{i = 1}^d \theta_i^{n_i}.
\]

For finite index sets \(I\) and \(J\), the spaces \(\bb{R}^I\) and \(\bb{N}^I\) consist of vectors indexed by \(I\), while \(\bb{R}^{I \times J}\) consists of matrices with rows indexed by \(I\) and columns indexed by \(J\). We write \(\*1_I \in \bb{R}^I\) for the corresponding all-ones vector, \(\*I_I\) for the identity matrix indexed by \(I\), and \(\*1_{I \times J}\) for the all-ones matrix indexed by \(I \times J\); integer subscripts such as \(\*1_{d \times m}\) specify matrix dimensions. The sets \(\bb{R}_{> 0}^d\) and \(\bb{R}_{\ge 0}^d\) denote the positive and nonnegative orthants, respectively.

For \(S \subseteq \stp{d}\) and a vector \(\*x\) indexed by \(\stp{d}\), we write \(\*x_S\) for its restriction to \(S\). Vectors on disjoint subsets of \(\stp{d}\) are identified with their concatenation in the original coordinate order.

For a function \(f\) on a subset of \(\bb Z^d\), its forward and backward differences are
\[
\nabla_i^+ f\tp{\*n} \defeq f\tp{\*n + \*e_i} - f\tp{\*n}, \qquad \nabla_i^- f\tp{\*n} \defeq f\tp{\*n - \*e_i} - f\tp{\*n},
\]
whenever the corresponding neighboring states lie in the domain. If a neighbor lies outside the state space, the corresponding difference is used only in a term whose transition rate is zero, and that term is interpreted as zero. A set \(S \subseteq \bb N^d\) is downward closed if \(\*n \in S\) and \(\*0 \le \*m \le \*n\) imply \(\*m \in S\). A set \(S \subseteq \bb Z^d\) is upward closed if \(\*n \in S\) and \(\*m \ge \*n\) imply \(\*m \in S\). For a function \(f: S \to \bb R\) and a probability measure \(\mu\) on a countable set \(S\), we write
\[
\supp\tp{f} \defeq \set{x \in S \cmid f\tp{x} \ne 0}, \qquad \supp\tp{\mu} \defeq \set{x \in S \cmid \mu\tp{x} > 0},
\]
and
\[
\+C_{\-c}\tp{S} \defeq \set{f: S \to \bb R \cmid \supp\tp{f} \text{ is finite}}.
\]
We use \(\*1_A\) for the indicator of an event or set \(A\), \(\delta_x\) for the Dirac measure at \(x\), and \(\+P\tp{S}\) for the set of probability measures on \(S\).

For an open set \(U \subseteq \bb{R}^d\) and \(k \in \bb{N}\), \(\+C^k\tp{U}\) denotes the space of \(k\)-times continuously differentiable functions on \(U\), while \(\+C_{\-c}^{\infty}\tp{U}\) denotes the space of smooth compactly supported functions. For a smooth function \(f\), we use \(\partial_i f\), \(\nabla f\), \(\nabla^2 f\), and \(\Delta f \defeq \sum_{i = 1}^d \partial_{ii} f\) for its partial derivatives, gradient, Hessian, and Laplacian, respectively.

For a probability measure \(\mu\), expectation, probability, variance, and covariance are denoted by \(\*E_{\mu}\), \(\*{Pr}_{\mu}\), \(\*{Var}_{\mu}\), and \(\*{Cov}_{\mu}\), respectively. When random variables and their laws are clear, the subscripts are omitted. For a random vector \(\*X \sim \mu\), we write
\[
\mean{\mu} \defeq \*E\stp{\*X}, \qquad \cov{\mu} \defeq \*{Cov}\stp{\*X},
\]
and use \(\var{\mu}\) in one dimension. For a convex function \(\Phi\), the \(\Phi\)-entropy of \(f\) is
\[
\*{Ent}_{\mu}^{\Phi}\stp{f} \defeq \*E_{\mu}\stp{\Phi\tp{f}} - \Phi\tp{\*E_{\mu}\stp{f}}.
\]
The choices \(\Phi\tp{x} = x^2\) and \(\Phi\tp{x} = x \log x\) give \(\*{Var}_{\mu}\stp{f}\) and, for nonnegative \(f\), \(\*{Ent}_{\mu}\stp{f}\), respectively, with the convention \(0 \log 0 \defeq 0\).

For probability measures \(\mu\) and \(\nu\) on a countable set \(S\), we write \(\mu \ll \nu\) when \(\mu\) is absolutely continuous with respect to \(\nu\), and define
\[
\-{TV}\tp{\mu, \nu} \defeq \frac{1}{2} \sum_{x \in S} \abs{\mu\tp{x} - \nu\tp{x}}, \qquad \-{KL}\tp{\mu \| \nu} \defeq \sum_{x \in \supp\tp{\mu}} \mu\tp{x} \log \frac{\mu\tp{x}}{\nu\tp{x}},
\]
and
\[
\chi^2\tp{\mu \| \nu} \defeq \sum_{x \in \supp\tp{\nu}} \frac{\tp{\mu\tp{x} - \nu\tp{x}}^2}{\nu\tp{x}}.
\]
The relative entropy and \(\chi^2\)-divergence are set to \(+\infty\) when \(\mu \not\ll \nu\). We use the analogous Radon--Nikodym definitions on general measurable spaces. The notation \(\mu_n \Rightarrow \mu\) denotes weak convergence.

For measures \(\mu \ll \nu\), we write \(\dd \mu / \dd \nu\) for the Radon--Nikodym derivative. In integrals, \(\nu\tp{\-d x}\) denotes integration with respect to \(\nu\). We write \(p\tp{x} \propto w\tp{x}\) when \(p\) is the normalization of a nonnegative weight \(w\), assuming that the normalizing constant is finite and positive.

For \(\mu \in \+P\tp{\bb N^d}\) and \(\boldsymbol{\theta} \in \bb R_{> 0}^d\) such that the normalizing constant below is finite, the exponential tilt of \(\mu\) by \(\boldsymbol{\theta}\) is
\[
\tp{\boldsymbol{\theta} * \mu}\tp{\*n} \defeq \frac{\boldsymbol{\theta}^{\*n} \mu\tp{\*n}}{\*E_{\*k \sim \mu}\stp{\boldsymbol{\theta}^{\*k}}}, \quad \*n \in \bb N^d.
\]
For \(\theta > 0\), we abbreviate \(\theta * \mu \defeq \tp{\theta \*1} * \mu\), and exponentiation of vectors is coordinatewise: \(e^{\boldsymbol{\theta}} \defeq \tp{e^{\theta_1}, \dots, e^{\theta_d}}\). Thus \(e^{\boldsymbol{\theta}} * \mu\) is the exponential tilt with natural parameter \(\boldsymbol{\theta}\).

For \(\mu, \nu \in \+P\tp{\bb N^d}\), their convolution is denoted by \(\mu * \nu\) and is defined by
\[
\tp{\mu * \nu}\tp{\*n} \defeq \sum_{\*0 \le \*k \le \*n} \mu\tp{\*k} \nu\tp{\*n - \*k}, \quad \*n \in \bb N^d.
\]
For sets \(A, B \subseteq \bb N^d\), their Minkowski sum is \(A + B \defeq \set{\*a + \*b \cmid \*a \in A, \, \*b \in B}\).

We use \(\inner{\cdot}{\cdot}\) and \(\norm{\cdot}_p\) for the Euclidean inner product and \(\ell_p\)-norms, and \(\inner{\*A}{\*B}_{\-{HS}} \defeq \operatorname{Tr}\tp{\*A^{\top} \*B}\) for the Hilbert--Schmidt inner product, where \(\*A^{\top}\) denotes the transpose of \(\*A\). The symbols \(\*I\), \(\*O\), and \(\diag\tp{\*v}\) denote the identity matrix, zero matrix, and diagonal matrix with diagonal \(\*v\), respectively. For symmetric matrices, \(\*A \preceq \*B\) and \(\*A \prec \*B\) mean that \(\*B - \*A\) is positive semidefinite and positive definite, respectively; the reverse orders are denoted by \(\succeq\) and \(\succ\). We write \(\lambda_{\min}\tp{\*A}\) and \(\lambda_{\max}\tp{\*A}\) for the smallest and largest eigenvalues of a symmetric matrix \(\*A\), and \(\*A^{\dagger}\) for its Moore--Penrose pseudoinverse. The symbols \(\otimes\) and \(\bigotimes\) denote products of measures or Kronecker products of matrices, according to context. We write \(L^p\tp{\mu}\) for the usual function spaces and \(\inner{f}{g}_{L^2\tp{\mu}} \defeq \*E_{\mu}\stp{fg}\).

We use \(\-{Law}\tp{X}\) for the law of a random variable \(X\). The distributions \(\-{Bin}\tp{n, p}\) and \(\-{Pois}\tp{\lambda}\) use their standard parameterizations. For \(r > 0\) and \(p \in \tp{0, 1}\), the negative binomial distribution \(\-{NB}\tp{r, p}\) has mass function
\[
\-{NB}\tp{r, p; k} \defeq \frac{\Gamma\tp{r + k}}{\Gamma\tp{r} k!} p^r \tp{1 - p}^k, \quad k \in \bb N.
\]
Here, \(\Gamma\) denotes Euler's gamma function. The distribution \(\-{Gamma}\tp{\alpha, \beta}\) uses the shape-rate parameterization. When a final argument is included, as in \(\-{Bin}\tp{n, p; k}\), it denotes the corresponding probability mass at \(k\).

For \(S \subseteq \stp{d}\) and \(\*X \sim \mu \in \+P\tp{\bb N^d}\), we write \(\mu_S \defeq \-{Law}\tp{\*X_S}\). For \(\*x \in \bb N^S\) satisfying \(\*{Pr}\stp{\*X_S = \*x} > 0\), we write \(\mu^{S \gets \*x} \defeq \-{Law}\tp{\*X \mid \*X_S = \*x}\) for the conditional law.

We use standard asymptotic notation. In \(O_{\xi}\tp{r}\), the implicit constant may depend on \(\xi\), while \(o_n\tp{1}\) denotes a quantity tending to zero as \(n \to \infty\); analogous conventions are used for other parameters and limiting variables.

For a metric space \(\Omega\) and \(T > 0\), let \(\+D\tp{\stp{0, T}, \Omega}\) denote the Skorokhod space of c\`adl\`ag paths from \(\stp{0, T}\) to \(\Omega\). For a path \(\*p\), \(\*p_{\stp{0, T}}\) denotes its full trajectory and \(\*p_t\) its value at time \(t\). If \(R\) is a Markov law on this path space, then \(R_t\) denotes its time-\(t\) marginal, while \(R_{s \to t}\) and \(R_{t \to s}\) denote its forward and backward transition kernels for \(0 \le s \le t \le T\).

\subsection{Discrete-Time Markov Chains and Continuous-Time Markov Processes}
\label{subsec:markov-chains-basics}

\paragraph{Discrete time.}

Let \(\Omega\) be a finite state space. A discrete-time Markov chain is characterized by a transition kernel \(P \in \bb R^{\Omega \times \Omega}\) satisfying \(P\tp{x, y} \ge 0\) and \(\sum_{y \in \Omega} P\tp{x, y} = 1\) for every \(x \in \Omega\). We identify the chain with \(P\) and view the kernel as the operator
\[
P f\tp{x} \defeq \sum_{y \in \Omega} P\tp{x, y} f\tp{y}.
\]
A probability measure \(\mu\) on \(\Omega\) is stationary for \(P\) if \(\mu P = \mu\). The kernel \(P\) is reversible with respect to \(\mu\) if
\[
\mu\tp{x} P\tp{x, y} = \mu\tp{y} P\tp{y, x}, \quad \forall x, y \in \Omega.
\]
Equivalently, \(P\) is self-adjoint on \(L^2\tp{\mu}\).

\begin{definition}
    The Markov chain \(P\) is irreducible if, for all \(x, y \in \Omega\), there exists \(t \in \bb N\) such that \(P^t\tp{x, y} > 0\). It is aperiodic if, for every \(x \in \Omega\), \(\gcd\set{t \in \bb N_{> 0} \cmid P^t\tp{x, x} > 0} = 1\). It is ergodic if it is both irreducible and aperiodic.
\end{definition}

\begin{theorem}[Fundamental Theorem of Markov Chains; see, e.g., \cite{LPW17}]
    Let \(P\) be an ergodic Markov chain on a finite state space \(\Omega\). Then \(P\) has a unique stationary distribution \(\mu\), and for every initial distribution \(\nu\) on \(\Omega\), \(\nu P^t \to \mu\) pointwise as \(t \to \infty\).
\end{theorem}

We measure the resulting convergence in total variation.

\begin{definition}
    Let \(P\) be an ergodic Markov chain on a finite state space \(\Omega\) with stationary distribution \(\mu\). For \(\eps \in \tp{0, 1}\), the \(\eps\)-mixing time of \(P\) is
    \[
    T_{\-{mix}}\tp{\eps; P} \defeq \max_{x \in \Omega} \min\set{t \in \bb N \cmid \-{TV}\tp{P^t\tp{x, \cdot}, \mu} \le \eps}.
    \]
\end{definition}

\paragraph{Continuous time.}

Let \(\tp{X_t}_{t \ge 0}\) be a time-homogeneous Markov process on a measurable state space \(\Omega\). Its transition semigroup is
\[
P_t f\tp{x} \defeq \*E\stp{f\tp{X_t} \mid X_0 = x}, \qquad P_{s + t} = P_s P_t.
\]
On a function space where this semigroup is strongly continuous, its generator is
\[
\+L f \defeq \lim_{t \to 0^+} \frac{P_t f - f}{t}
\]
on the set of functions for which the limit exists. A probability measure \(\mu\) is stationary if \(\*E_{\mu}\stp{P_t f} = \*E_{\mu}\stp{f}\) for every \(t \ge 0\), and the process is reversible with respect to \(\mu\) if every \(P_t\) is self-adjoint on \(L^2\tp{\mu}\). In this case, the generator \(\+L\) is also self-adjoint on its domain.

On a countable state space, a pure-jump process is specified by jump rates \(q\tp{x, y} \ge 0\) for \(x \ne y\) and has generator
\[
\+L f\tp{x} = \sum_{y \ne x} q\tp{x, y} \tp{f\tp{y} - f\tp{x}}.
\]
We assume that the total rate \(\sum_{y \ne x} q\tp{x, y}\) is finite at every state and that the process is non-explosive. Reversibility is then equivalent to the detailed balance condition:
\[
\mu\tp{x} q\tp{x, y} = \mu\tp{y} q\tp{y, x}, \quad \forall x, y \in \Omega.
\]
The process is irreducible if, for every \(x, y \in \Omega\), there exists \(t > 0\) such that \(\*{Pr}\stp{X_t = y \mid X_0 = x} > 0\). On a finite state space, irreducibility alone ensures convergence to the unique stationary distribution; no aperiodicity assumption is needed.

A time-inhomogeneous Markov process is described by operators
\[
P_{s, t} f\tp{x} \defeq \*E\stp{f\tp{X_t} \mid X_s = x}, \qquad 0 \le s \le t,
\]
satisfying \(P_{s, u} P_{u, t} = P_{s, t}\) for \(s \le u \le t\). Its time-dependent generator is
\[
\+L_t f \defeq \lim_{h \to 0^+} \frac{P_{t, t + h} f - f}{h}.
\]
This limit is understood on the domain where it exists.
For a time-inhomogeneous pure-jump process with rates \(q_t\tp{x, y}\),
\[
\+L_t f\tp{x} = \sum_{y \ne x} q_t\tp{x, y} \tp{f\tp{y} - f\tp{x}}.
\]

Under the usual well-posedness assumptions, the solution of the SDE
\[
\dd \*X_t = \*b\tp{t, \*X_t} \dd t + \boldsymbol{\sigma}\tp{t, \*X_t} \dd \*B_t
\]
is a time-inhomogeneous Markov process on \(\bb R^d\). Its generator acts on smooth functions as
\[
\+L_t f\tp{\*x} = \inner{\*b\tp{t, \*x}}{\nabla f\tp{\*x}} + \frac{1}{2} \inner{\boldsymbol{\sigma}\tp{t, \*x} \boldsymbol{\sigma}\tp{t, \*x}^{\top}}{\nabla^2 f\tp{\*x}}_{\-{HS}}.
\]

\subsection{Functional Inequalities and Mixing Times}
\label{subsec:functional-inequalities-mixing}

Let \(\+L\) be the generator of a reversible continuous-time Markov process with stationary measure \(\mu\). For a reversible discrete-time kernel \(P\), we set \(\+L_P \defeq P - I\). In either case, the associated Dirichlet form is
\[
\+E\tp{f, g} \defeq -\inner{f}{\+L g}_{L^2\tp{\mu}}.
\]
We write \(\+E_P\) for the form associated with \(\+L_P\). The form is initially defined on \(\+C_{\-c}\tp{\Omega}\), and we write
\[
\+D\tp{\+E} \defeq \set{f \in L^2\tp{\mu} \cmid \+E\tp{f, f} < \infty}
\]
for its energy domain after closure. Reversibility makes the Dirichlet form symmetric and nonnegative.

\begin{definition}
    Let \(\+L\) be a reversible Markov generator with stationary measure \(\mu\) and Dirichlet form \(\+E\).
    \begin{itemize}
        \item We say that \(\+L\) satisfies a \Poincare inequality with constant \(\gamma > 0\) if
        \[
        \gamma \, \*{Var}_{\mu}\stp{f} \le \+E\tp{f, f}, \quad \forall f \in \+D\tp{\+E}.
        \]
        The optimal constant \(\gamma\) is called the spectral gap of \(\+L\); for an irreducible finite-state chain, it is the smallest nonzero eigenvalue of \(-\+L\).
        \item We say that \(\+L\) satisfies a modified log-Sobolev inequality with constant \(\rho > 0\) if
        \[
        \rho \, \*{Ent}_{\mu}\stp{f} \le \+E\tp{\log f, f}, \quad \forall f \in \+A,
        \]
        where \(\+A \defeq \set{f \in \+D\tp{\+E} \cmid f > 0, \, \log f \in \+D\tp{\+E}}\).
    \end{itemize}
\end{definition}

For a discrete-time kernel \(P\), these definitions are applied to \(\+L_P\) and \(\+E_P\). The modified log-Sobolev inequality extends to nonnegative functions by approximation whenever both sides are well defined.

For a continuous-time semigroup \(\tp{P_t}_{t \ge 0}\), the \Poincare inequality implies variance contraction, while the modified log-Sobolev inequality implies entropy contraction. For every admissible \(f\), with \(f\) nonnegative in the entropy bound,
\[
\*{Var}_{\mu}\stp{P_t f} \le e^{-2 \gamma t} \*{Var}_{\mu}\stp{f}, \qquad \*{Ent}_{\mu}\stp{P_t f} \le e^{-\rho t} \*{Ent}_{\mu}\stp{f}.
\]
In finite discrete time, these constants yield the following mixing bounds.

\begin{theorem}[Functional Inequalities and Mixing Time; see, e.g., \cite{LPW17,BT03}]
    Let \(P\) be an ergodic reversible Markov chain on a finite state space \(\Omega\) with at least two states and stationary distribution \(\mu\), and suppose that \(P\) is positive semidefinite on \(L^2\tp{\mu}\). Set \(\mu_{\min} \defeq \min_{x \in \Omega} \mu\tp{x}\). Then, for every \(\eps \in \tp{0, 1 / 2}\), the following statements hold:
    \begin{itemize}
        \item If \(P\) satisfies a \Poincare inequality with constant \(\gamma\), then
        \[
        T_{\-{mix}}\tp{\eps; P} \le \ceil{\frac{1}{\gamma} \tp{\log\frac{1}{2 \eps} + \frac{1}{2} \log \frac{1}{\mu_{\min}}}}.
        \]
        \item If \(P\) satisfies a modified log-Sobolev inequality with constant \(\rho\), then
        \[
        T_{\-{mix}}\tp{\eps; P} \le \ceil{\frac{1}{\rho} \tp{\log\frac{1}{2 \eps^2} + \log\log \frac{1}{\mu_{\min}}}}.
        \]
    \end{itemize}
    \label{thm:functional-inequality-mixing}
\end{theorem}

\begin{remark}
    For a reversible kernel, lazification---replacing \(P\) with \(\bar P \defeq \tp{I + P} / 2\)---ensures positive semidefiniteness.
\end{remark}

\section{Bakry--\Emery Theory for Birth-Death Chains}
\label{sec:bakry-emery-birth-death}

In this section, we develop an integrated Bakry--\Emery calculus for a class of reversible birth-death chains on \(\bb N^d\). The main goal is to derive functional inequalities from curvature-type identities that remain valid after integration against the stationary measure.

The section is organized as follows. In \Cref{subsec:bakry-emery-preliminaries}, we recall the standard Bakry--\Emery calculus. In \Cref{subsec:birth-death-generators}, we introduce the birth-death generator, its Dirichlet form, and the associated \Poincare inequality. In \Cref{subsec:integrated-bakry-emery}, we compute the integrated \(\Gamma_2\) form. The resulting Bochner formula motivates the high-dimensional ultra-log-concavity condition introduced in \Cref{subsec:ultra-log-concavity}. We then exploit the full local curvature matrix to derive a discrete Brascamp--Lieb inequality in \Cref{subsec:discrete-brascamp-lieb}. Finally, in \Cref{subsec:variance-hardcore}, we illustrate the Brascamp--Lieb inequality through a direct application to the hardcore model.

The analysis in \Cref{subsec:birth-death-generators,subsec:integrated-bakry-emery,subsec:ultra-log-concavity} is inspired by \cite{Joh17}, which studies another canonical class of birth-death chains with unit birth rates, whereas we focus on chains with linear death rates.

\subsection{Bakry--\Emery Calculus}
\label{subsec:bakry-emery-preliminaries}

The Bakry--\Emery calculus associates first- and second-order bilinear operators with a Markov generator \cite{BE85,BGL14}. Let \(\+L\) be a Markov generator, and consider functions for which the expressions below are well defined.

\begin{definition}[Carr\'e du Champ Operators]
    The carr\'e du champ and iterated carr\'e du champ operators associated with \(\+L\) are
    \[
    \Gamma\tp{f, g} \defeq \frac{1}{2} \stp{\+L\tp{fg} - f \+L g - g \+L f}, \qquad \Gamma_2\tp{f, g} \defeq \frac{1}{2} \stp{\+L\Gamma\tp{f, g} - \Gamma\tp{f, \+L g} - \Gamma\tp{g, \+L f}}.
    \]
    \label{def:carre-du-champ-operators}
\end{definition}

Suppose that \(\+L\) is reversible with respect to a stationary measure \(\mu\). Stationarity and integration by parts give
\begin{equation}
    \*E_{\mu}\stp{\Gamma\tp{f, g}} = -\*E_{\mu}\stp{f \+L g} = \+E\tp{f, g}, \qquad \*E_{\mu}\stp{\Gamma_2\tp{f, g}} = \*E_{\mu}\stp{\+L f \, \+L g} = \*E_{\mu}\stp{f \+L^2 g}.
    \label{eq:integrated-carre-du-champ-identities}
\end{equation}
For \(\delta > 0\), the pointwise Bakry--\Emery curvature condition \(\-{CD}\tp{\delta, \infty}\) is
\[
\Gamma_2\tp{f, f} \ge \delta \, \Gamma\tp{f, f}
\]
for every admissible \(f\). Integrating this condition yields the integrated Bakry--\Emery inequality. On a finite state space, the integrated inequality is equivalent to the \Poincare inequality.

\begin{theorem}[Integrated Bakry--\Emery Criterion]
    \label{thm:integrated-bakry-emery-criterion}
    Let \(\+L\) be the generator of an irreducible continuous-time Markov chain on a finite state space \(\Omega\), reversible with respect to \(\mu\). For \(\delta > 0\), the following statements are equivalent:
    \begin{enumerate}
        \item The generator \(\+L\) satisfies the \Poincare inequality with constant \(\delta\).
        \item For every \(f: \Omega \to \bb R\),
        \[
        \*E_{\mu}\stp{\Gamma_2\tp{f, f}} \ge \delta \, \*E_{\mu}\stp{\Gamma\tp{f, f}}.
        \]
    \end{enumerate}
    Consequently, the pointwise condition \(\-{CD}\tp{\delta, \infty}\) implies the \Poincare inequality with constant \(\delta\).
    \label{thm:integrated-bakry-emery}
\end{theorem}

\begin{proof}[Proof of \Cref{thm:integrated-bakry-emery}]
    Let \(0 = \lambda_0 < \lambda_1 \le \cdots \le \lambda_{\abs{\Omega} - 1}\) be the eigenvalues of \(-\+L\), counted with multiplicity. Expanding a mean-zero function \(f\) in an orthonormal eigenbasis and using \eqref{eq:integrated-carre-du-champ-identities} shows that the \Poincare inequality is equivalent to \(\lambda_k \ge \delta\) for every \(k \ge 1\). The integrated Bakry--\Emery inequality is equivalent to \(\lambda_k^2 \ge \delta \lambda_k\) for every \(k \ge 1\), which is the same condition.
\end{proof}

For a diffusion generator, the defining feature of the calculus is the chain rule
\[
\+L \phi\tp{f} = \phi'\tp{f} \+L f + \phi''\tp{f} \, \Gamma\tp{f, f}.
\]
Under this condition, \(\-{CD}\tp{\delta,\infty}\) implies not only the \Poincare inequality, but also the stronger log-Sobolev inequality
\[
\delta \*{Ent}_{\mu}\stp{f^2} \le 2 \+E\tp{f, f}, \quad \forall f.
\]
For example, the Langevin generator \(\+L = \Delta - \nabla V \cdot \nabla\) is a diffusion generator, with
\[
\Gamma\tp{f, g} = \inner{\nabla f}{\nabla g}, \qquad \Gamma_2\tp{f, g} = \inner{\nabla^2 f}{\nabla^2 g}_{\-{HS}} + \nabla f^{\top} \nabla^2 V \nabla g.
\]
Consequently, \(\delta\)-strong log-concavity, i.e., \(\nabla^2 V \succeq \delta \*I\), is equivalent to \(\-{CD}\tp{\delta, \infty}\), and hence implies both the \Poincare and log-Sobolev inequalities with constant \(\delta\). 

Pure-jump generators generally do not satisfy the diffusion chain rule. Nevertheless, the operators introduced in \Cref{def:carre-du-champ-operators}, the identities in \eqref{eq:integrated-carre-du-champ-identities}, and the criterion in \Cref{thm:integrated-bakry-emery-criterion} remain valid under the hypotheses above. This makes the integrated criterion particularly well suited to the birth-death chains studied below.

\subsection{Birth-Death Generators, Dirichlet Forms, and \Poincare Inequalities}
\label{subsec:birth-death-generators}

Birth-death chains are continuous-time Markov chains on \(\bb N^d\) whose transitions move only to neighboring states. Such a chain is specified by birth rates \(\*b: \bb N^d \to \bb R_{\ge 0}^d\) and death rates \(\*d: \bb N^d \to \bb R_{\ge 0}^d\). Its generator is
\[
\+L f\tp{\*n} \defeq \sum_{i = 1}^d b_i\tp{\*n} \nabla_i^+ f\tp{\*n} + \sum_{i = 1}^d d_i\tp{\*n} \nabla_i^- f\tp{\*n}.
\]

We focus on the birth-death chain with linear death rates \(d_i\tp{\*n} = n_i\). Let \(\mu\) be a probability measure on \(\bb N^d\) with downward closed support. We choose the birth rates so that the chain is reversible with respect to \(\mu\). Equivalently, the detailed balance condition gives
\[
b_i\tp{\*n} = \frac{\tp{n_i + 1} \mu\tp{\*n + \*e_i}}{\mu\tp{\*n}}.
\]
We assume the convention that \(\mu\tp{\*n + \*e_i} = 0\) whenever \(\*n + \*e_i \notin \supp\tp{\mu}\). The downward closed assumption ensures that the death transitions remain inside the support.

Thus, the generator is
\begin{equation}
    \+L f\tp{\*n} \defeq \sum_{i = 1}^d \frac{\tp{n_i + 1} \mu\tp{\*n + \*e_i}}{\mu\tp{\*n}} \nabla_i^+ f\tp{\*n} + \sum_{i = 1}^d n_i \nabla_i^- f\tp{\*n}.
    \label{eq:birth-death-generator}
\end{equation}
For \(f, g \in \+C_{\-c}\tp{\supp\tp{\mu}}\), the carr\'e du champ operator is
\begin{equation}
    \begin{aligned}
        \Gamma\tp{f, g} &\defeq \frac{1}{2} \stp{\+L\tp{fg} - f \+L g - g \+L f} \\
        &= \frac{1}{2} \sum_{i = 1}^d \frac{\tp{n_i + 1} \mu\tp{\*n + \*e_i}}{\mu\tp{\*n}} \nabla_i^+ f\tp{\*n} \nabla_i^+ g\tp{\*n} + \frac{1}{2} \sum_{i = 1}^d n_i \nabla_i^- f\tp{\*n} \nabla_i^- g\tp{\*n}.
    \end{aligned}
    \label{eq:birth-death-carre-du-champ}
\end{equation}
Since the generator \(\+L\) is reversible with respect to \(\mu\), the associated Dirichlet form is
\begin{equation}
    \+E\tp{f, g} \defeq \*E_{\mu}\stp{\Gamma\tp{f, g}} = \*E_{\*n \sim \mu}\stp{\sum_{i = 1}^d \frac{\tp{n_i + 1} \mu\tp{\*n + \*e_i}}{\mu\tp{\*n}} \nabla_i^+ f\tp{\*n} \nabla_i^+ g\tp{\*n}}, \quad f, g \in \+D\tp{\+E},
    \label{eq:birth-death-dirichlet-form}
\end{equation}
where \(\+D\tp{\+E}\) is the energy domain introduced in \Cref{subsec:functional-inequalities-mixing}. The corresponding \Poincare inequality takes the form
\begin{equation}
    \gamma \, \*{Var}_{\mu}\stp{f} \le \*E_{\*n \sim \mu}\stp{\sum_{i = 1}^d \frac{\tp{n_i + 1} \mu\tp{\*n + \*e_i}}{\mu\tp{\*n}} \tp{\nabla_i^+ f\tp{\*n}}^2}, \quad \forall f \in \+D\tp{\+E}.
    \label{eq:birth-death-poincare-inequality}
\end{equation}
We refer to \eqref{eq:birth-death-poincare-inequality} as the weighted \Poincare inequality for \(\mu\) (equivalently, for the birth-death generator \(\+L\)) with constant \(\gamma\). The largest admissible constant \(\gamma\) is the spectral gap of \(\+L\).

\subsection{Integrated Bakry--\Emery Criterion and Integrated \texorpdfstring{\(\Gamma_2\)}{Gamma 2} Form}
\label{subsec:integrated-bakry-emery}

We now compute the integrated \(\Gamma_2\) form for the birth-death generator \eqref{eq:birth-death-generator}. By \Cref{thm:integrated-bakry-emery-criterion}, a bound of the form
\[
\*E_{\mu}\stp{\Gamma_2\tp{f, f}} \ge \delta \, \*E_{\mu}\stp{\Gamma\tp{f, f}}, \quad \forall f \in \+C_{\-c}\tp{\supp\tp{\mu}},
\]
yields a \Poincare inequality with constant \(\delta\) when \(\mu\) is finitely supported. For infinite support, we apply the criterion to finite restrictions and then pass to the limit; see the proof of \Cref{thm:ulc-poincare-inequality}. The integrated formulation is weaker than the pointwise curvature condition and allows cancellations under the stationary measure, which are essential in the discrete setting.

For the birth-death generator \(\+L\), the integrated carr\'e du champ is exactly the Dirichlet form in \eqref{eq:birth-death-dirichlet-form}:
\[
\*E_{\mu}\stp{\Gamma\tp{f, f}} = \+E\tp{f, f} = \*E_{\*n \sim \mu}\stp{\sum_{i = 1}^d \frac{\tp{n_i + 1} \mu\tp{\*n + \*e_i}}{\mu\tp{\*n}} \tp{\nabla_i^+ f\tp{\*n}}^2}.
\]
It remains to compute the integrated \(\Gamma_2\) form. The following theorem expresses \(\*E_{\mu}\stp{\Gamma_2\tp{f, g}}\) explicitly in terms of ratios between neighboring masses of the stationary measure \(\mu\). The remainder of this subsection establishes this identity, while its consequences for functional inequalities are developed in \Cref{subsec:ultra-log-concavity} and \Cref{subsec:discrete-brascamp-lieb}.

\begin{theorem}[Discrete Bochner Formula]
    Consider the birth-death generator \(\+L\) defined in \eqref{eq:birth-death-generator}, and set \(a\tp{\*n} \defeq \*n! \mu\tp{\*n}\) for \(\*n \in \bb N^d\). Then the identity below holds for every \(f, g \in \+C_{\-c}\tp{\supp\tp{\mu}}\):
    \[
    \begin{aligned}
        \*E_{\mu}\stp{\Gamma_2\tp{f, g}} =& \*E_{\*n \sim \mu}\stp{\sum_{i, j \in \stp{d}} \frac{a\tp{\*n + \*e_i + \*e_j}}{a\tp{\*n}} \nabla_i^+ \nabla_j^+ f\tp{\*n} \nabla_i^+ \nabla_j^+ g\tp{\*n}} \\
        &+ \*E_{\*n \sim \mu}\stp{\sum_{i, j \in \stp{d}} \tp{\frac{a\tp{\*n + \*e_i} a\tp{\*n + \*e_j}}{a\tp{\*n}^2} - \frac{a\tp{\*n + \*e_i + \*e_j}}{a\tp{\*n}}} \nabla_i^+ f\tp{\*n} \nabla_j^+ g\tp{\*n}} \\
        &+ \*E_{\*n \sim \mu}\stp{\sum_{i \in \stp{d}} \frac{a\tp{\*n + \*e_i}}{a\tp{\*n}} \nabla_i^+ f\tp{\*n} \nabla_i^+ g\tp{\*n}}.
    \end{aligned}
    \]
    \label{thm:integrated-gamma-2-form}
\end{theorem}

\begin{remark}
    The identity in \Cref{thm:integrated-gamma-2-form} is a discrete analogue of the Bochner formula for the Langevin generator \(\+L = \Delta - \nabla V \cdot \nabla\) on \(\bb R^d\):
    \[
    \Gamma_2\tp{f, g} = \inner{\nabla^2 f}{\nabla^2 g}_{\-{HS}} + \nabla f^{\top} \nabla^2 V \nabla g.
    \]
    The first term is a second-order energy, which is non-negative when \(f = g\). The second term is a first-order form determined by the local curvature of the potential \(V = -\log \mu\). Likewise, the first line of the identity in \Cref{thm:integrated-gamma-2-form} is a weighted discrete Hessian energy. The remaining two lines define a first-order form whose coefficients compare neighboring ratios of \(a\tp{\*n} = \*n! \mu\tp{\*n}\), thereby serving as a discrete analogue of the curvature matrix \(\nabla^2 V\).
\end{remark}

\begin{proof}[Proof of \Cref{thm:integrated-gamma-2-form}]
    Since \(f, g \in \+C_{\-c}\tp{\supp\tp{\mu}}\) and the birth-death graph is locally finite, \(\+L f\), \(\+L g\), and every summand appearing below are finitely supported. Hence all expectations are finite, and every change of variables is justified directly.

    For \(\*n \in \supp\tp{\mu}\), define the ratios
    \[
    r_i\tp{\*n} \defeq \frac{a\tp{\*n + \*e_i}}{a\tp{\*n}}, \qquad r_{ij}\tp{\*n} \defeq \frac{a\tp{\*n + \*e_i + \*e_j}}{a\tp{\*n}}.
    \]
    These quantities are extended by zero outside the support. Reversibility of \(\+L\) and the definition of \(\Gamma_2\) imply
    \[
    \*E_{\mu}\stp{\Gamma_2\tp{f, g}} = -\frac{1}{2} \*E_{\mu}\stp{\Gamma\tp{f, \+L g} + \Gamma\tp{g, \+L f}} = \*E_{\mu}\stp{\+L f \, \+L g}.
    \]
    Then the required discrete Bochner formula is
    \[
    \begin{aligned}
        \*E_{\mu}\stp{\+L f \, \+L g}
        =& \*E_{\*n \sim \mu}\stp{\sum_{i, j \in \stp{d}} r_{ij}\tp{\*n} \nabla_i^+ \nabla_j^+ f\tp{\*n} \nabla_i^+ \nabla_j^+ g\tp{\*n}} + \*E_{\*n \sim \mu}\stp{\sum_{i \in \stp{d}} r_i\tp{\*n} \nabla_i^+ f\tp{\*n} \nabla_i^+ g\tp{\*n}} \\
        &+ \*E_{\*n \sim \mu}\stp{\sum_{i, j \in \stp{d}} \tp{r_i\tp{\*n} r_j\tp{\*n} - r_{ij}\tp{\*n}} \nabla_i^+ f\tp{\*n} \nabla_j^+ g\tp{\*n}}.
    \end{aligned}
    \]
    
    To verify this identity, write
    \[
    u_i\tp{\*n} \defeq \nabla_i^+ f\tp{\*n}, \qquad v_j\tp{\*n} \defeq \nabla_j^+ g\tp{\*n}.
    \]
    Note that \(r_i\tp{\*n}\) is exactly the birth rate \(b_i\tp{\*n}\) in \eqref{eq:birth-death-generator}. The generator can then be written entirely in terms of forward differences as
    \[
    \+L f\tp{\*n} = \sum_{i \in \stp{d}} r_i\tp{\*n} u_i\tp{\*n} - \sum_{\substack{i \in \stp{d} \\ n_i > 0}} n_i u_i\tp{\*n - \*e_i},
    \]
    and analogously for \(g\). In the sums below, a predecessor term is understood to vanish when the corresponding coordinate is zero. We use the following identities:
    \[
    \mu\tp{\*n} n_i = \mu\tp{\*n - \*e_i} r_i\tp{\*n - \*e_i}, \qquad r_i\tp{\*n} r_j\tp{\*n + \*e_i} = r_{ij}\tp{\*n} = r_j\tp{\*n} r_i\tp{\*n + \*e_j}.
    \]

    Expand the product \(\+L f \, \+L g\) into four terms:
    \[
    \begin{aligned}
        \+L f\tp{\*n} \+L g\tp{\*n} =& \underbrace{\sum_{i, j \in \stp{d}} r_i\tp{\*n} r_j\tp{\*n} u_i\tp{\*n} v_j\tp{\*n}}_{\text{birth--birth}} + \underbrace{\sum_{i, j \in \stp{d}} n_i n_j u_i\tp{\*n - \*e_i} v_j\tp{\*n - \*e_j}}_{\text{death--death}} \\
        &- \underbrace{\sum_{i, j \in \stp{d}} r_i\tp{\*n} n_j u_i\tp{\*n} v_j\tp{\*n - \*e_j} - \sum_{i, j \in \stp{d}} n_i r_j\tp{\*n} u_i\tp{\*n - \*e_i} v_j\tp{\*n}}_{\text{birth--death}}.
    \end{aligned}
    \]
    The birth--birth product in \(\*E_{\mu}\stp{\+L f \, \+L g}\) is
    \[
    \*E_{\*n \sim \mu}\stp{\sum_{i, j \in \stp{d}} r_i\tp{\*n} r_j\tp{\*n} u_i\tp{\*n} v_j\tp{\*n}}.
    \]
    For the death--death product, first consider \(i \ne j\). Shifting \(\*n\) to \(\*n + \*e_i + \*e_j\) yields
    \[
    \begin{aligned}
        \*E_{\*n \sim \mu}\stp{\sum_{\substack{i, j \in \stp{d} \\ i \ne j}} n_i n_j u_i\tp{\*n - \*e_i} v_j\tp{\*n - \*e_j}} = \*E_{\*n \sim \mu}\stp{\sum_{\substack{i, j \in \stp{d} \\ i \ne j}} r_{ij}\tp{\*n} u_i\tp{\*n + \*e_j} v_j\tp{\*n + \*e_i}}.
    \end{aligned}
    \]
    On the diagonal, the decomposition \(n_i^2 = n_i \tp{n_i - 1} + n_i\), followed by shifts by \(2\*e_i\) and \(\*e_i\), respectively, gives
    \[
    \begin{aligned}
        \*E_{\*n \sim \mu}\stp{\sum_{i \in \stp{d}} n_i^2 u_i\tp{\*n - \*e_i} v_i\tp{\*n - \*e_i}} =& \, \*E_{\*n \sim \mu}\stp{\sum_{i \in \stp{d}} r_{ii}\tp{\*n} u_i\tp{\*n + \*e_i} v_i\tp{\*n + \*e_i}} \\
        &+ \*E_{\*n \sim \mu}\stp{\sum_{i \in \stp{d}} r_i\tp{\*n} u_i\tp{\*n} v_i\tp{\*n}}.
    \end{aligned}
    \]
    For the first birth--death product, shifting \(\*n\) to \(\*n + \*e_j\) and using detailed balance gives
    \[
    \*E_{\*n \sim \mu}\stp{\sum_{i, j \in \stp{d}} r_i\tp{\*n} n_j u_i\tp{\*n} v_j\tp{\*n - \*e_j}} = \*E_{\*n \sim \mu}\stp{\sum_{i, j \in \stp{d}} r_{ij}\tp{\*n} u_i\tp{\*n + \*e_j} v_j\tp{\*n}}.
    \]
    Similarly, shifting the second birth--death product by \(\*e_i\) gives
    \[
    \*E_{\*n \sim \mu}\stp{\sum_{i, j \in \stp{d}} n_i r_j\tp{\*n} u_i\tp{\*n - \*e_i} v_j\tp{\*n}} = \*E_{\*n \sim \mu}\stp{\sum_{i, j \in \stp{d}} r_{ij}\tp{\*n} u_i\tp{\*n} v_j\tp{\*n + \*e_i}}.
    \]
    Combining these four products gives
    \[
    \begin{aligned}
        \*E_{\mu}\stp{\+L f \, \+L g}
        =& \*E_{\*n \sim \mu}\stp{\sum_{i, j \in \stp{d}} r_{ij}\tp{\*n} \tp{u_i\tp{\*n + \*e_j} v_j\tp{\*n + \*e_i} - u_i\tp{\*n + \*e_j} v_j\tp{\*n} - u_i\tp{\*n} v_j\tp{\*n + \*e_i}}} \\
        &+ \*E_{\*n \sim \mu}\stp{\sum_{i, j \in \stp{d}} r_i\tp{\*n} r_j\tp{\*n} u_i\tp{\*n} v_j\tp{\*n}} + \*E_{\*n \sim \mu}\stp{\sum_{i \in \stp{d}} r_i\tp{\*n} u_i\tp{\*n} v_i\tp{\*n}}.
    \end{aligned}
    \]
    Finally,
    \[
    \begin{aligned}
        u_i\tp{\*n + \*e_j} - u_i\tp{\*n} &= \nabla_i^+ \nabla_j^+ f\tp{\*n}, \\
        v_j\tp{\*n + \*e_i} - v_j\tp{\*n} &= \nabla_i^+ \nabla_j^+ g\tp{\*n},
    \end{aligned}
    \]
    and the elementary identity \(U'V' - U'V - UV' = \tp{U' - U}\tp{V' - V} - UV\) transforms the first line into
    \[
    \*E_{\mu}\stp{\sum_{i, j \in \stp{d}} r_{ij}\tp{\*n} \tp{\nabla_i^+ \nabla_j^+ f\tp{\*n} \nabla_i^+ \nabla_j^+ g\tp{\*n} - u_i\tp{\*n} v_j\tp{\*n}}}.
    \]
    Substituting the definitions of \(r_i\), \(r_{ij}\), \(u_i\), \(v_j\) yields the claimed formula.
\end{proof}

\subsection{Ultra-Log-Concave Measures and Weighted \Poincare Inequalities}
\label{subsec:ultra-log-concavity}

The discrete Bochner formula in \Cref{thm:integrated-gamma-2-form} motivates a natural condition on the stationary measure \(\mu\) under which the integrated Bakry--\Emery criterion holds. Indeed, discarding the non-negative discrete Hessian term gives
\[
\begin{aligned}
    \*E_{\mu}\stp{\Gamma_2\tp{f, f}} \ge& \, \*E_{\*n \sim \mu}\stp{\sum_{i, j \in \stp{d}} \tp{\frac{a\tp{\*n + \*e_i} a\tp{\*n + \*e_j}}{a\tp{\*n}^2} - \frac{a\tp{\*n + \*e_i + \*e_j}}{a\tp{\*n}}} \nabla_i^+ f\tp{\*n} \nabla_j^+ f\tp{\*n}} \\
    &+ \*E_{\*n \sim \mu}\stp{\sum_{i \in \stp{d}} \frac{a\tp{\*n + \*e_i}}{a\tp{\*n}} \tp{\nabla_i^+ f\tp{\*n}}^2}.
\end{aligned}
\]
On the other hand, by \eqref{eq:birth-death-dirichlet-form}, the integrated carr\'e du champ is
\[
\*E_{\mu}\stp{\Gamma\tp{f, f}} = \*E_{\*n \sim \mu}\stp{\sum_{i = 1}^d \frac{\tp{n_i + 1} \mu\tp{\*n + \*e_i}}{\mu\tp{\*n}} \tp{\nabla_i^+ f\tp{\*n}}^2} = \*E_{\*n \sim \mu}\stp{\sum_{i \in \stp{d}} \frac{a\tp{\*n + \*e_i}}{a\tp{\*n}} \tp{\nabla_i^+ f\tp{\*n}}^2}.
\]
Consequently, a sufficient condition for the integrated Bakry--\Emery criterion
\[
\*E_{\mu}\stp{\Gamma_2\tp{f, f}} \ge \delta \, \*E_{\mu}\stp{\Gamma\tp{f, f}}, \quad \forall f \in \+C_{\-c}\tp{\supp\tp{\mu}}
\]
is that the following matrix inequality holds for every \(\*n \in \supp\tp{\mu}\):
\[
\tp{\frac{a\tp{\*n + \*e_i} a\tp{\*n + \*e_j}}{a\tp{\*n}^2} - \frac{a\tp{\*n + \*e_i + \*e_j}}{a\tp{\*n}}}_{i, j \in \stp{d}} + \tp{1 - \delta} \diag \tp{\frac{a\tp{\*n + \*e_i}}{a\tp{\*n}}}_{i \in \stp{d}} \succeq \*O.
\]
Restricting to the active coordinates and applying a diagonal congruence transformation yields an equivalent condition, which motivates the following high-dimensional notion of ultra-log-concavity.

\begin{definition}[Ultra-Log-Concave Measures]
    Let \(\mu\) be a probability measure on \(\bb N^d\) with downward closed support, and define \(a\tp{\*n} \defeq \*n! \mu\tp{\*n}\) for \(\*n \in \bb N^d\). We say that \(\mu\) is ultra-log-concave with parameter \(\delta > 0\), or \(\delta\)-ULC, if the following inequality holds for every \(\*n \in \supp\tp{\mu}\):
    \begin{equation}
        \tp{1 - \frac{a\tp{\*n} a\tp{\*n + \*e_i + \*e_j}}{a\tp{\*n + \*e_i} a\tp{\*n + \*e_j}}}_{i, j \in \+I_{\*n}} + \tp{1 - \delta} \diag \tp{\frac{a\tp{\*n}}{a\tp{\*n + \*e_i}}}_{i \in \+I_{\*n}} \succeq \*O.
        \label{eq:ulc-condition}
    \end{equation}
    Here, \(\+I_{\*n} \defeq \set{i \in \stp{d} \cmid \*n + \*e_i \in \supp\tp{\mu}}\) denotes the set of coordinates along which a birth transition is possible from \(\*n\). In particular, we say \(\mu\) is ultra-log-concave, or ULC, if it is \(1\)-ULC.
\end{definition}

\begin{theorem}
    Every \(\delta\)-ULC probability measure \(\mu\) on \(\bb N^d\) satisfies the weighted \Poincare inequality with constant \(\delta\):
    \[
    \delta \, \*{Var}_{\mu}\stp{f} \le \*E_{\*n \sim \mu}\stp{\sum_{i = 1}^d \frac{\tp{n_i + 1} \mu\tp{\*n + \*e_i}}{\mu\tp{\*n}} \tp{\nabla_i^+ f\tp{\*n}}^2}, \quad \forall f \in \+D\tp{\+E}.
    \]
    Equivalently, the birth-death generator \(\+L\) defined in \eqref{eq:birth-death-generator} has spectral gap at least \(\delta\).
    \label{thm:ulc-poincare-inequality}
\end{theorem}

\begin{remark}
    For measures supported on subsets of the hypercube \(\set{0, 1}^d\), the \(\delta\)-ULC condition and the resulting \Poincare inequality have previously appeared in \cite[Theorem~1.9]{CCCYZ25} and \cite[Theorem~1.2]{GJMPPS26}. An infinite-dimensional counterpart for point processes, whose configuration space may be viewed as an infinite-dimensional hypercube, appeared earlier in \cite[Theorem~3.4]{KKO13}. For a more detailed discussion of the connection between Glauber dynamics for point processes and birth-death chains for probability measures on \(\bb N^d\), see \Cref{sec:point-processes}.
\end{remark}

\begin{proof}[Proof of \Cref{thm:ulc-poincare-inequality}]
    For completeness, we make the passage to infinite support explicit. For \(N \in \bb N\), let
    \[
    S_N \defeq \supp\tp{\mu} \cap \set{\*n \in \bb N^d \cmid \norm{\*n}_{\infty} \le N}, \qquad \mu_N\tp{\*n} \defeq \frac{\mu\tp{\*n}}{\mu\tp{S_N}}, \quad \*n \in S_N.
    \]
    The set \(S_N\) is finite and downward closed. The ULC matrix for \(\mu_N\) is obtained from the ULC matrix for \(\mu\) by restricting to the active coordinates in \(S_N\) and, when a second step exits the box, adding a non-negative diagonal matrix. Hence \(\mu_N\) is also \(\delta\)-ULC. The finite-state integrated Bakry--\Emery criterion therefore gives
    \[
    \delta \, \*{Var}_{\mu_N}\stp{f} \le \+E_N\tp{f, f}
    \]
    for every function on \(S_N\). If \(f \in \+D\tp{\+E}\), then \(\*{Var}_{\mu_N}\stp{f} \to \*{Var}_{\mu}\stp{f}\) and \(\+E_N\tp{f, f} \to \+E\tp{f, f}\) as \(N \to \infty\). Passing to the limit proves the claimed inequality on \(\+D\tp{\+E}\).
\end{proof}

\begin{example}[Exact \(\delta\)-ULC Measures on \(\bb N\)]
    In one dimension, the \(\delta\)-ULC condition becomes
    \[
    1 - \frac{a\tp{n} a\tp{n + 2}}{a \tp{n + 1}^2} + \tp{1 - \delta} \frac{a\tp{n}}{a\tp{n + 1}} \ge 0, \quad \forall n \cmid n + 1 \in \supp\tp{\mu}.
    \]
    If equality holds for every \(n\) such that \(n + 1 \in \supp\tp{\mu}\), then a direct computation yields the following classification:
    \begin{itemize}
        \item If \(\delta \in \tp{0, 1}\), then \(\mu\) is a negative binomial distribution \(\-{NB}\tp{r, \delta}\) for some \(r > 0\).
        \item If \(\delta = 1\), then \(\mu\) is a Poisson distribution \(\-{Pois}\tp{\lambda}\) for some \(\lambda > 0\).
        \item If \(\delta > 1\), then \(\mu\) is a binomial distribution \(\-{Bin}\tp{N, 1 - 1 / \delta}\) for some \(N \in \bb N_{> 0}\).
    \end{itemize}
    Moreover, \(\delta\) is the optimal constant in the weighted \Poincare inequality for \(\mu\); equivalently, the birth-death generator \(\+L\) defined in \eqref{eq:birth-death-generator} has spectral gap \(\delta\). The centered identity function \(f\tp{n} \defeq n - \mean{\mu}\) is an eigenfunction of \(-\+L\) with eigenvalue \(\delta\).
    \label{ex:ulc-exact-1d}
\end{example}

\begin{example}[ULC Measures on \(\bb N\)]
    In one dimension, the ULC condition becomes
    \[
    1 - \frac{a\tp{n} a\tp{n + 2}}{a \tp{n + 1}^2} \ge 0, \quad \forall n \cmid n + 1 \in \supp\tp{\mu}.
    \]
    Equivalently, the sequence \(a\tp{n} = n! \mu\tp{n}\) is log-concave. This is precisely the classical notion of an ultra-log-concave measure, an important class of discrete probability measures; see \cite[Section~1]{AMM21} for further discussion. By \Cref{thm:ulc-poincare-inequality}, every ULC measure on \(\bb N\) satisfies the weighted \Poincare inequality
    \[
    \*{Var}_{\mu}\stp{f} \le \*E_{n \sim \mu}\stp{\frac{\tp{n + 1} \mu\tp{n + 1}}{\mu\tp{n}} \tp{\nabla^+ f\tp{n}}^2}, \quad \forall f \in \+D\tp{\+E}.
    \]
    Moreover, ultra-log-concavity of \(\mu\) implies that the ratios
    \[
    \frac{\tp{n + 1} \mu\tp{n + 1}}{\mu\tp{n}} = \frac{a\tp{n + 1}}{a\tp{n}}, \quad n \in \supp\tp{\mu},
    \]
    are non-increasing in \(n\). In particular, they are bounded above by \(\mu\tp{1} / \mu\tp{0}\). Consequently, the weighted inequality above yields the classical unweighted discrete \Poincare inequality
    \[
    \*{Var}_{\mu}\stp{f} \le \frac{\mu\tp{1}}{\mu\tp{0}} \*E_{\mu}\stp{\tp{\nabla^+ f}^2}, \quad \forall f \in \+D\tp{\+E}.
    \]
    This recovers \cite[Theorem~1.5]{Joh17}. The constant \(\mu\tp{1}/\mu\tp{0}\), however, is not optimal for the class of ULC measures: it can be improved to \(\mean{\mu}\); see \cite[Corollary~2.4]{DJ13}.
    \label{ex:one-dim-ulc-poincare}
\end{example}

\begin{example}[Product Measures on \(\bb N^d\)]
    Let \(\mu \defeq \bigotimes_{i = 1}^d \mu_i\) be a product measure on \(\bb N^d\). Then \(\mu\) is \(\delta\)-ULC if and only if each factor \(\mu_i\) is \(\delta\)-ULC. The spectral gap of the birth-death generator \(\+L\) defined in \eqref{eq:birth-death-generator} is exactly the minimum of the spectral gaps of the one-dimensional generators.
    \label{ex:ulc-product-measure}
\end{example}

\begin{example}[A Two-Queue Model]
    Fix \(b_{\-{long}}, b_{\-{short}} > 0\) and consider two queues with \(M/M/\infty\)-type service: at state \(\*n \in \bb N^2\), customers depart queue \(i\) at rate \(n_i\). Arrivals join the longer queue at rate \(b_{\-{long}}\) and the shorter queue at rate \(b_{\-{short}}\); when the queues are tied, each receives arrivals at rate \(b_{\-{long}}\). The queue-length vector therefore evolves as a birth-death chain on \(\bb N^2\), reversible with respect to
    \[
    \mu\tp{\*n} \propto \frac{1}{\*n!} b_{\-{long}}^{\max\set{n_1, n_2}} b_{\-{short}}^{\min\set{n_1, n_2}}, \quad \*n \in \bb N^2.
    \]
    Indeed, the ratio \(\tp{n_i + 1} \mu\tp{\*n + \*e_i} / \mu\tp{\*n}\) equals \(b_{\-{long}}\) when \(n_i \ge n_{3 - i}\) and \(b_{\-{short}}\) otherwise, exactly reproducing the prescribed birth rates.

    Write \(\Delta b \defeq b_{\-{long}} - b_{\-{short}}\), and let \(\*M_{\delta}\tp{\*n}\) denote the matrix on the left-hand side of \eqref{eq:ulc-condition}. For \(n_1 \ge n_2\), a direct computation yields
    \[
    \*M_{\delta}\tp{\*n} = \begin{cases}
        \tp{1 - \delta} \diag\tp{b_{\-{long}}^{-1}, b_{\-{short}}^{-1}}, & n_1 \ge n_2 + 2, \\
        \diag\tp{b_{\-{long}}^{-1} \tp{1 - \delta}, b_{\-{short}}^{-1} \tp{1 - \delta - \Delta b}}, & n_1 = n_2 + 1, \\
        b_{\-{long}}^{-1} \begin{pmatrix} 1 - \delta & \Delta b \\ \Delta b & 1 - \delta \end{pmatrix}, & n_1 = n_2.
    \end{cases}
    \]
    For \(n_2 > n_1\), the matrix is obtained by interchanging the two coordinates. These matrices are positive semidefinite if and only if \(\delta \le 1 - \abs{\Delta b}\). Hence a positive ULC parameter exists precisely when \(\abs{\Delta b} < 1\). In that case, \(\mu\) is \(\tp{1 - \abs{\Delta b}}\)-ULC, and the queueing process has spectral gap at least \(1 - \abs{\Delta b}\).
    \label{ex:two-queue-model}
\end{example}

\subsection{Local Curvature Matrix and Discrete Brascamp--Lieb Inequalities}
\label{subsec:discrete-brascamp-lieb}

For \(\*n \in \supp\tp{\mu}\), let \(\+I_{\*n} \defeq \set{i \in \stp{d} \cmid \*n + \*e_i \in \supp\tp{\mu}}\) be the set of active birth coordinates. Define the local curvature matrix
\begin{equation}
\*K_{\*n} \defeq \tp{1 - \frac{a\tp{\*n} a\tp{\*n + \*e_i + \*e_j}}{a\tp{\*n + \*e_i} a\tp{\*n + \*e_j}}}_{i, j \in \+I_{\*n}} + \diag \tp{\frac{a\tp{\*n}}{a\tp{\*n + \*e_i}}}_{i \in \+I_{\*n}},
\label{eq:local-curvature-matrix}
\end{equation}
where \(a\tp{\*n} \defeq \*n! \mu\tp{\*n}\). This matrix is the normalized first-order component of the discrete Bochner formula in \Cref{thm:integrated-gamma-2-form}.

Moreover, the \(\delta\)-ULC condition \eqref{eq:ulc-condition} is equivalent to the pointwise matrix inequality
\begin{equation}
    \*K_{\*n} \succeq \delta \, \diag \tp{\frac{a\tp{\*n}}{a\tp{\*n + \*e_i}}}_{i \in \+I_{\*n}}, \quad \forall \*n \in \supp\tp{\mu}.
    \label{eq:ulc-curvature-lower-bound}
\end{equation}
Thus, \Cref{thm:ulc-poincare-inequality} is consistent with the Bakry--\Emery principle that curvature lower bounds imply functional inequalities.

This interpretation naturally leads to a discrete analogue of the classical Brascamp--Lieb inequality. Indeed, if \(\mu\) is a probability measure on \(\bb R^d\) with density proportional to \(e^{-V}\), where \(V\) is twice continuously differentiable and strictly convex, then
\[
\*{Var}_{\mu}\stp{f} \le \*E_{\mu}\stp{\nabla f^{\top} \tp{\nabla^2 V}^{-1} \nabla f}.
\]
Rather than replacing the Hessian by a uniform lower bound (equivalently, assuming strong log-concavity), the Brascamp--Lieb inequality retains the full pointwise curvature. Motivated by this observation, we establish a discrete counterpart that depends on the entire local curvature matrix \(\*K_{\*n}\), rather than only the lower bound \eqref{eq:ulc-curvature-lower-bound}. While existing Bochner-type arguments for discrete Markov chains typically exploit only uniform curvature bounds to derive \Poincare inequalities \cite{KKO13,Joh17,GJMPPS26}, the following theorem shows that preserving the full state-dependent curvature yields a discrete Brascamp--Lieb inequality.

\begin{theorem}[Discrete Brascamp--Lieb Inequality]
    Let \(\mu\) be a probability measure on \(\bb N^d\) with downward closed support, and let \(\*K_{\*n}\) be the local curvature matrix defined in \eqref{eq:local-curvature-matrix}. Suppose that \(\*K_{\*n} \succ \*O\) for every \(\*n \in \supp\tp{\mu}\). Then
    \[
    \*{Var}_{\mu}\stp{f} \le \*E_{\mu}\stp{\tp{\nabla^+ f}_{\+I}^{\top} \*K^{-1} \tp{\nabla^+ f}_{\+I}}, \quad \forall f \in \+C_{\-c}\tp{\supp\tp{\mu}},
    \]
    where \(\tp{\cdot}_{\+I_{\*n}}\) denotes the restriction to the coordinates in \(\+I_{\*n} \defeq \set{i \in \stp{d} \cmid \*n + \*e_i \in \supp\tp{\mu}}\). The integrand on the right-hand side is evaluated using the corresponding local quantities; that is, \(\*K = \*K_{\*n}\), \(\+I = \+I_{\*n}\), and \(\nabla^+ f = \nabla^+ f\tp{\*n}\).
    \label{thm:discrete-brascamp-lieb}
\end{theorem}

\begin{remark}
    The inequality extends by continuity to the completion of \(\+C_{\-c}\tp{\supp\tp{\mu}}\) with respect to the norm
    \[
    \norm{f}_{\-{BL}}^2 \defeq \norm{f}_{L^2\tp{\mu}}^2 + \*E_{\mu}\stp{\tp{\nabla^+ f}_{\+I}^{\top} \*K^{-1} \tp{\nabla^+ f}_{\+I}}.
    \]
\end{remark}

\begin{remark}
    If \(\mu\) is \(\delta\)-ULC, then \eqref{eq:ulc-curvature-lower-bound} implies
    \[
    \*K_{\*n}^{-1} \preceq \frac{1}{\delta} \diag \tp{\frac{a\tp{\*n + \*e_i}}{a\tp{\*n}}}_{i \in \+I_{\*n}}, \quad \forall \*n \in \supp\tp{\mu}.
    \]
    Hence \Cref{thm:discrete-brascamp-lieb} recovers the weighted \Poincare inequality in \Cref{thm:ulc-poincare-inequality}.
\end{remark}

\begin{proof}[Proof of \Cref{thm:discrete-brascamp-lieb}]
    First suppose that \(\mu\) is finitely supported and replace \(f\) by \(f - \*E_{\mu}\stp{f}\).
    
    Let \(g \defeq \tp{-\+L}^{\dagger} f\), where \(\+L\) is the birth-death generator defined in \eqref{eq:birth-death-generator}, and \(\tp{-\+L}^{\dagger}\) is its Moore--Penrose pseudoinverse. Then \(g\) is the unique solution to the Poisson equation \(-\+L g = f\) with \(\*E_{\mu}\stp{g} = 0\).

    By reversibility of \(\+L\) and \eqref{eq:birth-death-dirichlet-form},
    \[
    \*{Var}_{\mu}\stp{f} = \*E_{\mu}\stp{f^2} = \*E_{\mu}\stp{f \, \tp{-\+L} g} = \*E_{\mu}\stp{\Gamma\tp{f, g}} = \*E_{\*n \sim \mu}\stp{\sum_{i \in \stp{d}} \frac{a\tp{\*n + \*e_i}}{a\tp{\*n}} \nabla_i^+ f\tp{\*n} \nabla_i^+ g\tp{\*n}}.
    \]
    For \(i \in \+I_{\*n}\), write \(u_i\tp{\*n} \defeq \nabla_i^+ f\tp{\*n}\) and \(v_i\tp{\*n} \defeq \frac{a\tp{\*n + \*e_i}}{a\tp{\*n}} \nabla_i^+ g\tp{\*n}\). Then the above expression becomes
    \[
    \*{Var}_{\mu}\stp{f} = \*E_{\*n \sim \mu}\stp{\sum_{i \in \+I_{\*n}} u_i\tp{\*n} v_i\tp{\*n}} = \*E_{\mu}\stp{\*u^{\top} \*v}.
    \]
    Since \(\*K_{\*n} \succ \*O\) for every \(\*n \in \supp\tp{\mu}\), the Cauchy--Schwarz inequality gives
    \begin{equation}
        \begin{aligned}
            \*{Var}_{\mu}\stp{f} &= \*E_{\mu}\stp{\tp{\*K^{-1 / 2} \*u}^{\top} \tp{\*K^{1 / 2} \*v}} \\
            &\le \*E_{\mu}\stp{\norm{\*K^{-1 / 2} \*u}_2^2}^{1 / 2} \*E_{\mu}\stp{\norm{\*K^{1 / 2} \*v}_2^2}^{1 / 2} \\
            &= \*E_{\mu}\stp{\*u^{\top} \*K^{-1} \*u}^{1 / 2} \*E_{\mu}\stp{\*v^{\top} \*K \*v}^{1 / 2}.
        \end{aligned}
        \label{eq:discrete-brascamp-lieb-cauchy-schwarz}
    \end{equation}

    On the other hand,
    \[
    \*{Var}_{\mu}\stp{f} = \*E_{\mu}\stp{f^2} = \*E_{\mu}\stp{\tp{\+L g}^2} = \*E_{\mu}\stp{\Gamma_2\tp{g, g}}.
    \]
    By the discrete Bochner formula in \Cref{thm:integrated-gamma-2-form},
    \begin{align*}
        \*{Var}_{\mu}\stp{f} &= \, \*E_{\*n \sim \mu}\stp{\sum_{i, j \in \stp{d}} \frac{a\tp{\*n + \*e_i + \*e_j}}{a\tp{\*n}} \tp{\nabla_i^+ \nabla_j^+ g\tp{\*n}}^2} + \*E_{\*n \sim \mu}\stp{\sum_{i \in \stp{d}} \frac{a\tp{\*n + \*e_i}}{a\tp{\*n}} \tp{\nabla_i^+ g\tp{\*n}}^2} \\
        &\phantom{={}} + \*E_{\*n \sim \mu}\stp{\sum_{i, j \in \stp{d}} \tp{\frac{a\tp{\*n + \*e_i} a\tp{\*n + \*e_j}}{a\tp{\*n}^2} - \frac{a\tp{\*n + \*e_i + \*e_j}}{a\tp{\*n}}} \nabla_i^+ g\tp{\*n} \nabla_j^+ g\tp{\*n}} \\
        &\ge \*E_{\*n \sim \mu}\stp{\sum_{i, j \in \+I_{\*n}} \tp{1 - \frac{a\tp{\*n} a\tp{\*n + \*e_i + \*e_j}}{a\tp{\*n + \*e_i} a\tp{\*n + \*e_j}}} v_i\tp{\*n} v_j\tp{\*n}} + \*E_{\*n \sim \mu}\stp{\sum_{i \in \+I_{\*n}} \frac{a\tp{\*n}}{a\tp{\*n + \*e_i}} v_i\tp{\*n}^2} \\
        &= \*E_{\mu}\stp{\*v^{\top} \*K \*v}.
    \end{align*}
    Together with \eqref{eq:discrete-brascamp-lieb-cauchy-schwarz}, this gives
    \[
    \*{Var}_{\mu}\stp{f} \le \*E_{\mu}\stp{\*u^{\top} \*K^{-1} \*u} = \*E_{\mu}\stp{\tp{\nabla^+ f}_{\+I}^{\top} \*K^{-1} \tp{\nabla^+ f}_{\+I}}.
    \]

    To pass to an arbitrary downward closed support, fix \(f \in \+C_{\-c}\tp{\supp\tp{\mu}}\) and use the coordinate-box restrictions \(S_N\) and conditional measures \(\mu_N\) introduced in the proof of \Cref{thm:ulc-poincare-inequality}. The local curvature matrix for \(\mu_N\) is obtained by restricting \(\*K_{\*n}\) to the active coordinates in \(S_N\) and adding non-negative diagonal entries when a second step exits the box; it is therefore positive definite. Apply the finite-state inequality to \(f|_{S_N}\). For all sufficiently large \(N\), every state at which \(\nabla^+ f\) is nonzero lies away from the boundary of the box, so its local curvature matrix agrees with \(\*K_{\*n}\). Letting \(N \to \infty\) proves the inequality for \(f \in \+C_{\-c}\tp{\supp\tp{\mu}}\).
\end{proof}

\begin{corollary}
    In the same setting as \Cref{thm:discrete-brascamp-lieb}, we have
    \[
    \cov{\mu} \preceq \*E_{\mu}\stp{P_{\+I}^{\top}\*K^{-1} P_{\+I}},
    \]
    where \(P_{\+I_{\*n}}\) is the projection onto the active coordinates \(\+I_{\*n} \defeq \set{i \in \stp{d} \cmid \*n + \*e_i \in \supp\tp{\mu}}\). In other words, \(P_{\+I_{\*n}}^{\top}\*K_{\*n}^{-1} P_{\+I_{\*n}}\) is simply \(\*K_{\*n}^{-1}\) zero-padded to \(\bb{R}^{d \times d}\).
    \label{cor:discrete-brascamp-lieb-cov-bound}
\end{corollary}

\begin{proof}[Proof of \Cref{cor:discrete-brascamp-lieb-cov-bound}]
    For any \(\*v \ne \*0\), consider the linear test function \(\ell_{\*v}\tp{\*n} \defeq \*v^{\top} \*n\). \Cref{thm:discrete-brascamp-lieb} gives
    \[
    \*v^{\top} \cov{\mu} \*v = \*{Var}_{\mu}\stp{\ell_{\*v}} \le \*E_{\mu}\stp{\tp{\nabla^+ \ell_{\*v}}_{\+I}^{\top} \*K^{-1} \tp{\nabla^+ \ell_{\*v}}_{\+I}} = \*E_{\mu}\stp{\*v_{\+I}^{\top} \*K^{-1} \*v_{\+I}} = \*v^{\top} \*E_{\mu}\stp{P_{\+I}^{\top}\*K^{-1} P_{\+I}} \*v.
    \]
\end{proof}

The discrete Brascamp--Lieb inequality in \Cref{thm:discrete-brascamp-lieb} is associated with the birth-death chain having linear death rates, whose generator is given by \eqref{eq:birth-death-generator}. Another canonical choice is the birth-death chain with unit birth rates:
\begin{equation}
    \+L f\tp{\*n} \defeq \sum_{i = 1}^d \nabla_i^+ f\tp{\*n} + \sum_{i = 1}^d \frac{\mu\tp{\*n - \*e_i}}{\mu\tp{\*n}} \nabla_i^- f\tp{\*n},
    \label{eq:unit-birth-generator}
\end{equation}
where \(\mu\) is assumed to be supported on an upward closed subset of \(\bb Z^d\). This generator was studied in \cite{Joh17} via the integrated Bakry--\Emery criterion. Although \cite{Joh17} restricts the state space to \(\bb N^d\), the argument extends readily to birth-death chains on \(\bb Z^d\).

Set
\[
\+E_+\tp{f, f} \defeq \*E_{\mu}\stp{\norm{\nabla^+ f}_2^2}, \qquad \+D\tp{\+E_+} \defeq \set{f \in L^2\tp{\mu} \cmid \+E_+\tp{f, f} < \infty}.
\]
It was shown in \cite[Theorem~1.5, Proposition~9.1]{Joh17} that the \Poincare inequality
\[
c \, \*{Var}_{\mu}\stp{f} \le \*E_{\mu}\stp{\norm{\nabla^+ f}_2^2}, \quad \forall f \in \+D\tp{\+E_+}
\]
holds under a \(c\)-log-concavity condition:
\begin{equation}
    \tp{1 - \frac{\mu\tp{\*n} \mu\tp{\*n - \*e_i - \*e_j}}{\mu\tp{\*n - \*e_i} \mu\tp{\*n - \*e_j}}}_{i, j \in \+I_{\*n}} \succeq c\, \diag\tp{\frac{\mu\tp{\*n}}{\mu\tp{\*n - \*e_i}}}_{i \in \+I_{\*n}}, \quad \forall \*n \in \supp\tp{\mu},
    \label{eq:c-log-concave}
\end{equation}
where \(\+I_{\*n} \defeq \set{i \in \stp{d} \cmid \*n - \*e_i \in \supp\tp{\mu}}\). As in the linear-death setting, one can strengthen this result to a discrete Brascamp--Lieb inequality by retaining the full local curvature matrix.

\begin{theorem}[Discrete Brascamp--Lieb Inequality, Unit Birth Rate Chain]
    Let \(\mu\) be a probability measure on \(\bb Z^d\) with upward closed support. For each \(\*n \in \supp\tp{\mu}\), define the local curvature matrix
    \[
    \*K_{\*n} \defeq \tp{1 - \frac{\mu\tp{\*n} \mu\tp{\*n - \*e_i - \*e_j}}{\mu\tp{\*n - \*e_i} \mu\tp{\*n - \*e_j}}}_{i, j \in \+I_{\*n}},
    \]
    where \(\+I_{\*n} \defeq \set{i \in \stp{d} \cmid \*n - \*e_i \in \supp\tp{\mu}}\). Suppose that \(\*K_{\*n} \succ \*O\) for every \(\*n \in \supp\tp{\mu}\). Then
    \[
    \*{Var}_{\mu}\stp{f} \le \*E_{\mu}\stp{\tp{\nabla^- f}_{\+I}^{\top} \*K^{-1} \tp{\nabla^- f}_{\+I}}, \quad \forall f \in \+C_{\-c}\tp{\supp\tp{\mu}},
    \]
    where \(\tp{\cdot}_{\+I_{\*n}}\) denotes the restriction to the coordinates in \(\+I_{\*n} \defeq \set{i \in \stp{d} \cmid \*n - \*e_i \in \supp\tp{\mu}}\). The integrand on the right-hand side is evaluated using the corresponding local quantities; that is, \(\*K = \*K_{\*n}\), \(\+I = \+I_{\*n}\), and \(\nabla^- f = \nabla^- f\tp{\*n}\).
    \label{thm:discrete-brascamp-lieb-unit-birth}
\end{theorem}

\begin{remark}
    The inequality extends by continuity to the completion of \(\+C_{\-c}\tp{\supp\tp{\mu}}\) with respect to the norm
    \[
    \norm{f}_{\-{BL}, -}^2 \defeq \norm{f}_{L^2\tp{\mu}}^2 + \*E_{\mu}\stp{\tp{\nabla^- f}_{\+I}^{\top} \*K^{-1} \tp{\nabla^- f}_{\+I}}.
    \]
\end{remark}

\begin{remark}
    Under the \(c\)-log-concavity condition,
    \[
    \*K_{\*n}^{-1} \preceq \frac{1}{c} \, \diag\tp{\frac{\mu\tp{\*n - \*e_i}}{\mu\tp{\*n}}}_{i \in \+I_{\*n}}, \quad \forall \*n \in \supp\tp{\mu}.
    \]
    Consequently, \Cref{thm:discrete-brascamp-lieb-unit-birth} recovers the \Poincare inequality proved in \cite{Joh17}. Indeed,
    \[
    \*{Var}_{\mu}\stp{f} \le \frac{1}{c} \*E_{\mu}\stp{\sum_{i \in \+I_{\*n}} \frac{\mu\tp{\*n - \*e_i}}{\mu\tp{\*n}} \tp{\nabla_i^- f\tp{\*n}}^2} = \frac{1}{c} \*E_{\mu}\stp{\norm{\nabla^+ f}_2^2}, \quad \forall f \in \+C_{\-c}\tp{\supp\tp{\mu}}.
    \]
\end{remark}

\begin{corollary}
    In the same setting as \Cref{thm:discrete-brascamp-lieb-unit-birth}, we have
    \[
    \cov{\mu} \preceq \*E_{\mu}\stp{P_{\+I}^{\top}\*K^{-1} P_{\+I}},
    \]
    where \(P_{\+I_{\*n}}\) is the projection onto the active coordinates \(\+I_{\*n} \defeq \set{i \in \stp{d} \cmid \*n - \*e_i \in \supp\tp{\mu}}\). In other words, \(P_{\+I_{\*n}}^{\top}\*K_{\*n}^{-1} P_{\+I_{\*n}}\) is simply \(\*K_{\*n}^{-1}\) zero-padded to \(\bb{R}^{d \times d}\).
\end{corollary}

\subsection{Variance Bounds in the Hardcore Model}
\label{subsec:variance-hardcore}

Let \(G = \tp{V, E}\) be a finite graph, and let \(\+I\tp{G}\) denote the collection of its independent sets. For a fugacity \(\lambda > 0\), the hardcore measure on \(G\) is defined by
\[
\mu_{G, \lambda}\tp{I} \propto \lambda^{\abs{I}}, \quad I \in \+I\tp{G}.
\]
We identify each independent set with its indicator vector in \(\set{0, 1}^V\), so that \(\mu_{G, \lambda}\) is supported on a downward closed subset of \(\bb N^V\). Throughout this subsection, let \(I \sim \mu_{G, \lambda}\).

The partition function (equivalently, the independence polynomial evaluated at \(\lambda\)), occupancy fraction, and variance fraction are defined by
\[
Z_G\tp{\lambda} \defeq \sum_{I \in \+I\tp{G}} \lambda^{\abs{I}}, \qquad \alpha_G\tp{\lambda} \defeq \frac{1}{\abs{V}} \*E\stp{\abs{I}}, \qquad V_G\tp{\lambda} \defeq \frac{1}{\abs{V}} \*{Var}\stp{\abs{I}}.
\]
The partition function and occupancy fraction, particularly their extremal behavior, have been studied extensively in combinatorics \cite{Kah01,GT04,Zha10,CR11,DJPR17,SSSZ19,DK25}, whereas comparatively less is known about the variance fraction \cite{DST25,ZX26}. The discrete Brascamp--Lieb inequality in \Cref{thm:discrete-brascamp-lieb} yields the following upper bound.

\begin{theorem}
    Let \(G\) be \(\Delta\)-regular, let \(\*A_G\) be its adjacency matrix, and set \(\lambda^{\star} \defeq -\lambda_{\min}\tp{\*A_G}\). Suppose that
    \[
    1 + \lambda^{-1} \ge \lambda^{\star}.
    \]
    Then
    \[
    V_G\tp{\lambda} \le \frac{\lambda}{1 + \tp{\Delta + 1} \lambda}.
    \]
    \label{thm:hardcore-cardinality-variance}
\end{theorem}

\begin{remark}
    A related variance bound follows from the spectral-independence consequences of \(\delta\)-ULC. In particular, for \(\delta \in \tp{0, 1}\), the stronger low-fugacity condition
    \[
    1 + \tp{1 - \delta} \lambda^{-1} > \lambda^{\star}
    \]
    implies that \(\mu_{G, \lambda}\) is \(\delta\)-ULC; see \cite[Section~5.2]{CCCYZ25}. By \Cref{cor:ulc-cov-bound}, it is therefore \(1 / \delta\)-spectrally independent:
    \[
    \cov{\mu_{G, \lambda}} \preceq \frac{1}{\delta} \diag\tp{\mean{\mu_{G, \lambda}}},
    \]
    and hence
    \[
    V_G\tp{\lambda} \le \frac{1}{\delta} \alpha_G\tp{\lambda}.
    \]
    At the boundary \(1 + \lambda^{-1} = \lambda^{\star}\), no positive \(\delta\) satisfies the stronger condition, so this estimate degenerates. In contrast, the bound in \Cref{thm:hardcore-cardinality-variance} remains finite at the boundary because it retains the full local curvature matrix and evaluates its inverse in the direction \(\*1_V\).
\end{remark}

\begin{proof}[Proof of \Cref{thm:hardcore-cardinality-variance}]
    For \(I \in \+I\tp{G}\), let \(N_G\tp{I}\) be the open neighborhood of \(I\), and define the set of unblocked vertices by
    \[
    V_I \defeq V \setminus \tp{I \cup N_G\tp{I}}.
    \]
    These vertices are precisely the active birth coordinates at \(I\). Substituting the hardcore weights into \eqref{eq:local-curvature-matrix} gives
    \begin{equation}
        \*K_I = \*A_{G\stp{V_I}} + \tp{1 + \lambda^{-1}} \*I_{V_I},
        \label{eq:hardcore-local-curvature}
    \end{equation}
    where \(\*A_{G\stp{V_I}}\) is the adjacency matrix of the subgraph induced by \(V_I\).

    Set \(s \defeq 1 + \lambda^{-1}\), and first suppose that \(s > \lambda^{\star}\). The curvature matrix at the empty independent set is
    \[
    \*K_{\varnothing} = \*A_G + s \*I_V \succ \*O.
    \]
    By \eqref{eq:hardcore-local-curvature}, each \(\*K_I\) is a principal submatrix of \(\*K_{\varnothing}\) and is therefore positive definite. Applying \Cref{thm:discrete-brascamp-lieb} to the cardinality function \(f\tp{I} \defeq \abs{I}\), for which \(\nabla_v^+ f\tp{I} = 1\) whenever \(v \in V_I\), gives
    \[
    \*{Var}\stp{\abs{I}} \le \*E\stp{\*1_{V_I}^{\top} \*K_I^{-1} \*1_{V_I}},
    \]
    where \(\*1_{V_I}\) denotes the all-ones vector indexed by \(V_I\).

    Fix \(I \in \+I\tp{G}\), and extend vectors indexed by \(V_I\) by zero outside \(V_I\). The variational representation of an inverse quadratic form gives
    \[
    \*1_{V_I}^{\top} \*K_I^{-1} \*1_{V_I} = \sup_{\*z \in \bb{R}^{V_I}} \tp{2 \*1_{V_I}^{\top} \*z - \*z^{\top} \*K_I \*z} \le \sup_{\*z \in \bb{R}^V} \tp{2 \*1_V^{\top} \*z - \*z^{\top} \*K_{\varnothing} \*z} = \*1_V^{\top} \*K_{\varnothing}^{-1} \*1_V.
    \]
    Since \(G\) is \(\Delta\)-regular,
    \[
    \*K_{\varnothing} \*1_V = \tp{\Delta + s} \*1_V,
    \]
    so
    \[
    \*1_V^{\top} \*K_{\varnothing}^{-1} \*1_V = \frac{\abs{V}}{\Delta + s} = \frac{\abs{V} \lambda}{1 + \tp{\Delta + 1} \lambda}.
    \]
    Combining the preceding inequalities and dividing by \(\abs{V}\) proves the claim when \(s > \lambda^{\star}\).

    If \(s = \lambda^{\star}\), choose \(\lambda_k \to \lambda^{-}\). Then \(1 + \lambda_k^{-1} > \lambda^{\star}\), so the strict case applies to each \(\lambda_k\). The corresponding hardcore measures converge to \(\mu_{G, \lambda}\), and the state space \(\+I\tp{G}\) is finite. Passing to the limit proves the boundary case.
\end{proof}

\section{Poisson Stochastic Localization}
\label{sec:poisson-stochastic-localization}

In this section, we study Poisson stochastic localization, a generalization of the negative-fields localization introduced in \cite{CE25}. Our analysis combines the localization-scheme framework of \cite{CE25} with the limiting-process perspective developed in \cite{CCCYZ25}. 

The section is organized as follows. In \Cref{subsec:localization-schemes-preliminaries}, we introduce the localization-scheme framework. In \Cref{subsec:noising-denoising,subsec:localization-process}, we construct Poisson stochastic localization from binomial thinning noise. In \Cref{subsec:associated-limiting-process}, we analyze the limiting process and its functional inequalities. In \Cref{subsec:approximate-conservation-variance-entropy}, we identify spectral independence as a sufficient condition for approximate conservation of variance and entropy. In \Cref{subsec:ulc-mlsi}, we prove a Wu-type modified log-Sobolev inequality for ultra-log-concave measures. Finally, in \Cref{subsec:ulc-log-concave-generating-functions}, we connect ultra-log-concavity with log-concavity of the probability-generating function.

The construction of the noising and denoising processes in \Cref{subsec:noising-denoising} has already appeared in the context of discrete generative modeling \cite{SBP26}. The denoising process itself is precisely the Poisson--\Follmer process, which has also found applications to log-concavity and functional inequalities \cite{KL18,ALS25,LRS25}.

\subsection{The Localization-Scheme Framework}
\label{subsec:localization-schemes-preliminaries}

Localization schemes turn a target measure into a martingale of progressively more informative conditional laws \cite{CE25}. This allows one to study functional inequalities for the target through the evolution of conditional variance and entropy.

\begin{definition}[Localization Processes and Schemes]
    Let \(\mu\) be a probability measure on a finite or countable state space \(\Omega\), and let \(T > 0\). A localization process for \(\mu\) is an \(\tp{\+F_t}_{t \in \stp{0, T}}\)-adapted \(\+P\tp{\Omega}\)-valued process \(\tp{\nu_t}_{t \in \stp{0, T}}\) such that \(\nu_0 = \mu\), \(\nu_T = \delta_X\) for some \(X \sim \mu\), and \(\tp{\*E_{\nu_t}\stp{f}}_{t \in \stp{0, T}}\) is an \(\tp{\+F_t}_{t \in \stp{0, T}}\)-martingale for every bounded \(f: \Omega \to \bb R\). A localization scheme is a rule assigning a localization process to each target measure in a prescribed class.
    \label{def:localization-scheme}
\end{definition}

We use an information-theoretic construction of such processes, equivalent to the measure-valued formulation above; see also \cite{EAM22}. Let \(\tp{Y_u}_{u \ge 0}\) be a time-homogeneous Markov process on \(\Omega\), initialized with \(Y_0 \sim \mu\). We call \(Y\) the noising process and choose it so that its transition semigroup is tractable and independent of \(\mu\). Fix a continuously differentiable, strictly decreasing map \(\tau\), defined for \(0 < t \le T\) and taking values in \(\bb R_{\ge 0}\), such that
\[
\tau\tp{T} = 0, \qquad \lim_{t \to 0^+} \tau\tp{t} = +\infty,
\]
and define the denoising process by \(X_t \defeq Y_{\tau\tp{t}}\). Thus \(X_T = Y_0\) has the target law \(\mu\), while small \(t\) corresponds to a heavily noised observation. Standard stochastic localization admits analogous descriptions using Ornstein--Uhlenbeck noise, a time-reversed Brownian bridge, or additive Brownian noise \cite{STZ25}.

Let \(\+F_t\) be the natural filtration of the denoising process for \(0 < t \le T\). The corresponding localization process is the posterior law
\begin{equation}
    \nu_t \defeq \-{Law}\tp{X_T \mid \+F_t} = \-{Law}\tp{X_T \mid X_t}, \qquad t \in \tp{0, T},
    \label{eq:localization-posterior-law}
\end{equation}
where the second equality follows from the Markov property. For every bounded \(f\), the tower property shows that \(\*E_{\nu_t}\stp{f} = \*E\stp{f\tp{X_T} \mid \+F_t}\) is a martingale. Hence \(\*E\stp{\*E_{\nu_t}\stp{f}} = \*E_{\mu}\stp{f}\), and \(\nu_T = \delta_{X_T}\). If the noising process is asymptotically uninformative, namely \(\+F_{0+} \defeq \bigcap_{t > 0} \+F_t\) is trivial, then \(\nu_t\) converges weakly to \(\mu\) almost surely as \(t \to 0^+\), and we set \(\nu_0 \defeq \mu\).

We next describe the Markov chains associated with a localization scheme. Write \(\mu_t \defeq \-{Law}\tp{X_t}\). For \(0 < s \le t \le T\), define the noising and denoising kernels by
\[
Q_{t \to s}^{\-n}\tp{x, y} \defeq \*{Pr}\stp{X_s = y \mid X_t = x}, \qquad Q_{s \to t}^{\-d}\tp{y, x} \defeq \*{Pr}\stp{X_t = x \mid X_s = y}.
\]
Bayes' rule gives the adjoint relation
\[
\mu_t\tp{x} Q_{t \to s}^{\-n}\tp{x, y} = \mu_s\tp{y} Q_{s \to t}^{\-d}\tp{y, x}.
\]
For \(t \in \tp{0, T}\), the associated noising-denoising kernel on the target space is
\[
Q_{T \leftrightarrow t} \defeq Q_{T \to t}^{\-n} Q_{t \to T}^{\-d}.
\]
A transition of \(Q_{T \leftrightarrow t}\) first noises a sample from \(\mu\) to time \(t\) and then resamples from the posterior. Conditional on \(X_t\), the original and resampled targets are independent with law \(\nu_t\). The adjoint relation shows that \(Q_{T \leftrightarrow t}\) is reversible with respect to \(\mu\) and positive semidefinite on \(L^2\tp{\mu}\).

The Dirichlet form of the associated kernel has a direct posterior representation. For every \(f \in L^2\tp{\mu}\),
\begin{equation}
    \+E_{Q_{T \leftrightarrow t}}\tp{f, f} = \*E\stp{\*{Var}_{\nu_t}\stp{f}},
    \label{eq:localization-associated-dirichlet-form}
\end{equation}
while, for positive \(f\),
\begin{equation}
    \+E_{Q_{T \leftrightarrow t}}\tp{\log f, f} = \*E\stp{\*{Cov}_{\nu_t}\stp{\log f, f}} \ge \*E\stp{\*{Ent}_{\nu_t}\stp{f}}.
    \label{eq:localization-associated-entropy-form}
\end{equation}
The inequality follows from \(\*E_{\nu_t}\stp{\log f} \le \log \*E_{\nu_t}\stp{f}\). Consequently, if the bounds
\[
\*E\stp{\*{Var}_{\nu_t}\stp{f}} \ge \gamma_t \, \*{Var}_{\mu}\stp{f}, \qquad \*E\stp{\*{Ent}_{\nu_t}\stp{f}} \ge \rho_t \, \*{Ent}_{\mu}\stp{f}
\]
hold uniformly in \(f\), with \(f\) nonnegative in the entropy bound, then \(Q_{T \leftrightarrow t}\) satisfies a \Poincare inequality with constant \(\gamma_t\) and a modified log-Sobolev inequality with constant \(\rho_t\). Such bounds are referred to as approximate conservation of variance and entropy.

In applications, approximate conservation of variance and entropy is typically proved by applying \Ito calculus to the conditional variance \(\*{Var}_{\nu_t}\stp{f}\) and entropy \(\*{Ent}_{\nu_t}\stp{f}\). Spectral independence and entropic stability then provide the key estimates needed to close the resulting differential inequalities.

\subsection{Noising and Denoising Processes}
\label{subsec:noising-denoising}

In this subsection, we directly define the noising process underlying Poisson stochastic localization and establish some basic properties of the noising and denoising processes. We study the resulting localization process in \Cref{subsec:localization-process}. An intrinsic construction of Poisson stochastic localization via Poisson bridges is given in \Cref{sec:poisson-bridge}.

For Poisson stochastic localization, the target distribution \(\mu\) is a probability measure on \(\bb N^d\). We choose the noising process \(\tp{\*Y_u}_{u \ge 0}\) to be the coordinatewise binomial thinning process initialized with \(\*Y_0 \sim \mu\). This time-homogeneous Markov process has generator
\begin{equation}
    \+L^{\-n} f\tp{\*n} \defeq \sum_{i = 1}^d n_i \nabla_i^- f\tp{\*n}.
    \label{eq:noising-generator}
\end{equation}

\begin{proposition}
    The transition semigroup \(\tp{P_u^{\-n}}_{u \ge 0}\) of the noising process is given by
    \[
    P_u^{\-n}\tp{\*n, \*n'} \defeq \*{Pr}\stp{\*Y_u = \*n' \mid \*Y_0 = \*n} = \prod_{i = 1}^d \-{Bin}\tp{n_i, e^{-u}; n_i'}, \quad \forall \*n, \*n' \in \bb N^d, \, u \ge 0.
    \]
    \label{prop:noising-transition-semigroup}
\end{proposition}

\begin{proof}[Proof of \Cref{prop:noising-transition-semigroup}]
    Since the coordinates evolve independently, it suffices to show that, for each \(i \in \stp{d}\),
    \[
    \*{Pr}\stp{Y_{v + u, i} = n_i' \mid Y_{v, i} = n_i} = \-{Bin}\tp{n_i, e^{-u}; n_i'}, \quad n_i, n_i' \in \bb N, \quad u, v \ge 0.
    \]
    Couple \(\tp{Y_{u, i}}_{u \ge 0}\) to a system of independently dying particles. Conditional on \(Y_{v, i} = n_i\), let each of the \(n_i\) particles die independently at rate \(1\). Each particle survives from time \(v\) to time \(v + u\) with probability \(e^{-u}\), and the linear death dynamics of \(Y_{u, i}\) coincide with those of the number of surviving particles. Hence \(Y_{v + u, i}\) has distribution \(\-{Bin}\tp{n_i, e^{-u}}\) conditional on \(Y_{v, i} = n_i\).
\end{proof}

Set \(\tau\tp{t} \defeq -\log t\) for \(t \in \left(0, 1\right]\). Since literal time reversal changes the path convention at jump times, we work with a c\`adl\`ag version of the denoising process. Define \(\*X_t \defeq \*Y_{\tau\tp{t}-}\) for \(t \in \tp{0, 1}\) and set \(\*X_1 \defeq \*Y_0\). For every deterministic \(t \in \tp{0, 1}\), \(\*Y_{\tau\tp{t}-} = \*Y_{\tau\tp{t}}\) almost surely. Thus, this c\`adl\`ag modification has the same transition laws as the literal time reversal. The following proposition is therefore an immediate consequence of \Cref{prop:noising-transition-semigroup}.

\begin{proposition}
    For \(0 < s \le t \le 1\), the denoising process \(\tp{\*X_t}_{t \in \left(0, 1\right]}\) satisfies
    \[
    \*{Pr}\stp{\*X_s = \*n' \mid \*X_t = \*n} = \prod_{i = 1}^d \-{Bin}\tp{n_i, \frac{s}{t}; n_i'}, \quad \forall \*n, \*n' \in \bb N^d.
    \]
    \label{prop:noising-transition-semigroup-denoising}
\end{proposition}

The noising process is a death process, so its time reversal is a birth process. The next proposition identifies the time-dependent generator of this denoising process.

\begin{proposition}
    The denoising process \(\tp{\*X_t}_{t \in \left(0, 1\right]}\) is a time-inhomogeneous Markov process with generator
    \begin{equation}
        \+L_t^{\-d} f\tp{\*n} \defeq \sum_{i = 1}^d q_{t, i}\tp{\*n} \nabla_i^+ f\tp{\*n}, \quad t \in \tp{0, 1},
        \label{eq:denoising-generator}
    \end{equation}
    where
    \[
    q_{t, i}\tp{\*n} \defeq \frac{\tp{n_i + 1} \mu_t\tp{\*n + \*e_i}}{t \mu_t\tp{\*n}}, \qquad \mu_t \defeq \-{Law}\tp{\*X_t}.
    \]
    \label{prop:denoising-generator}
\end{proposition}

\begin{proof}[Proof of \Cref{prop:denoising-generator}]
    Fix \(t \in \tp{0, 1}\), \(\*n \in \supp\tp{\mu_t}\), and \(i \in \stp{d}\). For \(h > 0\) such that \(t + h \le 1\), Bayes' rule and \Cref{prop:noising-transition-semigroup} give
    \[
    \begin{aligned}
        \*{Pr}\stp{\*X_{t + h} = \*n + \*e_i \mid \*X_t = \*n}
        &= \frac{\mu_{t + h}\tp{\*n + \*e_i}}{\mu_t\tp{\*n}} \*{Pr}\stp{\*X_t = \*n \mid \*X_{t + h} = \*n + \*e_i} \\
        &= \frac{\mu_{t + h}\tp{\*n + \*e_i}}{\mu_t\tp{\*n}} \tp{n_i + 1} \tp{\frac{t}{t + h}}^{\sum_{j = 1}^d n_j} \frac{h}{t + h}.
    \end{aligned}
    \]
    Therefore,
    \[
    \lim_{h \to 0^+} \frac{1}{h} \*{Pr}\stp{\*X_{t + h} = \*n + \*e_i \mid \*X_t = \*n} = \frac{\tp{n_i + 1} \mu_t\tp{\*n + \*e_i}}{t \mu_t\tp{\*n}} = q_{t, i}\tp{\*n}.
    \]
    The denoising process cannot jump downward, and the probability of two or more births during \(\stp{t, t + h}\) is \(o\tp{h}\). Hence its only off-diagonal jump rates are \(q_{t, i}\tp{\*n}\), which yields \eqref{eq:denoising-generator}.
\end{proof}

The birth-rate process \(\tp{\*q_t\tp{\*X_t}}_{t \in \left(0, 1\right]}\) is a martingale. This property yields an equivalent conditional-expectation formula for the rates.

\begin{proposition}
    The birth rates \(\tp{\*q_t\tp{\*X_t}}_{t \in \left(0, 1\right]}\) form a martingale with respect to the natural filtration generated by the denoising process \(\tp{\*X_t}_{t \in \left(0, 1\right]}\).
    \label{prop:denoising-birth-rates-martingale}
\end{proposition}

\begin{proof}[Proof of \Cref{prop:denoising-birth-rates-martingale}]
    For \(0 < s \le t \le 1\), define the noising kernel from time \(t\) back to time \(s\) by
    \[
    Q_{t \to s}^{\-n}\tp{\*n, \*n'} \defeq \*{Pr}\stp{\*X_s = \*n' \mid \*X_t = \*n} = \prod_{i = 1}^d \-{Bin}\tp{n_i, \frac{s}{t}; n_i'}.
    \]
    For \(\*n' \le \*n\) and \(i \in \stp{d}\), the binomial identity
    \[
    \frac{n_i + 1}{t} Q_{t \to s}^{\-n}\tp{\*n, \*n'} = \frac{n_i' + 1}{s} Q_{t \to s}^{\-n}\tp{\*n + \*e_i, \*n' + \*e_i}
    \]
    holds. Thus, for every \(\*n' \in \supp\tp{\mu_s}\),
    \[
    \begin{aligned}
        \*E\stp{q_{t, i}\tp{\*X_t} \mid \*X_s = \*n'}
        &= \sum_{\*n \ge \*n'} \frac{\mu_t\tp{\*n} Q_{t \to s}^{\-n}\tp{\*n, \*n'}}{\mu_s\tp{\*n'}} \frac{\tp{n_i + 1} \mu_t\tp{\*n + \*e_i}}{t \mu_t\tp{\*n}} \\
        &= \frac{n_i' + 1}{s \mu_s\tp{\*n'}} \sum_{\*n \ge \*n'} \mu_t\tp{\*n + \*e_i} Q_{t \to s}^{\-n}\tp{\*n + \*e_i, \*n' + \*e_i} \\
        &= \frac{\tp{n_i' + 1} \mu_s\tp{\*n' + \*e_i}}{s \mu_s\tp{\*n'}} \\
        &= q_{s, i}\tp{\*n'}.
    \end{aligned}
    \]
    In the first equality, we used Bayes' rule; in the third, we used the transition formula from \(t\) back to \(s\). The Markov property therefore gives
    \[
    \*E\stp{\*q_t\tp{\*X_t} \mid \+F_s} = \*E\stp{\*q_t\tp{\*X_t} \mid \*X_s} = \*q_s\tp{\*X_s},
    \]
    where \(\tp{\+F_t}_{t \in \left(0, 1\right]}\) is the natural filtration of the denoising process.
\end{proof}

\begin{corollary}
    The birth rates of the denoising process admit the conditional-expectation representation
    \[
    \*q_t\tp{\*X_t} = \frac{1}{1 - t} \*E\stp{\*X_1 - \*X_t \mid \*X_t}, \quad t \in \tp{0, 1}.
    \]
    \label{cor:denoising-birth-rates-conditional-expectation}
\end{corollary}

\begin{proof}[Proof of \Cref{cor:denoising-birth-rates-conditional-expectation}]
    The compensated process
    \[
    \*X_u - \int_t^u \*q_s\tp{\*X_s} \dd s, \quad u \in \stp{t, 1},
    \]
    is a martingale. Hence, for \(t \in \tp{0, 1}\),
    \[
    \*E\stp{\*X_1 - \*X_t \mid \+F_t} = \int_t^1 \*E\stp{\*q_s\tp{\*X_s} \mid \+F_t} \dd s = \int_t^1 \*q_t\tp{\*X_t} \dd s = \tp{1 - t} \*q_t\tp{\*X_t},
    \]
    where the second equality follows from \Cref{prop:denoising-birth-rates-martingale}. Since the denoising process is Markov, conditioning on \(\+F_t\) can be replaced by conditioning on \(\*X_t\), which proves the claim.
\end{proof}

\subsection{Localization Process}
\label{subsec:localization-process}

Let \(\tp{\*X_t}_{t \in \left(0, 1\right]}\) be the c\`adl\`ag denoising process with target measure \(\mu\) defined in \Cref{subsec:noising-denoising}, and let \(\tp{\+F_t}_{t \in \left(0, 1\right]}\) denote the natural filtration. We assume throughout this subsection that \(\mu\) has finite first moments. The localization process for \(\mu\) is the measure-valued process
\[
\nu_t \defeq \-{Law}\tp{\*X_1 \mid \+F_t} = \-{Law}\tp{\*X_1 \mid \*X_t}, \quad t \in \left(0, 1\right],
\]
with \(\nu_0 \defeq \mu\). More explicitly, for \(t \in \left(0, 1\right]\) and \(\*x \in \supp\tp{\mu_t}\), define \(\nu_{t, \*x} \defeq \-{Law}\tp{\*X_1 \mid \*X_t = \*x}\). Then \(\nu_t = \nu_{t, \*X_t}\) for all \(t \in \left(0, 1\right]\). The measure \(\nu_{t,\*x}\) is given explicitly as follows. At \(t=1\), \(\nu_{1, \*x} = \delta_{\*x}\). For \(t \in \tp{0, 1}\), \Cref{prop:noising-transition-semigroup-denoising} and Bayes' rule yield
\begin{equation}
    \nu_{t, \*x}\tp{\*n} = \frac{\mu\tp{\*n}}{\mu_t\tp{\*x}} \prod_{i = 1}^d \binom{n_i}{x_i} t^{x_i} \tp{1 - t}^{n_i - x_i}, \quad \*n \ge \*x.
    \label{eq:posterior-formula}
\end{equation}

This explicit formula \(\nu_t = \nu_{t, \*X_t}\) also determines the path regularity of the localization process. Almost surely, the process \(\tp{\*X_t}_{t \in \tp{0, 1}}\) has finitely many jumps, and \(t \mapsto \nu_{t, \*x}\) is continuous in total variation on \(\tp{0, 1}\) for every fixed \(\*x\). Consequently, at each jump time \(t\), \(\nu_{t-} = \nu_{t, \*X_{t-}}\) and \(\nu_t = \nu_{t, \*X_t}\). Moreover, almost surely, \(\*X_t = \*0\) for all sufficiently small \(t > 0\) and \(\*X_t = \*X_1\) for all \(t < 1\) sufficiently close to \(1\). It follows that, almost surely,
\[
\lim_{t \to 0^+} \-{TV}\tp{\nu_t, \mu} = 0, \qquad \lim_{t \to 1^-} \-{TV}\tp{\nu_t, \delta_{\*X_1}} = 0.
\]
Therefore, \(\tp{\nu_t}_{t \in \stp{0, 1}}\) is c\`adl\`ag with respect to the total-variation topology.

\begin{definition}
    Poisson stochastic localization is the localization scheme that assigns \(\tp{\nu_t}_{t \in \stp{0, 1}}\) to each target measure \(\mu\).
\end{definition}

We next show that Poisson stochastic localization is a linear-tilt localization scheme in the sense of \cite[Section~2.4.5]{CE25}, a key property for establishing approximate conservation of variance and entropy. We then derive exact dissipation identities for the conditional \(\Phi\)-entropies \(\*E\stp{\*{Ent}_{\nu_t}^{\Phi}\stp{f}}\), with variance and entropy as important special cases. These results are stated in \Cref{thm:poisson-linear-tilt,thm:phi-entropy-evolution}. Their proofs follow from direct \Ito calculus and are deferred to \Cref{subsec:proof-ito-calculus}.

\begin{theorem}
    Let \(\tp{\*M_t}_{t \in \stp{0, 1}}\) be the compensated birth process defined coordinatewise by
    \[
    \dd M_{t, i} \defeq \dd X_{t, i} - q_{t, i}\tp{\*X_{t-}} \dd t.
    \]
    Define \(\tp{\*Z_t}_{t \in \stp{0, 1}}\) by \(Z_{0, i} \defeq 0\) and
    \[
    \dd Z_{t, i} \defeq
    \begin{cases}
        \displaystyle \frac{\dd M_{t, i}}{\mean{\nu_{t-}}_i - X_{t-, i}}, & \mean{\nu_{t-}}_i > X_{t-, i}, \\
        0, & \mean{\nu_{t-}}_i = X_{t-, i}.
    \end{cases}
    \]
    Then \(\*Z_t\) is a martingale on every interval \(\stp{0, T}\) with \(T < 1\), and Poisson stochastic localization satisfies the linear-tilt equation
    \[
    \dd \nu_t\tp{\*n} = \nu_{t-}\tp{\*n} \inner{\*n - \mean{\nu_{t-}}}{\dd \*Z_t}, \quad \forall \*n \in \supp\tp{\mu}, \, t \in \tp{0, 1}.
    \]
    \label{thm:poisson-linear-tilt}
\end{theorem}

\begin{theorem}
    Let \(\Phi\) be a differentiable convex function on an interval \(I\), and let \(f \in \+C_{\-{c}}\tp{\supp\tp{\mu}}\) satisfy \(f\tp{\supp\tp{\mu}} \subseteq I\). Define
    \[
    \Psi\tp{u, v} \defeq \Phi\tp{u + v} - \Phi\tp{u} - v \Phi'\tp{u}.
    \]
    Then
    \[
    \dd \*E\stp{\*{Ent}_{\nu_t}^{\Phi}\stp{f}} = -\frac{1}{1 - t} \*E\stp{\sum_{i = 1}^d \tp{\mean{\nu_t}_i - X_{t, i}} \Psi\tp{\*E_{\nu_t}\stp{f}, \frac{\*{Cov}_{\*n \sim \nu_t}\stp{f\tp{\*n}, n_i}}{\mean{\nu_t}_i - X_{t, i}}}} \dd t.
    \]
    In particular:
    \begin{itemize}
        \item If \(\Phi\tp{x} = x^2\), then
        \[
        \dd \*E\stp{\*{Var}_{\nu_t}\stp{f}} = -\frac{1}{1 - t} \*E\stp{\sum_{i = 1}^d \frac{\*{Cov}_{\*n \sim \nu_t}\stp{f\tp{\*n}, n_i}^2}{\mean{\nu_t}_i - X_{t, i}}} \dd t.
        \]
        \item If \(\Phi\tp{x} = x \log x\) and \(f \ge 0\), then
        \[
        \dd \*E\stp{\*{Ent}_{\nu_t}\stp{f}} = -\frac{1}{1 - t} \*E\stp{\*E_{\nu_t}\stp{f} \sum_{i = 1}^d \Psi\tp{\mean{\nu_t}_i - X_{t, i}, \frac{\*{Cov}_{\*n \sim \nu_t}\stp{f\tp{\*n}, n_i}}{\*E_{\nu_t}\stp{f}}}} \dd t,
        \]
        with the convention \(0 \log 0 \defeq 0\).
    \end{itemize}
    Every quotient with zero denominator is interpreted as zero; the corresponding covariance then also vanishes.
    \label{thm:phi-entropy-evolution}
\end{theorem}

\subsection{Associated Limiting Process}
\label{subsec:associated-limiting-process}

Every localization scheme naturally induces a family of Markov chains; see \cite[Section~2.2]{CE25}. Each transition consists of a noising step followed by a denoising step. We now identify these chains and their infinitesimal limit for Poisson stochastic localization.

Let \(\tp{\*X_t}_{t \in \left(0, 1\right]}\) be the denoising process with target measure \(\mu\), and let \(\tp{\nu_t}_{t \in \stp{0, 1}}\) be the corresponding localization process. For \(0 < s \le t \le 1\), write
\[
Q_{t \to s}^{\-n}\tp{\*n, \*n'} \defeq \*{Pr}\stp{\*X_s = \*n' \mid \*X_t = \*n}, \qquad Q_{s \to t}^{\-d}\tp{\*n', \*n} \defeq \*{Pr}\stp{\*X_t = \*n \mid \*X_s = \*n'}.
\]
For \(t \in \tp{0, 1}\), the associated noising-denoising kernel is
\[
Q_{1 \leftrightarrow t} \defeq Q_{1 \to t}^{\-n} Q_{t \to 1}^{\-d}.
\]
Thus, each transition first applies binomial thinning from time \(1\) to time \(t\) and then samples from the posterior at time \(t\). By construction, \(Q_{1 \leftrightarrow t}\) is reversible with respect to \(\mu\).

Define the constants
\begin{equation}
    \gamma_t \defeq \inf_{\substack{f \in \+C_{\-c}\tp{\supp\tp{\mu}} \\ \*{Var}_{\mu}\stp{f} > 0}} \frac{\*E\stp{\*{Var}_{\nu_t}\stp{f}}}{\*{Var}_{\mu}\stp{f}}, \qquad \rho_t \defeq \inf_{\substack{f \in \+C_{\-c}\tp{\supp\tp{\mu}}, \, f \ge 0 \\ \*{Ent}_{\mu}\stp{f} > 0}} \frac{\*E\stp{\*{Ent}_{\nu_t}\stp{f}}}{\*{Ent}_{\mu}\stp{f}}.
    \label{eq:approximate-conservation-constants}
\end{equation}
These constants quantify the fractions of variance and entropy retained on average by the posterior at time \(t\). The variance decomposition and its entropy analogue identify them with functional-inequality constants of \(Q_{1 \leftrightarrow t}\), as follows.

\begin{lemma}[{\cite[Proposition~19]{CE25}}]
    Let \(\gamma_t\) and \(\rho_t\) be defined as in \eqref{eq:approximate-conservation-constants}. Then \(Q_{1 \leftrightarrow t}\) satisfies the \Poincare inequality with constant \(\gamma_t\) and the modified log-Sobolev inequality with constant \(\rho_t\):
    \[
    \gamma_t \, \*{Var}_{\mu}\stp{f} \le \+E_{Q_{1 \leftrightarrow t}}\tp{f, f}, \quad \forall f \in \+C_{\-c}\tp{\supp\tp{\mu}},
    \]
    \[
    \rho_t \, \*{Ent}_{\mu}\stp{f} \le \+E_{Q_{1 \leftrightarrow t}}\tp{\log f, f}, \quad \forall f \in \+C_{\-c}\tp{\supp\tp{\mu}}, \, f \ge 0,
    \]
    and the constant \(\gamma_t\) is the optimal \Poincare constant.
    \label{lem:associated-markov-chain-functional-inequalities}
\end{lemma}

The preceding lemma is an instance of the localization framework of \cite{CE25}. Following the limiting-process viewpoint in \cite{CCCYZ25}, we next send \(t \to 1^-\) and accelerate \(Q_{1 \leftrightarrow t}\) by the factor \(1 / \tp{1 - t}\). The limit is the birth-death chain from \eqref{eq:birth-death-generator}, and the corresponding functional inequalities pass to the limit.

\begin{theorem}
    Suppose that the target distribution \(\mu\) is supported on a downward closed subset of \(\bb N^d\). For \(t \in \tp{0, 1}\), set \(\+L^{Q_{1 \leftrightarrow t}} \defeq Q_{1 \leftrightarrow t} - I\). For \(f \in \+C_{\-c}\tp{\supp\tp{\mu}}\) and \(\*n \in \supp\tp{\mu}\), define
    \[
    \+L f\tp{\*n} \defeq \lim_{\eps \to 0^+} \frac{1}{\eps} \+L^{Q_{1 \leftrightarrow 1 - \eps}} f\tp{\*n}.
    \]
    Then
    \[
    \+L f\tp{\*n} = \sum_{i = 1}^d \frac{\tp{n_i + 1} \mu\tp{\*n + \*e_i}}{\mu\tp{\*n}} \nabla_i^+ f\tp{\*n} + \sum_{i = 1}^d n_i \nabla_i^- f\tp{\*n},
    \]
    which coincides with \eqref{eq:birth-death-generator}.
    \label{thm:associated-limiting-process}
\end{theorem}

\begin{proof}[Proof of \Cref{thm:associated-limiting-process}]
    Set \(b_i\tp{\*n} \defeq \tp{n_i + 1}\mu\tp{\*n + \*e_i}/\mu\tp{\*n}\). We first compute the transition rates of \(Q_{1 \leftrightarrow 1 - \eps}\). For each fixed \(\*n \in \supp\tp{\mu}\), \Cref{prop:noising-transition-semigroup-denoising} gives
    \[
    \begin{aligned}
        Q_{1 \to 1 - \eps}^{\-n}\tp{\*n, \*n} &= 1 - \eps \norm{\*n}_1 + O_{\*n}\tp{\eps^2}, \\
        Q_{1 \to 1 - \eps}^{\-n}\tp{\*n, \*n - \*e_i} &= \eps n_i + O_{\*n}\tp{\eps^2}, \quad i \in \stp{d}, \, n_i > 0,
    \end{aligned}
    \]
    while the total probability of two or more deaths is \(O_{\*n}\tp{\eps^2}\).
    
    The reverse kernel is the posterior \(Q_{1 - \eps \to 1}^{\-d}\tp{\*x, \cdot} = \nu_{1 - \eps, \*x}\). For fixed \(\*x \in \supp\tp{\mu}\), \Cref{prop:noising-transition-semigroup-denoising} yields
    \[
    \mu_{1 - \eps}\tp{\*x} = \tp{\mu Q_{1 \to 1 - \eps}^{\-n}} \tp{\*x} = \tp{1 - \eps}^{\norm{\*x}_1} \stp{\mu\tp{\*x} + \eps \sum_{i = 1}^d \tp{x_i + 1}\mu\tp{\*x + \*e_i} + O_{\*x}\tp{\eps^2}}.
    \]
    This expansion remains valid when \(\supp\tp{\mu}\) is infinite. Specifically, for \(\*n = \*x + \*k\), we have
    \[
    Q_{1 \to 1 - \eps}^{\-n}\tp{\*n, \*x} = \prod_{i = 1}^d \binom{x_i + k_i}{x_i} \tp{1 - \eps}^{x_i} \eps^{k_i} = \tp{1 - \eps}^{\norm{\*x}_1} \eps^{\norm{\*k}_1} \prod_{i = 1}^d \binom{x_i + k_i}{x_i}.
    \]
    Note that \(\prod_i \binom{x_i + k_i}{x_i}\) grows only polynomially in \(\*k\), whereas \(\eps^{\norm{\*k}_1}\) decays geometrically. Consequently, the sum over terms with \(\norm{\*k}_1 \ge 2\) is \(O_{\*x}\tp{\eps^2}\). Substituting this into the expression for \(\nu_{1 - \eps, \*x}\) gives
    \begin{equation}
        \begin{aligned}
            Q_{1 - \eps \to 1}^{\-d}\tp{\*x, \*x} &= 1 - \eps \sum_{i = 1}^d b_i\tp{\*x} + O_{\*x}\tp{\eps^2}, \\
            Q_{1 - \eps \to 1}^{\-d}\tp{\*x, \*x + \*e_i} &= \eps b_i\tp{\*x} + O_{\*x}\tp{\eps^2}, \quad i \in \stp{d},
        \end{aligned}
        \label{eq:posterior-expansion}
    \end{equation}
    with total remaining posterior mass \(O_{\*x}\tp{\eps^2}\).
    
    Composing the two kernels, we obtain
    \[
    \begin{aligned}
        Q_{1 \leftrightarrow 1 - \eps}\tp{\*n, \*n - \*e_i} &= \eps n_i + O_{\*n}\tp{\eps^2}, \quad i \in \stp{d} \\
        Q_{1 \leftrightarrow 1 - \eps}\tp{\*n, \*n + \*e_i} &= \eps b_i\tp{\*n} + O_{\*n}\tp{\eps^2}, \quad i \in \stp{d} \\
        Q_{1 \leftrightarrow 1 - \eps}\tp{\*n, \*n} &= 1 - \eps \sum_{i = 1}^d \tp{n_i + b_i\tp{\*n}} + O_{\*n}\tp{\eps^2},
    \end{aligned}
    \]
    with total remaining mass \(O_{\*n}\tp{\eps^2}\). Hence, for every \(f \in \+C_{\-c}\tp{\supp\tp{\mu}}\) and \(\*n \in \supp\tp{\mu}\),
    \[
    \+L f\tp{\*n} = \lim_{\eps \to 0^+} \frac{\tp{Q_{1 \leftrightarrow 1 - \eps} - I}f\tp{\*n}}{\eps} = \sum_{i = 1}^d b_i\tp{\*n}\nabla_i^+ f\tp{\*n} + \sum_{i = 1}^d n_i \nabla_i^- f\tp{\*n}.
    \]
    This proves the asserted generator limit.
\end{proof}

\begin{theorem}
    Let \(\gamma_t\) and \(\rho_t\) be defined as in \eqref{eq:approximate-conservation-constants}, and suppose that \(\mu\) has downward closed support and finite first moments. Set
    \[
    \gamma \defeq \liminf_{\eps \to 0^+} \frac{\gamma_{1 - \eps}}{\eps}, \qquad \rho \defeq \liminf_{\eps \to 0^+} \frac{\rho_{1 - \eps}}{\eps}.
    \]
    Then \(\mu\) satisfies the weighted \Poincare inequality:
    \[
    \gamma \, \*{Var}_{\mu}\stp{f} \le \*E_{\*n \sim \mu}\stp{\sum_{i = 1}^d \frac{\tp{n_i + 1} \mu\tp{\*n + \*e_i}}{\mu\tp{\*n}} \tp{\nabla_i^+ f\tp{\*n}}^2} = \+E\tp{f, f}, \quad \forall f \in \+D\tp{\+E},
    \]
    and the weighted Wu-type modified log-Sobolev inequality:
    \[
    \rho \, \*{Ent}_{\mu}\stp{f} \le \*E_{\*n \sim \mu}\stp{\sum_{i = 1}^d \frac{\tp{n_i + 1} \mu\tp{\*n + \*e_i}}{\mu\tp{\*n}} \Psi\tp{f\tp{\*n}, \nabla_i^+ f\tp{\*n}}} \le \+E\tp{\log f, f}, \quad \forall f \in \+A.
    \]
    Here, \(\+E\tp{\cdot, \cdot}\) is the Dirichlet form of \(\+L\) defined as in \eqref{eq:birth-death-dirichlet-form}, and \(\Psi\tp{u, v} \defeq \tp{u + v} \log\tp{u + v} - u \log u - v \tp{\log u + 1}\) with the convention \(\Psi\tp{0, 0} = 0\). The corresponding domains are
    \[
    \+D\tp{\+E} \defeq \set{f \in L^2\tp{\mu} \cmid \+E\tp{f, f} < \infty}, \qquad \+A \defeq \set{f \in L^2\tp{\mu} \cmid f > 0, \, f \in \+D\tp{\+E}, \, \log f \in \+D\tp{\+E}}.
    \]
    \label{thm:limiting-process-functional-inequalities}
\end{theorem}

\begin{remark}
    The modified log-Sobolev inequality obtained above is stronger than the standard inequality involving \(\+E\tp{\log f, f}\). Such an inequality was first introduced by Wu \cite{Wu00}. A more detailed discussion of this class of functional inequalities is deferred to \Cref{subsec:ulc-mlsi}.
\end{remark}

\begin{proof}[Proof of \Cref{thm:limiting-process-functional-inequalities}]
    Let \(\Phi\) be differentiable and convex, and write
    \[
    \Psi_{\Phi}\tp{u, v} \defeq \Phi\tp{u + v} - \Phi\tp{u} - v \Phi'\tp{u}.
    \]
    It suffices initially to take \(f \in \+C_{\-c}\tp{\supp\tp{\mu}}\). For \(\*x \in \supp\tp{\mu}\), the posterior expansion \eqref{eq:posterior-expansion} and the differentiability of \(\Phi\) give
    \[
    \*{Ent}_{\nu_{1 - \eps, \*x}}^{\Phi}\stp{f} = \eps \sum_{i = 1}^d b_i\tp{\*x} \Psi_{\Phi}\tp{f\tp{\*x}, \nabla_i^+ f\tp{\*x}} + o_{\*x}\tp{\eps}.
    \]
    Write \(S_f \defeq \set{\*x \in \supp\tp{\mu} \cmid \exists \*n \in \supp\tp{f}, \, \*n \ge \*x}\). Note that \(S_f\) is finite, and \(\*{Ent}_{\nu_{1 - \eps, \*x}}^{\Phi}\stp{f} = 0\) for \(\*x \notin S_f\). Hence, using \(\mu_{1 - \eps}\tp{\*x} \to \mu\tp{\*x}\) and summing over \(\*x \in S_f\) yields
    \[
    \lim_{\eps \to 0^+} \frac{1}{\eps} \*E\stp{\*{Ent}_{\nu_{1 - \eps}}^{\Phi}\stp{f}} = \*E_{\*n \sim \mu}\stp{\sum_{i = 1}^d b_i\tp{\*n} \Psi_{\Phi}\tp{f\tp{\*n}, \nabla_i^+ f\tp{\*n}}}.
    \]
    
    Taking \(\Phi\tp{x} = x^2\) and using the definition of \(\gamma_{1 - \eps}\), we have
    \[
    \frac{\gamma_{1 - \eps}}{\eps} \, \*{Var}_{\mu}\stp{f} \le \frac{1}{\eps} \*E\stp{\*{Var}_{\nu_{1 - \eps}}\stp{f}}.
    \]
    Passing to the lower limit gives
    \[
    \gamma \, \*{Var}_{\mu}\stp{f} \le \*E_{\*n \sim \mu}\stp{\sum_{i = 1}^d b_i\tp{\*n} \tp{\nabla_i^+ f\tp{\*n}}^2} = \+E\tp{f, f}.
    \]
    Likewise, for \(\Phi\tp{x} = x \log x\) and nonnegative \(f \in \+C_{\-c}\tp{\supp\tp{\mu}}\),
    \[
    \rho \, \*{Ent}_{\mu}\stp{f} \le \*E_{\*n \sim \mu}\stp{\sum_{i = 1}^d b_i\tp{\*n} \Psi\tp{f\tp{\*n}, \nabla_i^+ f\tp{\*n}}}.
    \]
    By the elementary inequality \(\Psi\tp{u, v} \le v \tp{\log\tp{u + v} - \log u}\),
    \[
    \*E_{\*n \sim \mu}\stp{\sum_{i = 1}^d b_i\tp{\*n} \Psi\tp{f\tp{\*n}, \nabla_i^+ f\tp{\*n}}} \le \*E_{\*n \sim \mu}\stp{\sum_{i = 1}^d b_i\tp{\*n} \nabla_i^+ \log f\tp{\*n} \nabla_i^+ f\tp{\*n}} = \+E\tp{\log f, f}.
    \]
    
    Finally, the finite-first-moment assumption implies
    \[
    \sum_{\*n \in \supp\tp{\mu}} \mu\tp{\*n} \sum_{i = 1}^d b_i\tp{\*n} = \*E_{\mu}\stp{\sum_{i = 1}^d n_i} < \infty.
    \]
    Thus, truncation in value followed by restriction to the coordinate boxes \(S_N\) from the proof of \Cref{thm:ulc-poincare-inequality} shows that \(\+C_{\-c}\tp{\supp\tp{\mu}}\) is a core for \(\+D\tp{\+E}\), while the non-negative \(f \in \+C_{\-c}\tp{\supp\tp{\mu}}\) form a core for \(\+A\). The boundary contributions vanish by the preceding finite-conductance identity. Passing to the limit extends the two inequalities to the stated domains.
\end{proof}

\subsection{Conservation of Variance and Entropy via Spectral Independence}
\label{subsec:approximate-conservation-variance-entropy}

By \Cref{lem:associated-markov-chain-functional-inequalities} and \Cref{thm:limiting-process-functional-inequalities}, establishing functional inequalities for the noising-denoising chain \(Q_{1 \leftrightarrow t}\) and the limiting process \(\+L\) reduces to deriving lower bounds on the constants \(\gamma_t\) and \(\rho_t\) from \eqref{eq:approximate-conservation-constants}. This reduction is the ``approximate conservation of variance and entropy'' step of the localization-scheme framework; see \cite[Section~3]{CE25}. Its key ingredient is to apply \Ito calculus to the conditional variance and entropy along the localization process and derive differential inequalities of the form
\begin{equation}
    \partial_t \*E\stp{\*{Var}_{\nu_t}\stp{f}} \ge -\alpha_t \*E\stp{\*{Var}_{\nu_t}\stp{f}}, \qquad \partial_t \*E\stp{\*{Ent}_{\nu_t}\stp{f}} \ge -\alpha_t \*E\stp{\*{Ent}_{\nu_t}\stp{f}}.
    \label{eq:variance-entropy-differential-inequalities}
\end{equation}

For Poisson stochastic localization, \Cref{thm:phi-entropy-evolution} already gives the evolution equations for the conditional variance and entropy. It therefore remains to identify conditions ensuring that these evolution equations yield the differential inequalities in \eqref{eq:variance-entropy-differential-inequalities}. To this end, we introduce a notion of spectral independence adapted to Poisson stochastic localization. The resulting sufficient condition is stated in \Cref{prop:sufficient-condition-approximate-conservation} and generalizes \cite[Propositions~3.3 and~3.7]{CCCYZ25}.

Throughout this subsection, we write \(\Omega \defeq \set{\*x \in \bb N^d \cmid \exists \*x' \in \supp\tp{\mu}, \, \*x \le \*x'}\). In particular, \(\Omega = \supp\tp{\mu}\) when \(\mu\) has downward closed support.

\begin{definition}
    Let \(\mu\) be a probability measure on \(\bb N^d\). We say that \(\mu\) is \(\eta\)-spectrally independent with respect to Poisson stochastic localization if
    \[
    \cov{\mu} \preceq \eta \diag\tp{\mean{\mu}}.
    \]
    Unless otherwise specified, we refer to this property simply as spectral independence.
\end{definition}

\begin{remark}
    This definition differs slightly from the standard notion of spectral independence. In particular, when \(\mu\) is supported on a subset of \(\set{0,1}^d\), the standard definition requires
    \[
    \cov{\mu} \preceq \tp{1 + \eta} \diag\tp{\*{Var}_{\*n \sim \mu}\stp{n_i}}_{i \in \stp{d}}.
    \]
    In the \(\set{0, 1}^d\) setting, our spectral independence is slightly weaker than the standard one. This difference reflects the distinct localization schemes underlying the two notions: ours arises naturally from Poisson stochastic localization, whereas the standard notion is induced by coordinate-by-coordinate localization.
\end{remark}

For probability measures with downward closed support, spectral independence with respect to Poisson stochastic localization can be viewed as the weighted \Poincare inequality restricted to the class of linear functions.

\begin{proposition}
    Let \(\mu\) be a probability measure on \(\bb N^d\) with downward closed support. Write \(\ell_{\*v}\tp{\*x} \defeq \*v^{\top} \*x\). Then \(\mu\) is \(\eta\)-spectrally independent if and only if \(\mu\) satisfies the weighted \Poincare inequality for the class of linear functions with constant \(1 / \eta\):
    \[
    \frac{1}{\eta} \*{Var}_{\mu}\stp{\ell_{\*v}} \le \*E_{\*n \sim \mu}\stp{\sum_{i = 1}^d \frac{\tp{n_i + 1} \mu\tp{\*n + \*e_i}}{\mu\tp{\*n}} \tp{\nabla_i^+ \ell_{\*v}\tp{\*n}}^2}, \quad \forall \*v \in \bb R^d \setminus\set{\*0}.
    \]
    \label{prop:spectral-independence-weighted-poincare}
\end{proposition}

\begin{proof}[Proof of \Cref{prop:spectral-independence-weighted-poincare}]
    The Dirichlet form satisfies
    \[
    \+E\tp{f, f} = \*E_{\*n \sim \mu}\stp{\sum_{i = 1}^d \frac{\tp{n_i + 1} \mu\tp{\*n + \*e_i}}{\mu\tp{\*n}} \tp{\nabla_i^+ f\tp{\*n}}^2} = \*E_{\*n \sim \mu}\stp{\sum_{i = 1}^d n_i \tp{\nabla_i^- f\tp{\*n}}^2}.
    \]
    Thus, for any \(\*v \in \bb R^d \setminus \set{\*0}\), the weighted \Poincare inequality for \(\ell_{\*v}\) is equivalent to
    \[
    \*v^{\top} \cov{\mu} \*v = \*{Var}_{\mu}\stp{\ell_{\*v}} \le \eta \*E_{\*n \sim \mu}\stp{\sum_{i = 1}^d n_i v_i^2} = \eta \sum_{i = 1}^d \mean{\mu}_i v_i^2 = \eta \*v^{\top} \diag\tp{\mean{\mu}} \*v.
    \]
    Since this holds for every \(\*v \in \bb R^d \setminus \set{\*0}\) if and only if \(\cov{\mu} \preceq \eta \diag\tp{\mean{\mu}}\), the result follows.
\end{proof}

The following corollary follows directly from \Cref{thm:ulc-poincare-inequality}.

\begin{corollary}
    Every \(\delta\)-ULC probability measure \(\mu\) on \(\bb N^d\) is \(1/\delta\)-spectrally independent.
    \label{cor:ulc-cov-bound}
\end{corollary}

\begin{definition}
    Let \(\mu\) be a probability measure on \(\bb N^d\). For \(\*x \in \Omega\), the weighted pinning of \(\mu\) at \(\*x\) is the probability measure defined by
    \[
    \mu^{\*x}\tp{\*k} \propto \mu\tp{\*x + \*k} \prod_{i = 1}^d \binom{x_i + k_i}{x_i}, \quad \*k \in \bb N^d \cmid \*x + \*k \in \supp\tp{\mu}.
    \]
    \label{def:weighted-pinning}
\end{definition}

\begin{remark}
    The definition becomes more transparent in factorial-weighted form:
    \[
    \*k! \mu^{\*x}\tp{\*k} \propto \tp{\*x + \*k}! \mu\tp{\*x + \*k}, \quad \*k \in \bb N^d \cmid \*x + \*k \in \supp\tp{\mu}.
    \]
\end{remark}

Spectral independence for pinnings and tilts implies approximate conservation of variance and entropy.

\begin{proposition}
    Fix \(t \in \left[0, 1\right)\). Suppose that \(\tp{1 - t} * \mu^{\*x}\) is \(\eta_t\)-spectrally independent for every \(\*x \in \Omega\). Then, for every \(f \in \+C_{\-c}\tp{\supp\tp{\mu}}\),
    \[
    \partial_t \*E\stp{\*{Var}_{\nu_t}\stp{f}} \ge -\frac{\eta_t}{1 - t} \*E\stp{\*{Var}_{\nu_t}\stp{f}}.
    \]
    Moreover, suppose that \(\boldsymbol{\theta} * \mu^{\*x}\) is \(\eta\)-spectrally independent for every \(\*x \in \Omega\) and \(\boldsymbol{\theta} \in \bb R_{> 0}^d\). Then, for every nonnegative \(f \in \+C_{\-c}\tp{\supp\tp{\mu}}\),
    \[
    \partial_t \*E\stp{\*{Ent}_{\nu_t}\stp{f}} \ge -\frac{\eta}{1 - t} \*E\stp{\*{Ent}_{\nu_t}\stp{f}}, \quad \forall t \in \left[0, 1\right).
    \]
    \label{prop:sufficient-condition-approximate-conservation}
\end{proposition}

The variance inequality in \Cref{prop:sufficient-condition-approximate-conservation} follows from the Cauchy--Schwarz inequality, whereas the entropy inequality follows from a standard log-Laplace transform argument. We defer the proof to \Cref{subsec:proof-sufficient-condition-approximate-conservation}.

\subsection{Modified Log-Sobolev Inequalities for ULC Measures}
\label{subsec:ulc-mlsi}

In this subsection, we establish modified log-Sobolev inequalities for the associated limiting process \(\+L\) when the target measure \(\mu\) is ultra-log-concave. By contrast, the \Poincare inequality in \Cref{thm:ulc-poincare-inequality} holds for the broader class of \(\delta\)-ULC measures, as follows from the integrated Bakry--\Emery argument developed in \Cref{sec:bakry-emery-birth-death}. Our main result is a weighted Wu-type modified log-Sobolev inequality with constant \(1\). 

The proof proceeds in three steps. First, ULC is preserved under weighted pinnings and exponential tilts. Second, \Cref{cor:ulc-cov-bound} verifies the entropy-conservation criterion in \Cref{prop:sufficient-condition-approximate-conservation} with \(\eta = 1\). Finally, \Cref{thm:limiting-process-functional-inequalities} transfers the resulting bound to \(\+L\). This argument is specific to ULC measures: for general \(\delta\)-ULC measures with \(\delta < 1\), the additional diagonal term in \eqref{eq:ulc-condition} is not invariant under arbitrary exponential tilts.

\begin{lemma}
    ULC is preserved under weighted pinning and exponential tilting. More precisely, if \(\mu\) is ULC, then \(\mu^{\*x}\) is ULC for every \(\*x \in \supp\tp{\mu}\), and \(\boldsymbol{\theta} * \mu\) is ULC for every \(\boldsymbol{\theta} \in \bb R_{> 0}^d\).
    \label{lem:ulc-closure-pinning-exponential-tilt}
\end{lemma}

\begin{proof}[Proof of \Cref{lem:ulc-closure-pinning-exponential-tilt}]
    Write \(a\tp{\*n} \defeq \*n! \mu\tp{\*n}\). For \(\*x \in \supp\tp{\mu}\), the factorial-weighted mass function of \(\mu^{\*x}\) satisfies
    \[
    a_{\*x}\tp{\*k} \defeq \*k! \mu^{\*x}\tp{\*k} = c_{\*x} a\tp{\*x + \*k}
    \]
    for some \(c_{\*x} > 0\). The support of \(\mu^{\*x}\) is downward closed, and its active coordinates at \(\*k\) are exactly those of \(\mu\) at \(\*x + \*k\). Moreover, for every pair of active coordinates \(i, j\),
    \[
    \frac{a_{\*x}\tp{\*k} a_{\*x}\tp{\*k + \*e_i + \*e_j}}{a_{\*x}\tp{\*k + \*e_i} a_{\*x}\tp{\*k + \*e_j}} = \frac{a\tp{\*x + \*k} a\tp{\*x + \*k + \*e_i + \*e_j}}{a\tp{\*x + \*k + \*e_i} a\tp{\*x + \*k + \*e_j}}.
    \]
    Thus the local curvature matrix of \(\mu^{\*x}\) at \(\*k\) equals that of \(\mu\) at \(\*x + \*k\), so \(\mu^{\*x}\) is ULC.

    Now fix \(\boldsymbol{\theta} \in \bb R_{> 0}^d\) and set \(\wt{\mu} \defeq \boldsymbol{\theta} * \mu\). For its normalizing constant \(Z_{\boldsymbol{\theta}}\),
    \[
    \wt{a}\tp{\*n} \defeq \*n! \wt{\mu}\tp{\*n} = Z_{\boldsymbol{\theta}}^{-1} \boldsymbol{\theta}^{\*n} a\tp{\*n}, \qquad \boldsymbol{\theta}^{\*n} \defeq \prod_{i = 1}^d \theta_i^{n_i}.
    \]
    The tilt has the same support as \(\mu\), and the geometric factor cancels from each ratio in the ULC matrix:
    \[
    \frac{\wt{a}\tp{\*n} \wt{a}\tp{\*n + \*e_i + \*e_j}}{\wt{a}\tp{\*n + \*e_i} \wt{a}\tp{\*n + \*e_j}} = \frac{a\tp{\*n} a\tp{\*n + \*e_i + \*e_j}}{a\tp{\*n + \*e_i} a\tp{\*n + \*e_j}}.
    \]
    Hence \(\wt{\mu}\) has the same ULC matrix as \(\mu\) at every state and is therefore ULC.
\end{proof}

\begin{theorem}
    Every ULC probability measure \(\mu\) on \(\bb N^d\) with finite first moments satisfies the weighted Wu-type modified log-Sobolev inequality with constant \(1\):
    \[
    \*{Ent}_{\mu}\stp{f} \le \*E_{\*n \sim \mu}\stp{\sum_{i = 1}^d \frac{\tp{n_i + 1} \mu\tp{\*n + \*e_i}}{\mu\tp{\*n}} \Psi\tp{f\tp{\*n}, \nabla_i^+ f\tp{\*n}}} \le \+E\tp{\log f, f}, \quad \forall f \in \+A.
    \]
    Here, \(\Psi\tp{u, v} \defeq \tp{u + v} \log\tp{u + v} - u \log u - v \tp{\log u + 1}\), with \(\Psi\tp{0, 0} \defeq 0\), and \(\+E\) and \(\+A\) are defined in \Cref{thm:limiting-process-functional-inequalities}.
    \label{thm:ulc-modified-log-sobolev-inequality}
\end{theorem}

\begin{proof}[Proof of \Cref{thm:ulc-modified-log-sobolev-inequality}]
    Fix \(\*x \in \supp\tp{\mu}\) and \(\boldsymbol{\theta} \in \bb R_{> 0}^d\). By \Cref{lem:ulc-closure-pinning-exponential-tilt}, \(\boldsymbol{\theta} * \mu^{\*x}\) is ULC, and hence \Cref{cor:ulc-cov-bound} gives
    \[
    \cov{\boldsymbol{\theta} * \mu^{\*x}} \preceq \diag\tp{\mean{\boldsymbol{\theta} * \mu^{\*x}}}.
    \]
    Applying \Cref{prop:sufficient-condition-approximate-conservation} with \(\eta = 1\) yields
    \[
    \partial_t \*E\stp{\*{Ent}_{\nu_t}\stp{f}} \ge -\frac{1}{1 - t} \*E\stp{\*{Ent}_{\nu_t}\stp{f}}, \quad \forall t \in \left[0, 1\right).
    \]
    Integrating this differential inequality gives \(\*E\stp{\*{Ent}_{\nu_t}\stp{f}} \ge \tp{1 - t} \*{Ent}_{\mu}\stp{f}\), and therefore
    \[
    \rho_t = \inf_{\substack{f \in \+C_{\-c}\tp{\supp\tp{\mu}}, \, f \ge 0 \\ \*{Ent}_{\mu}\stp{f} > 0}} \frac{\*E\stp{\*{Ent}_{\nu_t}\stp{f}}}{\*{Ent}_{\mu}\stp{f}} \ge \exp\tp{-\int_0^t \frac{1}{1 - s} \dd s} = 1 - t.
    \]
    Consequently, \Cref{thm:limiting-process-functional-inequalities} gives the claimed inequality because
    \[
    \rho = \liminf_{\eps \to 0^+} \frac{\rho_{1 - \eps}}{\eps} \ge 1.
    \]
\end{proof}

\begin{example}[Poisson Point Processes]
    Wu's modified log-Sobolev inequality \cite{Wu00} is the Poisson-space counterpart of \Cref{thm:ulc-modified-log-sobolev-inequality}. Let \(\bb P\) be the law of a Poisson point process on a measurable space \(X\) with intensity measure \(\nu\), and define the add-one operator by \(D_z F\tp{\omega} \defeq F\tp{\omega + \delta_z} - F\tp{\omega}\). Then, for every suitable nonnegative functional \(F\),
    \[
    \*{Ent}_{\bb P}\stp{F} \le \*E_{\bb P}\stp{\int_X \Psi\tp{F, D_z F} \nu\tp{\-d z}}.
    \]
    At least informally, Wu's inequality can be recovered from \Cref{thm:ulc-modified-log-sobolev-inequality} by identifying a simple point configuration on \(X\) with its indicator in the infinite-dimensional hypercube \(\set{0,1}^X\). Under this identification, the add-one operator \(D_z\) is the continuum analogue of the discrete coordinatewise increment \(\nabla_z^+\), and the Poisson law is ULC in the corresponding infinite-dimensional sense. Applying \Cref{thm:ulc-modified-log-sobolev-inequality} therefore yields Wu's modified log-Sobolev inequality on Poisson space.
\end{example}

\begin{example}[ULC Measures on \(\bb N\)]
    In one dimension, \Cref{thm:ulc-modified-log-sobolev-inequality} specializes to
    \[
    \*{Ent}_{\mu}\stp{f} \le \*E_{n \sim \mu}\stp{\frac{\tp{n + 1} \mu\tp{n + 1}}{\mu\tp{n}} \Psi\tp{f\tp{n}, \nabla^+ f\tp{n}}}, \quad \forall f \in \+A.
    \]
    Since \(a\tp{n} \defeq n! \mu\tp{n}\) is log-concave, the birth rates
    \[
    \frac{\tp{n + 1} \mu\tp{n + 1}}{\mu\tp{n}} = \frac{a\tp{n + 1}}{a\tp{n}}, \quad n \in \supp\tp{\mu},
    \]
    are non-increasing in \(n\). In particular, they are bounded above by \(\mu\tp{1} / \mu\tp{0}\). The weighted inequality therefore implies the standard Wu-type modified log-Sobolev inequality
    \[
    \*{Ent}_{\mu}\stp{f} \le \frac{\mu\tp{1}}{\mu\tp{0}} \*E_{\mu}\stp{\Psi\tp{f, \nabla^+ f}}, \quad \forall f \in \+A.
    \]
    This recovers \cite[Theorem~1.3]{Joh17}. L\'opez-Rivera and Shenfeld subsequently improved the constant \(\mu\tp{1} / \mu\tp{0}\) to \(\abs{\log \mu\tp{0}}\) and conjectured that it could be further improved to \(\mean{\mu}\) \cite{LRS25}. We note that their conjecture is false and that \(\abs{\log \mu\tp{0}}\) is the sharp constant. Indeed, applying the inequality to \(f \defeq \delta_0\) yields 
    \[
    \*{Ent}_{\mu}\stp{f} = \mu\tp{0}\abs{\log \mu\tp{0}}, \qquad \*E_{\mu}\stp{\Psi\tp{f, \nabla^+ f}} = \mu\tp{0}.
    \] 
    Thus, any admissible constant must be at least \(\abs{\log \mu\tp{0}}\).
    \label{ex:one-dim-ulc-mlsi}
\end{example}

The proof of the weighted Wu-type modified log-Sobolev inequality in \Cref{thm:ulc-modified-log-sobolev-inequality} relies on the fact that ULC measures are \(1\)-spectrally independent, which is a direct consequence of the weighted \Poincare inequality for \(\delta\)-ULC measures established in \Cref{thm:ulc-poincare-inequality}. We note that such a weighted \Poincare inequality can also be derived within the localization-scheme framework, without invoking Bakry--\Emery theory. The key ingredient is establishing spectral independence via the trickle-down theorem. See \Cref{sec:trickle-down} for an alternative proof.

\subsection{ULC Measures and Log-Concave Probability-Generating Functions}
\label{subsec:ulc-log-concave-generating-functions}

In this subsection, we characterize ultra-log-concavity in terms of spectral independence and log-concavity of the probability-generating function.

Gurvits introduced strong log-concavity for entire functions with nonnegative coefficients \cite{Gur09a}, while completely log-concave polynomials were introduced in \cite{AOGV21}. In the homogeneous setting, both notions coincide with Lorentzian polynomials \cite{BH20}. For inhomogeneous measures on \(\bb N^d\) with downward closed support, we show that these notions are equivalent to our local-matrix condition \eqref{eq:ulc-condition} defining ULC measures, as well as \(1\)-spectral independence of the tilted pinnings. This extends \cite[Proposition~3.8]{CCCYZ25} from measures on downward closed subsets of the hypercube \(\set{0, 1}^d\) to measures on downward closed subsets of \(\bb N^d\).

\begin{definition}
    Let \(\mu\) be a probability measure on \(\bb N^d\). Its probability-generating function is the formal power series
    \[
    g_{\mu}\tp{\boldsymbol{\theta}} \defeq \sum_{\*n \in \bb N^d} \mu\tp{\*n} \boldsymbol{\theta}^{\*n}.
    \]
    When \(\mu\) has finite support, \(g_{\mu}\) is its generating polynomial. Following \cite{Gur09a}, we say that \(g_{\mu}\) is strongly log-concave if, for every \(\*x \in \bb N^d\),
    \[
    \partial^{\*x} g_{\mu} \defeq \partial_1^{x_1} \cdots \partial_d^{x_d} g_{\mu}
    \]
    is either identically zero or finite, positive, and log-concave on \(\bb R_{> 0}^d\). Following \cite{AOGV21}, we say that \(g_{\mu}\) is completely log-concave if, for every \(k \in \bb N\) and \(\*v_1, \dots, \*v_k \in \bb R_{\ge 0}^d\),
    \[
    \partial_{\*v_1} \dots \partial_{\*v_k} g_{\mu}
    \]
    is either identically zero or finite, positive, and log-concave on \(\bb R_{> 0}^d\). Here \(\partial_{\*v} \defeq \sum_{i = 1}^d v_i \partial_i\) denotes the directional derivative in the direction of \(\*v\).
\end{definition}

\begin{theorem}
    Let \(\mu\) be a probability measure on \(\bb N^d\) with downward closed support. Then the following statements are equivalent:
    \begin{enumerate}
        \item The measure \(\mu\) is ultra-log-concave.
        \item For every \(\*x \in \supp\tp{\mu}\) and \(t \in \left[0, 1\right)\), the tilted pinning \(\tp{1 - t} * \mu^{\*x}\) is \(1\)-spectrally independent.
        \item For every \(\*x \in \supp\tp{\mu}\) and \(\boldsymbol{\theta} \in \bb R_{> 0}^d\), the tilted pinning \(\boldsymbol{\theta} * \mu^{\*x}\) is \(1\)-spectrally independent.
        \item The probability-generating function \(g_{\mu}\) is strongly log-concave.
        \item The probability-generating function \(g_{\mu}\) is completely log-concave.
    \end{enumerate}
    \label{thm:ulc-generating-function-characterization}
\end{theorem}

\begin{remark}
    We prove here only the equivalence of Items~1--4. Establishing the equivalence with Item~5 requires an appropriate homogenization of \(g_{\mu}\), together with the theory of Lorentzian polynomials developed in \cite{BH20}. Since the implication \(5 \Rightarrow 4\) is immediate, it remains only to establish \(1 \Rightarrow 5\); we defer its proof to \Cref{subsec:ulc-completely-log-concavity}.
\end{remark}

\begin{proof}[Proof of \Cref{thm:ulc-generating-function-characterization} (Equivalence of Items~1--4)]
    The implication \(3 \Rightarrow 2\) is immediate. Write \(a\tp{\*n} \defeq \*n! \mu\tp{\*n}\). For \(\*x \in \supp\tp{\mu}\), define
    \[
    h_{\*x}\tp{\boldsymbol{\theta}} \defeq \partial^{\*x} g_{\mu}\tp{\boldsymbol{\theta}} = \sum_{\*k \in \bb N^d} \frac{a\tp{\*x + \*k}}{\*k!} \boldsymbol{\theta}^{\*k}.
    \]
    Write \(\mu^{\*x, \boldsymbol{\theta}} \defeq \boldsymbol{\theta} * \mu^{\*x}\). Since
    \[
    \mu^{\*x, \boldsymbol{\theta}}\tp{\*k} \propto \boldsymbol{\theta}^{\*k} \mu\tp{\*x + \*k} \prod_{i = 1}^d \binom{x_i + k_i}{x_i} \propto \frac{a\tp{\*x + \*k} \boldsymbol{\theta}^{\*k}}{\*k!},
    \]
    we have
    \[
    \mu^{\*x, \boldsymbol{\theta}}\tp{\*k} = \frac{a\tp{\*x + \*k} \boldsymbol{\theta}^{\*k}}{\*k! \, h_{\*x}\tp{\boldsymbol{\theta}}}.
    \]
    Differentiating the log-partition function gives
    \begin{equation}
        \nabla^2 \log h_{\*x}\tp{\boldsymbol{\theta}} = \diag\tp{\boldsymbol{\theta}}^{-1} \tp{\cov{\mu^{\*x, \boldsymbol{\theta}}} - \diag\tp{\mean{\mu^{\*x, \boldsymbol{\theta}}}}} \diag\tp{\boldsymbol{\theta}}^{-1}.
        \label{eq:tilted-pinning-log-hessian}
    \end{equation}

    \paragraph{\(1 \Rightarrow 3\).} Suppose that \(\mu\) is ULC. By \Cref{lem:ulc-closure-pinning-exponential-tilt}, \(\mu^{\*x, \boldsymbol{\theta}}\) is ULC for every \(\*x \in \supp\tp{\mu}\) and \(\boldsymbol{\theta} \in \bb R_{> 0}^d\). Hence, by \Cref{cor:ulc-cov-bound}, \(\mu^{\*x, \boldsymbol{\theta}}\) is \(1\)-spectrally independent.

    \paragraph{\(2 \Rightarrow 1\).} Suppose that Statement~2 holds, fix \(\*x \in \supp\tp{\mu}\), and set \(\eps \defeq 1 - t\). Applying \eqref{eq:tilted-pinning-log-hessian} with \(\boldsymbol{\theta} = \eps \*1\) shows that
    \[
    \nabla^2 \log h_{\*x}\tp{\eps \*1} \preceq \*O, \quad \forall \eps \in \tp{0, 1}.
    \]
    Letting \(\eps \to 0^+\) gives
    \[
    \nabla^2 \log h_{\*x}\tp{\*0} = \frac{\tp{a\tp{\*x + \*e_i + \*e_j}}_{i, j \in \stp{d}}}{a\tp{\*x}} - \frac{\tp{a\tp{\*x + \*e_i} a\tp{\*x + \*e_j}}_{i, j \in \stp{d}}}{a\tp{\*x}^2} \preceq \*O.
    \]
    Equivalently,
    \[
    \tp{1 - \frac{a\tp{\*x} a\tp{\*x + \*e_i + \*e_j}}{a\tp{\*x + \*e_i} a\tp{\*x + \*e_j}}}_{i, j \in \+I_{\*x}} \succeq \*O,
    \]
    where \(\+I_{\*x} \defeq \set{i \in \stp{d} \cmid \*x + \*e_i \in \supp\tp{\mu}}\) denotes the set of active coordinates at \(\*x\). This is precisely the ULC condition in \eqref{eq:ulc-condition} with \(\delta = 1\).

    \paragraph{\(3 \Leftrightarrow 4\).} If \(\*x \notin \supp\tp{\mu}\), downward closure implies that \(\partial^{\*x} g_{\mu}\) is identically zero. If \(\*x \in \supp\tp{\mu}\), then \eqref{eq:tilted-pinning-log-hessian} and the positive definiteness of \(\diag\tp{\boldsymbol{\theta}}\) show that \(h_{\*x} = \partial^{\*x} g_{\mu}\) is log-concave on \(\bb R_{> 0}^d\) if and only if
    \[
    \cov{\mu^{\*x, \boldsymbol{\theta}}} \preceq \diag\tp{\mean{\mu^{\*x, \boldsymbol{\theta}}}}, \quad \forall \boldsymbol{\theta} \in \bb R_{> 0}^d.
    \]
    This is exactly the equivalence of items~3 and~4.
\end{proof}

We emphasize that a substantial part of \Cref{thm:ulc-generating-function-characterization} does not require the support of \(\mu\) to be downward closed.\footnote{In fact, strong log-concavity of the probability-generating function forces the support to be \(M^{\natural}\)-convex.} In particular, the equivalence \(3 \Leftrightarrow 4\) continues to hold by essentially the same calculation, except that one must consider \(\*x \in \bb N^d\) for which there exists \(\*x' \in \supp\tp{\mu}\) satisfying \(\*x \le \*x'\), rather than \(\*x \in \supp\tp{\mu}\). Consequently, following the proof mechanism underlying \Cref{thm:ulc-modified-log-sobolev-inequality}:
\[
\begin{aligned}
    \text{strong log-concavity of } g_{\mu} &\Longrightarrow 1 \text{-spectral independence of tilted pinnings } \boldsymbol{\theta} * \mu^{\*x} \\
    &\Longrightarrow \text{approximate conservation of entropy} \\
    &\Longrightarrow \text{modified log-Sobolev inequality for limiting process } \+L,
\end{aligned}
\]
we obtain the following more general conclusion: for any probability measure \(\mu\) on \(\bb N^d\) with finite first moments and a strongly log-concave probability-generating function, its limiting process \(\+L\) associated with Poisson stochastic localization satisfies a modified log-Sobolev inequality with constant \(1\).

The support condition affects only the precise form of the limiting process \(\+L\). When \(\supp\tp{\mu}\) is downward closed, \(\+L\) is the birth-death process in \eqref{eq:birth-death-generator}. When \(\supp\tp{\mu} \subseteq \set{\*n \in \bb N^d \cmid \norm{\*n}_1 = N}\) is homogeneous, \(\+L\) is the continuous-time down-up walk
\[
\+L f\tp{\*n} = \sum_{i = 1}^d n_i \sum_{j = 1}^d \frac{a\tp{\*n - \*e_i + \*e_j}}{\sum_{k = 1}^d a\tp{\*n - \*e_i + \*e_k}} \tp{f\tp{\*n - \*e_i + \*e_j} - f\tp{\*n}}, \qquad a\tp{\*n} = \*n! \mu\tp{\*n}.
\]
This yields a natural extension to \(\bb N^d\) of the result of \cite{CGM21}, together with a new intrinsic proof based on Poisson stochastic localization.

\section{Poisson-Type Concentration and Maximum Entropy}
\label{sec:concentration-maximum-entropy}

In this section, we derive two consequences of the weighted Wu-type modified log-Sobolev inequality in \Cref{thm:ulc-modified-log-sobolev-inequality}: Poisson-type concentration inequalities and a maximum-entropy principle. Since the proofs rely only on this functional inequality, both results extend beyond the class of ULC measures. In particular, they extend the results of \cite{Joh17,AMM21} from one-dimensional ULC measures to the high-dimensional class \(\-{WLSI}\tp{1}\). The key ingredient is a comparison with the moment-generating function of a product Poisson law.

The class of ULC measures also enjoys several important closure properties, including closure under convolution, marginalization, coordinatewise binomial thinning, and particlewise stochastic projection. We collect these results in \Cref{sec:ulc-closure-properties}. Several other properties of one-dimensional ULC measures also admit high-dimensional extensions, including the monotone entropic law of thin numbers \cite{Yu09} (see \cite[Theorem~1]{YJ09} for a concise and complete formulation) and the concavity of entropy under thinning \cite{YJ09}. Since their proofs require only minor modifications of the one-dimensional arguments, we omit the details.

Throughout this section, let \(\mu\) be a probability measure on \(\bb N^d\) with downward closed support and finite mean vector \(\*m \defeq \mean{\mu} \in \bb R_{> 0}^d\), and let \(\*X \sim \mu\). For \(\*n \in \supp\tp{\mu}\), write
\[
b_i\tp{\*n} \defeq \frac{\tp{n_i + 1} \mu\tp{\*n + \*e_i}}{\mu\tp{\*n}}, \quad i \in \stp{d}.
\]
We write \(\mu \in \-{WLSI}\tp{1}\) if \(\mu\) satisfies the weighted Wu-type modified log-Sobolev inequality with constant \(1\):
\begin{equation}
    \*{Ent}_{\mu}\stp{f} \le \*E_{\mu}\stp{\sum_{i = 1}^d b_i \Psi\tp{f, \nabla_i^+ f}}, \quad \forall f \in \+A.
    \label{eq:weighted-wu-log-sobolev}
\end{equation}
Here, \(\Psi\tp{u, v} \defeq \tp{u + v} \log\tp{u + v} - u \log u - v \tp{\log u + 1}\), with \(\Psi\tp{0, 0} \defeq 0\), and \(\+A\) denotes the domain defined in \Cref{thm:limiting-process-functional-inequalities}. By \Cref{thm:ulc-modified-log-sobolev-inequality}, every ULC measure with finite first moments belongs to \(\-{WLSI}\tp{1}\).

\subsection{Poisson Moment-Generating Function Domination}

The following comparison of moment-generating functions extends \cite[Lemma~2.1]{AMM21} from one-dimensional ULC measures to the high-dimensional class \(\-{WLSI}\tp{1}\).

\begin{theorem}
    Suppose that \(\mu \in \-{WLSI}\tp{1}\), and let \(\*Z \sim \pi_{\*m} \defeq \bigotimes_{i = 1}^d \-{Pois}\tp{m_i}\). Then, for every \(\boldsymbol{\lambda} \in \bb R^d\),
    \[
    \*E\stp{e^{\inner{\boldsymbol{\lambda}}{\*X}}} \le \exp\tp{\sum_{i = 1}^d m_i \tp{e^{\lambda_i} - 1}} = \*E\stp{e^{\inner{\boldsymbol{\lambda}}{\*Z}}}.
    \]
    \label{thm:poisson-laplace-transform-domination}
\end{theorem}

\begin{proof}[Proof of \Cref{thm:poisson-laplace-transform-domination}]
    For \(N \in \bb N_{> 0}\), set \(X_i^{\tp{N}} \defeq \min\set{X_i, N}\), \(\*m^{\tp{N}} \defeq \*E\stp{\*X^{\tp{N}}}\), and
    \[
    M_N\tp{\boldsymbol{\lambda}} \defeq \*E\stp{e^{\inner{\boldsymbol{\lambda}}{\*X^{\tp{N}}}}}, \qquad \chi_N\tp{\boldsymbol{\lambda}} \defeq \log M_N\tp{\boldsymbol{\lambda}}.
    \]
    Define
    \[
    f_{N, \boldsymbol{\lambda}}\tp{\*n} \defeq \exp\tp{\sum_{i = 1}^d \lambda_i \min\set{n_i, N}}.
    \]
    Both \(f_{N, \boldsymbol{\lambda}}\) and \(\log f_{N, \boldsymbol{\lambda}}\) are bounded, as are their forward differences. Moreover, detailed balance and finite first moments give \(\*E_{\mu}\stp{\sum_{i = 1}^d b_i} = \sum_{i = 1}^d m_i < \infty\). Thus both functions have finite energy, and hence \(f_{N, \boldsymbol{\lambda}} \in \+A\). If \(n_i < N\), then
    \[
    \Psi\tp{f_{N, \boldsymbol{\lambda}}\tp{\*n}, \nabla_i^+ f_{N, \boldsymbol{\lambda}}\tp{\*n}} = f_{N, \boldsymbol{\lambda}}\tp{\*n} \tp{\lambda_i e^{\lambda_i} - e^{\lambda_i} + 1},
    \]
    whereas the increment vanishes when \(n_i \ge N\). Detailed balance gives
    \[
    \begin{aligned}
        \*E_{\mu}\stp{\sum_{i = 1}^d b_i \Psi\tp{f_{N, \boldsymbol{\lambda}}, \nabla_i^+ f_{N, \boldsymbol{\lambda}}}} &= \*E\stp{\sum_{i = 1}^d b_i\tp{\*X} \tp{\lambda_i e^{\lambda_i} - e^{\lambda_i} + 1} f_{N, \boldsymbol{\lambda}}\tp{\*X} \*1_{\set{X_i < N}}} \\
        &= \*E\stp{\sum_{i = 1}^d \tp{\lambda_i e^{\lambda_i} - e^{\lambda_i} + 1} X_i f_{N, \boldsymbol{\lambda}}\tp{\*X - \*e_i} \*1_{\set{X_i \le N}}} \\
        &= \*E\stp{\sum_{i = 1}^d \tp{\lambda_i - 1 + e^{-\lambda_i}} X_i f_{N, \boldsymbol{\lambda}}\tp{\*X} \*1_{\set{X_i \le N}}} \\
        &\le \*E\stp{\sum_{i = 1}^d \tp{\lambda_i - 1 + e^{-\lambda_i}} X_i^{\tp{N}} f_{N, \boldsymbol{\lambda}}\tp{\*X}}.
    \end{aligned}
    \]
    Here, the last inequality uses \(\lambda_i - 1 + e^{-\lambda_i} \ge 0\). Moreover,
    \[
    \*{Ent}_{\mu}\stp{f_{N, \boldsymbol{\lambda}}} = \inner{\boldsymbol{\lambda}}{\*E\stp{\*X^{\tp{N}} f_{N, \boldsymbol{\lambda}}\tp{\*X}}} - M_N\tp{\boldsymbol{\lambda}} \chi_N\tp{\boldsymbol{\lambda}}, \qquad \*E\stp{X_i^{\tp{N}} f_{N, \boldsymbol{\lambda}}\tp{\*X}} = \partial_i M_N\tp{\boldsymbol{\lambda}}.
    \]
    Hence, substituting \(f_{N, \boldsymbol{\lambda}}\) into \eqref{eq:weighted-wu-log-sobolev} and dividing by \(M_N\tp{\boldsymbol{\lambda}}\) yields
    \[
    \inner{\boldsymbol{\lambda}}{\nabla \chi_N\tp{\boldsymbol{\lambda}}} - \chi_N\tp{\boldsymbol{\lambda}} \le \sum_{i = 1}^d \tp{\lambda_i - 1 + e^{-\lambda_i}} \partial_i \chi_N\tp{\boldsymbol{\lambda}}.
    \]
    Equivalently,
    \begin{equation}
        \sum_{i = 1}^d \tp{1 - e^{-\lambda_i}} \partial_i \chi_N\tp{\boldsymbol{\lambda}} \le \chi_N\tp{\boldsymbol{\lambda}}.
        \label{eq:poisson-laplace-differential-inequality}
    \end{equation}

    Fix \(\boldsymbol{\lambda} \in \bb R^d\). For \(s \le 0\), define the coordinatewise path from \(\*0\) to \(\boldsymbol{\lambda}\) by
    \[
    \lambda_i\tp{s} \defeq \log\tp{1 + e^s \tp{e^{\lambda_i} - 1}}.
    \]
    Then \(\lambda_i\tp{0} = \lambda_i\), \(\lambda_i\tp{s} \to 0\) as \(s \to -\infty\), and \(\partial_s \lambda_i\tp{s} = 1 - e^{-\lambda_i\tp{s}}\). Hence \eqref{eq:poisson-laplace-differential-inequality} implies
    \[
    \partial_s \tp{e^{-s} \chi_N\tp{\boldsymbol{\lambda}\tp{s}}} \le 0.
    \]
    Moreover, differentiability of \(\chi_N\) at \(\*0\) gives
    \[
    \lim_{s \to -\infty} e^{-s} \chi_N\tp{\boldsymbol{\lambda}\tp{s}} = \sum_{i = 1}^d \partial_i \chi_N\tp{\*0} \tp{e^{\lambda_i} - 1} = \sum_{i = 1}^d m_i^{\tp{N}} \tp{e^{\lambda_i} - 1}.
    \]
    Monotonicity along this path therefore yields
    \[
    \chi_N\tp{\boldsymbol{\lambda}} \le \sum_{i = 1}^d m_i^{\tp{N}} \tp{e^{\lambda_i} - 1}.
    \]
    Finally, \(\*X^{\tp{N}} \to \*X\) pointwise and \(\*m^{\tp{N}} \to \*m\). Fatou's lemma gives
    \[
    \*E\stp{e^{\inner{\boldsymbol{\lambda}}{\*X}}} \le \liminf_{N \to \infty} \exp\tp{\chi_N\tp{\boldsymbol{\lambda}}} \le \exp\tp{\sum_{i = 1}^d m_i \tp{e^{\lambda_i} - 1}} = \*E\stp{e^{\inner{\boldsymbol{\lambda}}{\*Z}}},
    \]
    proving the claim.
\end{proof}

\subsection{Poisson-Type Concentration Inequalities}
\label{subsec:concentration}

The moment-generating-function bound in \Cref{thm:poisson-laplace-transform-domination} yields mixed upper- and lower-tail estimates through a multivariate Chernoff argument. For one-dimensional ULC measures, the resulting Poisson-type tail bounds recover \cite[Theorem~1.1]{AMM21}.

Define the Bennett function by
\[
h\tp{x} \defeq 2 \frac{\tp{1 + x} \log\tp{1 + x} - x}{x^2}, \quad x \ge -1,
\]
with \(h\tp{0} \defeq 1\) and the convention \(0 \log 0 \defeq 0\).

\begin{theorem}
    Suppose that \(\mu \in \-{WLSI}\tp{1}\) and \(\*m \defeq \mean{\mu} \in \bb R_{> 0}^d\). For \(\*t \in \bb R^d\) satisfying \(t_i > -m_i\) for every \(i \in \stp{d}\), define the orthant event
    \[
    \+O_{\*t} \defeq \bigcap_{\substack{i \in \stp{d} \\ t_i > 0}} \set{X_i - m_i \ge t_i} \cap \bigcap_{\substack{i \in \stp{d} \\ t_i < 0}} \set{X_i - m_i \le t_i}.
    \]
    Then
    \[
    \*{Pr}_{\mu}\stp{\+O_{\*t}} \le \exp\tp{-\sum_{i = 1}^d \frac{t_i^2}{2 m_i} h\tp{\frac{t_i}{m_i}}}.
    \]
    \label{thm:poisson-orthant-concentration}
\end{theorem}

\begin{proof}[Proof of \Cref{thm:poisson-orthant-concentration}]
    \Cref{thm:poisson-laplace-transform-domination} gives the following bound on the centered log moment-generating function:
    \begin{equation}
        \log \*E\stp{e^{\inner{\boldsymbol{\lambda}}{\*X - \*m}}} \le \sum_{i = 1}^d m_i \tp{e^{\lambda_i} - \lambda_i - 1} = \log \*E\stp{e^{\inner{\boldsymbol{\lambda}}{\*Z - \*m}}}.
        \label{eq:centered-log-mgf-bound}
    \end{equation}
    Set
    \[
    \lambda_i \defeq \log\tp{1 + \frac{t_i}{m_i}}.
    \]
    Each \(\lambda_i\) has the same sign as \(t_i\). Consequently, on \(\+O_{\*t}\),
    \[
    \inner{\boldsymbol{\lambda}}{\*X - \*m} \ge \inner{\boldsymbol{\lambda}}{\*t}.
    \]
    Markov's inequality and \eqref{eq:centered-log-mgf-bound} therefore give
    \[
    \begin{aligned}
        \*{Pr}_{\mu}\stp{\+O_{\*t}} &\le \exp\tp{-\inner{\boldsymbol{\lambda}}{\*t}} \*E\stp{e^{\inner{\boldsymbol{\lambda}}{\*X - \*m}}} \\
        &\le \exp\tp{-\sum_{i = 1}^d \stp{\tp{m_i + t_i} \log\tp{1 + \frac{t_i}{m_i}} - t_i}} \\
        &= \exp\tp{-\sum_{i = 1}^d \frac{t_i^2}{2 m_i} h\tp{\frac{t_i}{m_i}}}.
    \end{aligned}
    \]
\end{proof}

\subsection{Poisson Maximum-Entropy Principle}
\label{subsec:maximum-entropy}

For a probability measure \(\mu\) on \(\bb N^d\), define its Shannon entropy by
\[
H\tp{\mu} \defeq -\sum_{\*n \in \supp\tp{\mu}} \mu\tp{\*n} \log \mu\tp{\*n}.
\]
The one-dimensional Poisson maximum-entropy principle for ULC measures was established in \cite{Joh07}; see also \cite{JKM13,Yu10}. The comparison of moment-generating functions in \Cref{thm:poisson-laplace-transform-domination} yields the following high-dimensional extension to \(\-{WLSI}\tp{1}\), together with a relative-entropy remainder.

\begin{theorem}
    Suppose that \(\mu \in \-{WLSI}\tp{1}\) has mean \(\*m \in \bb R_{> 0}^d\), and let \(\pi_{\*m} \defeq \bigotimes_{i = 1}^d \-{Pois}\tp{m_i}\). Then
    \begin{equation}
        H\tp{\pi_{\*m}} - H\tp{\mu} \ge \-{KL}\tp{\mu \| \pi_{\*m}} \ge 0.
        \label{eq:poisson-maximum-entropy-stability}
    \end{equation}
    Consequently, \(\pi_{\*m}\) is the unique maximizer of Shannon entropy among all measures in \(\-{WLSI}\tp{1}\) with mean \(\*m\).
    \label{thm:poisson-maximum-entropy}
\end{theorem}

\begin{proof}[Proof of \Cref{thm:poisson-maximum-entropy}]
    Let \(\*Z \sim \pi_{\*m}\). Applying \Cref{thm:poisson-laplace-transform-domination} with \(\boldsymbol{\lambda} = -s \*e_i\) gives
    \begin{equation}
        \*E\stp{e^{-s X_i}} \le \*E\stp{e^{-s Z_i}}, \quad \forall s \ge 0, \, i \in \stp{d}.
        \label{eq:poisson-coordinate-laplace-order}
    \end{equation}
    For every \(n \in \bb N\), Frullani's integral identity gives
    \begin{equation}
        \log n! = \int_0^{+\infty} \frac{e^{-s}}{s} \tp{n - \frac{1 - e^{-s n}}{1 - e^{-s}}} \dd s.
        \label{eq:frullani-log-factorial}
    \end{equation}
    \Cref{thm:poisson-laplace-transform-domination} implies that \(\*X\) has finite exponential moments. Consequently, \(H\tp{\mu} < \infty\) and \(\*E\stp{\log X_i!} < \infty\) for every \(i \in \stp{d}\); the analogous quantities for \(\*Z\) are finite as well. Taking expectations in \eqref{eq:frullani-log-factorial}, using Tonelli's theorem, and subtracting the resulting identities for \(X_i\) and \(Z_i\), whose means both equal \(m_i\), gives
    \begin{equation}
        \*E\stp{\log Z_i!} - \*E\stp{\log X_i!} = \int_0^{+\infty} \frac{e^{-s}}{s \tp{1 - e^{-s}}} \tp{\*E\stp{e^{-s Z_i}} - \*E\stp{e^{-s X_i}}} \dd s \ge 0,
        \label{eq:factorial-moment-comparison}
    \end{equation}
    where the inequality follows from \eqref{eq:poisson-coordinate-laplace-order}.

    The product Poisson mass function is
    \[
    \pi_{\*m}\tp{\*n} = e^{-\sum_{i = 1}^d m_i} \prod_{i = 1}^d \frac{m_i^{n_i}}{n_i!}, \quad \*n \in \bb N^d.
    \]
    Since \(\mean{\mu} = \*m\), a direct computation gives
    \[
    \-{KL}\tp{\mu \| \pi_{\*m}} = -H\tp{\mu} + \sum_{i = 1}^d \tp{m_i - m_i \log m_i} + \sum_{i = 1}^d \*E\stp{\log X_i!},
    \]
    whereas
    \[
    H\tp{\pi_{\*m}} = \sum_{i = 1}^d \tp{m_i - m_i \log m_i} + \sum_{i = 1}^d \*E\stp{\log Z_i!}.
    \]
    Subtracting these identities and applying \eqref{eq:factorial-moment-comparison} gives
    \[
    H\tp{\pi_{\*m}} - H\tp{\mu} = \-{KL}\tp{\mu \| \pi_{\*m}} + \sum_{i = 1}^d \tp{\*E\stp{\log Z_i!} - \*E\stp{\log X_i!}} \ge \-{KL}\tp{\mu \| \pi_{\*m}}.
    \]
    This proves \eqref{eq:poisson-maximum-entropy-stability}. The product Poisson law is ULC and therefore belongs to \(\-{WLSI}\tp{1}\) by \Cref{thm:ulc-modified-log-sobolev-inequality}. Finally, if \(H\tp{\mu} = H\tp{\pi_{\*m}}\), then \eqref{eq:poisson-maximum-entropy-stability} forces \(\-{KL}\tp{\mu \| \pi_{\*m}} = 0\), and hence \(\mu = \pi_{\*m}\), proving uniqueness.
\end{proof}

\section{Examples}
\label{sec:applications}

In this section, we establish ultra-log-concavity for polymatroid and Potts models and derive the resulting algorithmic consequences.

Our notion of \(\delta\)-ULC extends the criterion in \cite[Theorem~1.9]{CCCYZ25} from distributions on \(\set{0, 1}^d\) to distributions on \(\bb N^d\). Consequently, all of the applications in \cite[Section~5]{CCCYZ25} also fit our framework, and we do not repeat them here.

\subsection{Weighted Independent Vectors in Polymatroids}
\label{subsec:polymatroid}

We first recall the rank-function description of a discrete polymatroid.

\begin{definition}[Discrete Polymatroids]
    An integral polymatroid rank function on \(\stp{d}\) is a map \(r:2^{\stp{d}} \to \bb N\) satisfying the following properties:
    \begin{itemize}
        \item Normalization: \(r\tp{\varnothing} = 0\).
        \item Monotonicity: \(r\tp{S} \le r\tp{T}\) whenever \(S \subseteq T \subseteq \stp{d}\).
        \item Submodularity: \(r\tp{S} + r\tp{T} \ge r\tp{S \cup T} + r\tp{S \cap T}\) for all \(S, T \subseteq \stp{d}\).
    \end{itemize}
    For \(\*n \in \bb N^d\) and \(S \subseteq \stp{d}\), write \(\*n\tp{S} \defeq \sum_{i \in S} n_i\). The discrete polymatroid associated with \(r\) is the set of independent vectors
    \[
    P_r \defeq \set{\*n \in \bb N^d \cmid \forall S \subseteq \stp{d}, \, \*n\tp{S} \le r\tp{S}}.
    \]
\end{definition}

The set \(P_r\) is finite and downward closed. Given an activity vector \(\boldsymbol{\lambda} \in \bb R_{> 0}^d\), define the Poisson-weighted measure on \(P_r\) by
\[
\mu_{r, \boldsymbol{\lambda}}\tp{\*n} \propto \frac{\boldsymbol{\lambda}^{\*n}}{\*n!}, \quad \*n \in P_r.
\]
Equivalently, \(\mu_{r, \boldsymbol{\lambda}}\) is the law of a vector whose coordinates are independent Poisson random variables with means \(\lambda_1, \dots, \lambda_d\), conditioned to lie in \(P_r\). If \(r\) is the rank function of a matroid, then \(P_r \subseteq \set{0, 1}^d\), and this construction recovers the weighted independent-set measure considered in \cite[Section~5.1]{CCCYZ25}.

\begin{theorem}
    For every integral polymatroid rank function \(r\) and every activity vector \(\boldsymbol{\lambda} \in \bb R_{> 0}^d\), the measure \(\mu_{r, \boldsymbol{\lambda}}\) is ULC.
    \label{thm:polymatroid-independent-vectors-ulc}
\end{theorem}

\begin{remark}
    Lorentzian polynomials provide an alternative proof of \Cref{thm:polymatroid-independent-vectors-ulc}; see \cite{BH20}. Indeed, it suffices to observe that \(P_r\) is \(M^{\natural}\)-convex and then apply a suitable homogenization.
\end{remark}

\begin{proof}[Proof of \Cref{thm:polymatroid-independent-vectors-ulc}]
    Write \(\mu \defeq \mu_{r, \boldsymbol{\lambda}}\) and \(a\tp{\*n} \defeq \*n! \mu\tp{\*n}\). Then
    \[
    a\tp{\*n} \propto \boldsymbol{\lambda}^{\*n}, \quad \*n \in P_r.
    \]
    Fix \(\*n \in P_r\), and let \(\+I_{\*n} \defeq \set{k \in \stp{d} \cmid \*n + \*e_k \in P_r}\) be its set of active coordinates. For every \(i, j \in \+I_{\*n}\),
    \[
    \frac{a\tp{\*n} a\tp{\*n + \*e_i + \*e_j}}{a\tp{\*n + \*e_i} a\tp{\*n + \*e_j}} = \*1_{\set{\*n + \*e_i + \*e_j \in P_r}}.
    \]
    Consequently, for \(\delta = 1\), the matrix on the left-hand side of \eqref{eq:ulc-condition} reduces to
    \[
    \*M_{\*n} \defeq \tp{\*1_{\set{\*n + \*e_i + \*e_j \notin P_r}}}_{i, j \in \+I_{\*n}}.
    \]
    It remains to show that \(\*M_{\*n}\) is positive semidefinite.

    For \(S \subseteq \stp{d}\), define the slack
    \[
    s_{\*n}\tp{S} \defeq r\tp{S} - \*n\tp{S}.
    \]
    Since \(r\) is submodular and \(\*n\tp{\cdot}\) is modular, \(s_{\*n}\) is a nonnegative, integer-valued submodular function. Moreover,
    \[
    i \in \+I_{\*n} \iff \*n + \*e_i \in P_r \iff s_{\*n}\tp{S} \ge 1 \text{ for every } S \subseteq \stp{d} \text{ containing } i,
    \]
    while, for \(i, j \in \+I_{\*n}\),
    \[
    \tp{\*M_{\*n}}_{ij} = 1 \iff \*n + \*e_i + \*e_j \notin P_r \iff s_{\*n}\tp{S} = 1 \text{ for some } S \subseteq \stp{d} \text{ containing } i, j.
    \]

    Let \(B_{\*n} \defeq \set{i \in \+I_{\*n} \cmid \tp{\*M_{\*n}}_{ii} = 1}\), and define a relation on \(B_{\*n}\) by setting \(i \sim j\) whenever \(\tp{\*M_{\*n}}_{ij} = 1\). This relation is reflexive and symmetric. To prove transitivity, suppose that \(i \sim j\) and \(j \sim k\). Choose \(S, T \subseteq \stp{d}\) such that \(i, j \in S\), \(j, k \in T\), and \(s_{\*n}\tp{S} = s_{\*n}\tp{T} = 1\). Since \(j\) is active and belongs to \(S \cap T\), while \(i\) is active and belongs to \(S \cup T\), the preceding characterization, together with the submodularity of \(s_{\*n}\), gives
    \[
    2 = s_{\*n}\tp{S} + s_{\*n}\tp{T} \ge s_{\*n}\tp{S \cap T} + s_{\*n}\tp{S \cup T} \ge 2.
    \]
    Both terms on the right are at least \(1\), so equality holds throughout and \(s_{\*n}\tp{S \cup T} = 1\). Therefore, \(i \sim k\), and \(\sim\) is an equivalence relation on \(B_{\*n}\). Moreover, for any \(i \in \+I_{\*n} \setminus B_{\*n}\), \(\*n + 2 \*e_i \in P_r\), so \(s_{\*n}\tp{S} \ge 2\) for every \(S \subseteq \stp{d}\) containing \(i\). Thus \(\tp{\*M_{\*n}}_{ij} = \tp{\*M_{\*n}}_{ji} = 0\) for every \(j \in \+I_{\*n}\).

    Therefore, after reordering the coordinates, \(\*M_{\*n}\) is block diagonal: each equivalence class of \(\sim\) contributes an all-ones block, while the coordinates outside \(B_{\*n}\) contribute a zero block. Every such block is positive semidefinite, so \(\*M_{\*n} \succeq \*O\). Since this holds for every \(\*n \in P_r\), the measure \(\mu_{r, \boldsymbol{\lambda}}\) is ULC.
\end{proof}

\begin{corollary}
    The probability-generating function of \(\mu_{r, \boldsymbol{\lambda}}\) is completely log-concave.
\end{corollary}

The ULC conclusion has an immediate algorithmic consequence. By \Cref{thm:ulc-modified-log-sobolev-inequality}, the linear-death chain for \(\mu_{r, \boldsymbol{\lambda}}\) satisfies a modified log-Sobolev inequality with constant \(1\). Uniformization followed by lazification yields the following sampler and mixing bound.

\begin{corollary}
    Set \(R \defeq r\tp{\stp{d}}\), \(\Lambda \defeq \sum_{i = 1}^d \lambda_i\), \(C \defeq R + \Lambda\). Define a Markov kernel \(Q\) on \(P_r\) by
    \[
    Q\tp{\*n, \*n - \*e_i} \defeq \frac{n_i}{C}, \quad i \in \stp{d}, \, n_i > 0, \qquad Q\tp{\*n, \*n + \*e_i} \defeq \frac{\lambda_i}{C}, \quad i \in \+I_{\*n},
    \]
    with the remaining probability assigned to staying at \(\*n\), where \(\+I_{\*n} \defeq \set{i \in \stp{d} \cmid \*n + \*e_i \in P_r}\). Then \(Q\) is reversible with respect to \(\mu_{r, \boldsymbol{\lambda}}\), and its lazification \(\bar Q \defeq \tp{I + Q} / 2\) satisfies a modified log-Sobolev inequality with constant at least \(1 / \tp{2 C}\). Consequently, for every \(\eps \in \tp{0, 1 / 2}\),
    \[
    T_{\-{mix}}\tp{\eps; \bar Q} = O_{\boldsymbol{\lambda}}\tp{\tp{d + R} \tp{\log R + \log \log\tp{d + R} + \log \frac{1}{\eps}}}.
    \]
    Moreover, this chain can be implemented directly given a membership oracle for \(P_r\).
    \label{cor:polymatroid-birth-death-mixing}
\end{corollary}

\subsection{Antiferromagnetic \texorpdfstring{\(q\)}{q}-Potts Model}
\label{subsec:q-potts}

Let \(G = \tp{V, E}\) be a graph, and fix \(q \ge 3\). For a configuration \(\sigma \in \stp{q}^V\), let
\[
m_G\tp{\sigma} \defeq \abs{\set{vw \in E \cmid \sigma_v = \sigma_w}}
\]
be the number of monochromatic edges. The \(q\)-state Potts measure with interaction parameter \(B > 0\) is
\[
\mu_{G, B}\tp{\sigma} \propto B^{m_G\tp{\sigma}}, \quad \sigma \in \stp{q}^V.
\]
The model is ferromagnetic for \(B > 1\) and antiferromagnetic for \(0 < B < 1\).

To place this measure within the downward closed framework, we designate color \(q\) as a reference color. Rather than using \(q\) one-hot coordinates at each vertex, we represent color \(q\) by \(\*0 \in \bb N^{q - 1}\) and color \(k \in \stp{q - 1}\) by \(\*e_k\). Equivalently, define \(\varphi: \stp{q}^V \to \bb N^{V \times \stp{q - 1}}\) by
\begin{equation}
    \varphi\tp{\sigma}_{v, k} \defeq \*1_{\set{\sigma_v = k}}, \quad v \in V, \, k \in \stp{q - 1}.
    \label{eq:downward-closed-potts-encoding}
\end{equation}
Its image is
\[
\Omega \defeq \set{\*x \in \set{0, 1}^{V \times \stp{q - 1}} \cmid \sum_{k = 1}^{q - 1} x_{v, k} \le 1, \quad \forall v \in V},
\]
which is downward closed. We henceforth identify \(\mu_{G, B}\) with this pushforward and regard it as a probability measure on \(\Omega \subseteq \bb N^{V \times \stp{q - 1}}\).

\begin{theorem}
    Let \(G\) be a graph with maximum degree at most \(\Delta\), where \(\Delta \ge 1\), and let \(\*A_G\) be its adjacency matrix. Set \(\lambda^{\star} \defeq -\lambda_{\min}\tp{\*A_G}\), and suppose that \(q \ge 3\), \(0 < B < 1\), and
    \[
    \lambda^{\star} \tp{1 - B} B^{1 - \Delta} < 1.
    \]
    Then the embedded antiferromagnetic Potts measure \(\mu_{G, B}\) is \(\delta\)-ULC with
    \[
    \delta \defeq 1 - \lambda^{\star} \tp{1 - B} B^{1 - \Delta}.
    \]
    \label{thm:antiferromagnetic-potts-ulc}
\end{theorem}

\begin{proof}[Proof of \Cref{thm:antiferromagnetic-potts-ulc}]
    Set \(p \defeq q - 1\). Fix \(\*x \in \Omega\), write \(\sigma \defeq \varphi^{-1}\tp{\*x}\), and let \(U \defeq \set{v \in V \cmid \sigma_v = q}\). The active coordinates at \(\*x\) are precisely \(U \times \stp{p}\). If \(U = \varnothing\), the ULC condition is vacuous, so assume that \(U \ne \varnothing\). Since \(\Omega \subseteq \set{0, 1}^{V \times \stp{p}}\), the factorial-weighted mass satisfies \(a\tp{\*x} \defeq \*x! \mu_{G, B}\tp{\*x} = \mu_{G, B}\tp{\*x}\). Let \(\*M_{\*x}\) denote the first matrix on the left-hand side of \eqref{eq:ulc-condition}. A direct computation gives
    \[
    \tp{\*M_{\*x}}_{\tp{v, k}, \tp{w, \ell}} = \begin{cases}
        1, & v = w, \\
        1 - B^2, & v \ne w, \, vw \in E, \, k = \ell, \\
        1 - B, & v \ne w, \, vw \in E, \, k \ne \ell, \\
        0, & v \ne w, \, vw \notin E.
    \end{cases}
    \]
    Let \(\*A_U\) be the adjacency matrix of the induced subgraph \(G\stp{U}\), let \(\*I_U\) be the identity matrix indexed by \(U\), and let \(\*J_p\) and \(\*I_p\) be the \(p \times p\) all-ones and identity matrices, respectively. Grouping the active coordinates by vertex gives
    \[
    \*M_{\*x} = \*I_U \otimes \*J_p + \tp{1 - B} \*A_U \otimes \tp{\*J_p + B \*I_p}.
    \]
    Diagonalizing the matrices acting on the color coordinates yields
    \[
    \lambda_{\min}\tp{\*M_{\*x}} = B \tp{1 - B} \lambda_{\min}\tp{\*A_U} + \min\set{0, \, p \tp{1 + \tp{1 - B} \lambda_{\min}\tp{\*A_U}}}.
    \]
    By eigenvalue interlacing, \(\lambda_{\min}\tp{\*A_U} \ge \lambda_{\min}\tp{\*A_G} = -\lambda^{\star}\). Moreover, since \(B^{1 - \Delta} \ge 1\), the hypothesis also implies \(\lambda^{\star} \tp{1 - B} < 1\). Hence,
    \[
    \lambda_{\min}\tp{\*M_{\*x}} \ge -B \tp{1 - B} \lambda^{\star} + \min\set{0, \, p \tp{1 - \tp{1 - B} \lambda^{\star}}} = -B \tp{1 - B} \lambda^{\star}.
    \]
    Let \(\*D_{\*x}\) denote the diagonal matrix in \eqref{eq:ulc-condition}. For \(c \in \stp{q}\), let \(d_c^{\sigma}\tp{v}\) be the number of neighbors of \(v\) having color \(c\) under \(\sigma\). Let \(\*e_{v, k}\) denote the standard basis vector corresponding to \(\tp{v, k}\). Switching an active vertex \(v \in U\) from color \(q\) to color \(k \in \stp{p}\) gives
    \[
    \tp{\*D_{\*x}}_{\tp{v, k}, \tp{v, k}} = \frac{a\tp{\*x}}{a\tp{\*x + \*e_{v, k}}} = B^{d_q^{\sigma}\tp{v} - d_k^{\sigma}\tp{v}} \ge B^{\Delta},
    \]
    where the inequality follows from \(d_q^{\sigma}\tp{v} - d_k^{\sigma}\tp{v} \le \deg_G\tp{v} \le \Delta\). Thus \(\lambda_{\min}\tp{\*D_{\*x}} \ge B^{\Delta}\), and hence
    \[
    \*M_{\*x} + \tp{1 - \delta} \*D_{\*x} \succeq \tp{-B \tp{1 - B} \lambda^{\star} + \tp{1 - \delta} B^{\Delta}} \*I_{U \times \stp{p}} = \*O,
    \]
    where the equality follows from the definition of \(\delta\). Hence the ULC condition holds at every \(\*x \in \Omega\).
\end{proof}

Set \(n \defeq \abs{V}\). For \(v \in V\) and \(c \in \stp{q}\), let \(d_c^{\sigma}\tp{v}\) be the number of neighbors of \(v\) having color \(c\), and let \(\sigma^{v \gets c}\) be the configuration obtained from \(\sigma\) by assigning color \(c\) to \(v\). The transition kernel of the single-site heat-bath Glauber dynamics is
\[
P_{\-{GD}}\tp{\sigma, \sigma^{v \gets c}} \defeq \frac{1}{n} \frac{B^{d_c^{\sigma}\tp{v}}}{\sum_{a = 1}^q B^{d_{a}^{\sigma}\tp{v}}}, \quad v \in V, \, c \in \stp{q}.
\]
A direct comparison between the linear-death chain \eqref{eq:birth-death-generator} and \(P_{\-{GD}}\) yields the following result. The proofs of \Cref{cor:q-potts-rapid-mixing,cor:q-potts-random-regular-mixing} are deferred to \Cref{subsec:proofs-antiferromagnetic-potts}.

\begin{corollary}
    In the setting of \Cref{thm:antiferromagnetic-potts-ulc}, \(P_{\-{GD}}\) has spectral gap at least
    \[
    \gamma_{\-{GD}} \ge \frac{\delta B^{\Delta}}{q n}.
    \]
    \label{cor:q-potts-rapid-mixing}
\end{corollary}

Thus, for fixed \(q\), \(P_{\-{GD}}\) is a polynomial-time sampler whenever \(\delta^{-1}\) and \(B^{-\Delta}\) are polynomially bounded. For a uniformly random \(\Delta\)-regular graph, the eigenvalue estimates in \cite{Fri08,Bor20} give
\[
\lambda^{\star} = -\lambda_{\min}\tp{\*A_G} \le \tp{2 + o_n\tp{1}} \sqrt{\Delta - 1}
\]
with probability \(1 - o_n\tp{1}\).

\begin{corollary}
    Fix \(q \ge 3\) and \(\zeta \in \tp{0, 1 / 2}\). Let \(G\) be a uniformly random \(\Delta\)-regular graph on \(n\) vertices, and suppose that
    \[
    1 - B \le \tp{\frac{1}{2} - \zeta} \frac{\log \Delta}{\Delta}.
    \]
    For all sufficiently large \(\Delta\), with probability \(1 - o_n\tp{1}\), the embedded Potts measure is \(\delta\)-ULC with \(\delta\) as in \Cref{thm:antiferromagnetic-potts-ulc} and \(\delta \ge 1 / 2\), and
    \[
    T_{\-{mix}}\tp{\eps; P_{\-{GD}}} = O_q\tp{\sqrt{\Delta} n^2 \log \frac{1}{\eps}}.
    \]
    \label{cor:q-potts-random-regular-mixing}
\end{corollary}

For each fixed \(q\), this regime includes values \(B < B_c\) for all sufficiently large \(\Delta\), where the conjectured tree-uniqueness threshold \(B_c\) is
\[
B_c \defeq \max\set{0, 1 - \frac{q}{\Delta}}.
\]
This formula for the threshold has been established for \(q = 3, 4\) and for fixed \(q \ge 5\) when \(\Delta\) is sufficiently large \cite{GGY18,BBR23,BBBR23}.

\section{Laguerre Dynamics}
\label{sec:laguerre}

In this section, we consider the continuous limit of the birth-death chains in \Cref{sec:bakry-emery-birth-death} and identify the scaling limit of the linear-death chain with Laguerre diffusion.

The Laguerre diffusion is a canonical reversible diffusion on the positive orthant. Let \(V \in \+C^3\tp{\bb{R}_{> 0}^d}\) be such that \(\mu\tp{\-d \*x} \propto e^{-V\tp{\*x}} \dd \*x\) is a probability measure on \(\bb R_{> 0}^d\), and assume the usual conditions ensuring that the following SDE is well defined and non-explosive:
\[
\dd X_{t, i} = \tp{1 - X_{t, i} \partial_i V\tp{\*X_t}} \dd t + \sqrt{2 X_{t, i}} \dd B_{t, i}, \quad i \in \stp{d},
\]
where \(\tp{\*B_t}_{t \ge 0}\) is a standard \(d\)-dimensional Brownian motion. The generator is
\[
\+L f\tp{\*x} \defeq \sum_{i = 1}^d \tp{1 - x_i \partial_i V\tp{\*x}} \partial_i f\tp{\*x} + \sum_{i = 1}^d x_i \partial_{ii} f\tp{\*x},
\]
and its divergence-form representation shows that \(\+L\) is reversible with respect to \(\mu\). The associated Dirichlet form is
\[
\+E\tp{f, g} \defeq -\*E_{\mu}\stp{f \+L g} = \*E_{\mu}\stp{\sum_{i = 1}^d x_i \partial_i f\tp{\*x} \partial_i g\tp{\*x}}.
\]
Accordingly, the \Poincare inequality is weighted:
\[
\gamma \, \*{Var}_{\mu}\stp{f} \le \*E_{\mu}\stp{\sum_{i = 1}^d x_i \tp{\partial_i f\tp{\*x}}^2}, \quad \forall f \in \+C_{\-c}^{\infty}\tp{\bb{R}_{> 0}^d}.
\]

\paragraph{Generator and Dirichlet form.}

For \(N \in \bb{N}_{> 0}\), define
\[
\iota_N\tp{\*n} \defeq \frac{\*n + \*1}{N}, \qquad \mu_N\tp{\*n} \propto \exp\tp{-V\tp{\iota_N\tp{\*n}}}, \quad \*n \in \bb{N}^d,
\]
and let \(\+L_N\) and \(\+E_N\) denote the birth-death generator and Dirichlet form associated with \(\mu_N\). Writing \(\*x = \iota_N\tp{\*n}\), the birth rates are
\[
b_{N, i}\tp{\*n} = \frac{\tp{n_i + 1} \mu_N\tp{\*n + \*e_i}}{\mu_N\tp{\*n}} = N x_i \exp\tp{V\tp{\*x} - V\tp{\*x + N^{-1} \*e_i}} = N x_i - x_i \partial_i V\tp{\*x} + O\tp{N^{-1}}.
\]
Write \(f_N \defeq f \circ \iota_N\). Direct Taylor expansion gives
\[
\begin{aligned}
    \+L_N f_N\tp{\*n} &= \sum_{i = 1}^d b_{N, i}\tp{\*n} \nabla_i^+ f_N\tp{\*n} + \sum_{i = 1}^d n_i \nabla_i^- f_N\tp{\*n} \\
    &= \sum_{i = 1}^d \tp{N x_i - x_i \partial_i V\tp{\*x} + O\tp{N^{-1}}} \tp{f\tp{\*x + N^{-1} \*e_i} - f\tp{\*x}} \\
    &\quad + \sum_{i = 1}^d \tp{N x_i - 1} \tp{f\tp{\*x - N^{-1} \*e_i} - f\tp{\*x}} \\
    &= N^{-1} \tp{\sum_{i = 1}^d \tp{1 - x_i \partial_i V\tp{\*x}} \partial_i f\tp{\*x} + \sum_{i = 1}^d x_i \partial_{ii} f\tp{\*x}} + O\tp{N^{-2}}.
\end{aligned}
\]
Hence, locally uniformly on \(\bb{R}_{> 0}^d\),
\[
N \+L_N f_N\tp{\*n} \to \+L f\tp{\*x} \quad \text{as } N \to \infty, \qquad \*x = \iota_N\tp{\*n}.
\]
The pushforward measures of \(\mu_N\) under \(\iota_N\) converge to \(\mu\) by Riemann-sum convergence. Moreover, if \(g_N \defeq g \circ \iota_N\), then
\begin{equation}
    N \+E_N\tp{f_N, g_N} \to \+E\tp{f, g} \quad \text{as } N \to \infty,
    \label{eq:dirichlet-form-convergence}
\end{equation}
for smooth compactly supported \(f, g\).

\paragraph{Local curvature matrix and \(\kappa\)-Laguerre-ULC.}

Define the local curvature matrix in the Laguerre setting by
\begin{equation}
    \*K_{\*x} \defeq \nabla^2 V\tp{\*x} + \diag \tp{\frac{\partial_i V\tp{\*x}}{x_i}}_{i \in \stp{d}}.
    \label{eq:laguerre-local-curvature}
\end{equation}
We say that \(\mu\) is \(\kappa\)-Laguerre-ULC, where \(\kappa > 0\), if
\begin{equation}
    \*K_{\*x} \succeq \kappa \, \diag \tp{\frac{1}{x_i}}_{i \in \stp{d}}, \quad \forall \*x \in \bb{R}_{> 0}^d.
    \label{eq:laguerre-ulc-condition}
\end{equation}

These definitions are justified by the fact that they arise as the scaling limit of the discrete local curvature matrix in \eqref{eq:local-curvature-matrix} and the \(\delta\)-ULC condition. Let \(a_N\tp{\*n} \defeq \*n! \mu_N\tp{\*n}\), let \(\*K_{\*n}^{\tp{N}}\) be the discrete local curvature matrix in \eqref{eq:local-curvature-matrix}, and set
\[
\*D_{\*n}^{\tp{N}} \defeq \diag \tp{\frac{a_N\tp{\*n}}{a_N\tp{\*n + \*e_i}}}_{i \in \stp{d}}.
\]
A direct entrywise Taylor expansion yields, locally uniformly on the positive orthant,
\[
\*K_{\*n}^{\tp{N}} = N^{-2} \*K_{\*x} + O\tp{N^{-3}}, \qquad \*D_{\*n}^{\tp{N}} = N^{-1} \diag \tp{\frac{1}{x_i}}_{i \in \stp{d}} + O\tp{N^{-2}}, \qquad \*x = \iota_N\tp{\*n}.
\]
Therefore, for any sequence \(\delta_N\) satisfying \(N \delta_N \to \kappa\),
\[
N^2 \tp{\*K_{\*n}^{\tp{N}} - \delta_N \*D_{\*n}^{\tp{N}}} \to \*K_{\*x} - \kappa \, \diag \tp{\frac{1}{x_i}}_{i \in \stp{d}}.
\]
Thus \eqref{eq:laguerre-ulc-condition} is precisely the scaling limit of the discrete \(\delta_N\)-ULC condition. Under uniform control of the Taylor remainders, for every \(0 < \kappa' < \kappa\), it yields a sequence \(\delta_N\) such that \(N \delta_N \to \kappa'\) and \(\mu_N\) is \(\delta_N\)-ULC. Without such uniform control, the same conclusion follows by first applying this observation on compact truncations and then removing the truncation by a standard cutoff argument.

\paragraph{\Poincare and Brascamp--Lieb Inequalities.}

\begin{theorem}[Laguerre \Poincare and Brascamp--Lieb Inequalities]
    If \(\mu\) is \(\kappa\)-Laguerre-ULC, then
    \[
    \kappa \, \*{Var}_{\mu}\stp{f} \le \*E_{\mu}\stp{\sum_{i = 1}^d x_i \tp{\partial_i f\tp{\*x}}^2}, \quad \forall f \in \+C_{\-c}^{\infty}\tp{\bb{R}_{> 0}^d}.
    \]
    More generally, if the matrix \(\*K_{\*x}\) defined in \eqref{eq:laguerre-local-curvature} is positive definite for every \(\*x \in \bb{R}_{> 0}^d\), then
    \[
    \*{Var}_{\mu}\stp{f} \le \*E_{\mu}\stp{\tp{\nabla f}^{\top} \*K^{-1} \nabla f}, \quad \forall f \in \+C_{\-c}^{\infty}\tp{\bb{R}_{> 0}^d},
    \]
    where \(\*K = \*K_{\*x}\) in the integrand.
    \label{thm:laguerre-functional-inequalities}
\end{theorem}

\begin{remark}
    \Cref{thm:laguerre-functional-inequalities} may also be proved by developing the integrated Bakry--\Emery calculus directly for the Laguerre diffusion. The argument parallels the one in \Cref{sec:bakry-emery-birth-death}.
\end{remark}

\begin{remark}
    The Laguerre Brascamp--Lieb inequality can also be derived from the intertwining framework of \cite[Theorem~3.1]{ABJ18}. Indeed, for the auxiliary Langevin generator \(\+L_0 = \Delta - \nabla V \cdot \nabla\), choose the distortion matrix
    \[
    \*A\tp{\*x} \defeq \diag \tp{\frac{1}{x_i}}_{i \in \stp{d}}.
    \]
    Since \(\+L_0 x_i = -\partial_i V\tp{\*x}\), the curvature matrix in that framework is
    \[
    \*A^{-1} \*M_{\*A} \*A = \nabla^2 V\tp{\*x} - \tp{\+L_0 \*A^{-1}}\tp{\*x} \*A\tp{\*x} = \nabla^2 V\tp{\*x} + \diag \tp{\frac{\partial_i V\tp{\*x}}{x_i}}_{i \in \stp{d}} = \*K_{\*x}.
    \]
    Consequently, \cite[Theorem~3.1]{ABJ18}, together with a cutoff approximation on the positive orthant, yields the same Brascamp--Lieb inequality.
\end{remark}

\begin{proof}[Proof of \Cref{thm:laguerre-functional-inequalities}]
    Suppose first that \(\mu\) is \(\kappa\)-Laguerre-ULC, and fix \(0 < \kappa' < \kappa\). By the curvature limit above, there are \(\delta_N > 0\) such that \(N \delta_N \to \kappa'\) and \(\mu_N\) is \(\delta_N\)-ULC. Applying \Cref{thm:ulc-poincare-inequality} to \(f_N \defeq f \circ \iota_N\) and multiplying the resulting inequality by \(N\) gives
    \[
    N \delta_N \, \*{Var}_{\mu_N}\stp{f_N} \le N \+E_N\tp{f_N, f_N}.
    \]
    The Riemann-sum convergence of \(\mu_N\), together with the Dirichlet-form limit \eqref{eq:dirichlet-form-convergence}, yields
    \[
    \kappa' \, \*{Var}_{\mu}\stp{f} \le \+E\tp{f, f} = \*E_{\mu}\stp{\sum_{i = 1}^d x_i \tp{\partial_i f\tp{\*x}}^2}.
    \]
    Letting \(\kappa' \to \kappa\) proves the weighted \Poincare inequality.

    Now suppose that \(\*K_{\*x} \succ \*O\) for every \(\*x \in \bb{R}_{> 0}^d\). On compact truncations, the curvature expansion implies that \(\*K_{\*n}^{\tp{N}} \succ \*O\) for all sufficiently large \(N\). Hence \Cref{thm:discrete-brascamp-lieb} gives
    \[
    \*{Var}_{\mu_N}\stp{f_N} \le \*E_{\mu_N}\stp{\tp{\nabla^+ f_N}^{\top} \tp{\*K^{\tp{N}}}^{-1} \nabla^+ f_N}.
    \]
    Uniformly on the support of \(f\),
    \[
    \nabla^+ f_N\tp{\*n} = \frac{1}{N} \nabla f\tp{\*x} + O\tp{N^{-2}}, \qquad \tp{\*K_{\*n}^{\tp{N}}}^{-1} = N^2 \*K_{\*x}^{-1} + O\tp{N}.
    \]
    Therefore, the integrand on the right-hand side converges to \(\tp{\nabla f}^{\top} \*K^{-1} \nabla f\). Passing to the limit in the discrete Brascamp--Lieb inequality and then removing the compact truncation proves the claim.
\end{proof}

\begin{example}[Product Gamma Measures]
    Fix \(\alpha_1, \dots, \alpha_d, \kappa > 0\), and let
    \[
    \mu \defeq \bigotimes_{i = 1}^d \-{Gamma}\tp{\alpha_i, \kappa}.
    \]
    A direct calculation shows that \(\mu\) is exactly \(\kappa\)-Laguerre-ULC; that is, equality holds in \eqref{eq:laguerre-ulc-condition}. Moreover, the weighted \Poincare inequality in \Cref{thm:laguerre-functional-inequalities} has the sharp constant \(\kappa\): for each \(i\), the function \(f_i\tp{\*x} \defeq x_i - \alpha_i / \kappa\) is an eigenfunction of \(-\+L\) with eigenvalue \(\kappa\).

    This example also admits a discrete approximation that preserves the curvature constant. For \(N > \kappa\), let \(X_{N, 1}, \dots, X_{N, d}\) be independent random variables with \(X_{N, i} \sim \-{NB}\tp{\alpha_i, \delta_N}\), where \(\delta_N \defeq \kappa / N\). Then
    \[
    \-{Law}\tp{\frac{1}{N} \tp{X_{N, 1}, \dots, X_{N, d}}} \Rightarrow \mu,
    \]
    while the product negative binomial law is exactly \(\delta_N\)-ULC. Its birth-death generator has spectral gap \(\delta_N\), so the generator accelerated by \(N\) has spectral gap \(N \delta_N = \kappa\), consistent with the Laguerre limit.
\end{example}

\paragraph{Unit-birth chain and Langevin diffusion.}

The relationship between the linear-death chain in \eqref{eq:birth-death-generator} and the Laguerre diffusion reflects a broader discrete-to-continuous correspondence. An analogous limit connects the unit-birth chain with the Langevin diffusion. In \Cref{subsec:discrete-brascamp-lieb}, we introduced the unit-birth chain \eqref{eq:unit-birth-generator} on \(\bb{Z}^d\), the \(c\)-log-concavity condition \eqref{eq:c-log-concave}, and the corresponding discrete Brascamp--Lieb inequality in \Cref{thm:discrete-brascamp-lieb-unit-birth}. Let \(\mu\tp{\-d \*x} \propto e^{-V\tp{\*x}} \dd \*x\) be a probability measure on \(\bb{R}^d\) with a smooth potential \(V\), and discretize \(\mu\) on the grid \(N^{-1} \bb{Z}^d\). Accelerating the corresponding unit-birth chain by \(N^2\) yields the following limits, whose proofs we omit:
\begin{itemize}
    \item The generator and Dirichlet form converge to those of the Langevin diffusion. More generally, if the discrete \Poincare inequalities have constants \(\gamma_N\) such that \(N^2 \gamma_N \to \gamma\), then their continuous-space limit is the standard \Poincare inequality
    \[
    \gamma \, \*{Var}_{\mu}\stp{f} \le \*E_{\mu}\stp{\norm{\nabla f}_2^2}, \quad \forall f \in \+C_{\-c}^{\infty}\tp{\bb{R}^d}.
    \]
    \item If \(\mu\) is \(\alpha\)-strongly log-concave, i.e., \(\nabla^2 V \succ \alpha \*I\), then the lattice approximations are \(c_N\)-log-concave with \(N^2 c_N \to \alpha\). In particular, the discrete \Poincare inequality for \(c\)-log-concave measures in \cite[Theorem~1.5 and Proposition~9.1]{Joh17} yields, through this limit, a discrete approximation proof of the standard \Poincare inequality for strongly log-concave measures.
    \item If \(\mu\) is strictly log-concave, i.e., \(\nabla^2 V \succ \*O\), then the discrete Brascamp--Lieb inequality in \Cref{thm:discrete-brascamp-lieb-unit-birth} converges to the standard Brascamp--Lieb inequality
    \[
    \*{Var}_{\mu}\stp{f} \le \*E_{\mu}\stp{\tp{\nabla f}^{\top} \tp{\nabla^2 V}^{-1} \nabla f}, \quad \forall f \in \+C_{\-c}^{\infty}\tp{\bb{R}^d}.
    \]
\end{itemize}

\section*{AI Disclosure}

AI tools were used to assist in searching the literature, streamlining the proofs, polishing the exposition, and typesetting the paper. The proof of \Cref{thm:poisson-laplace-transform-domination} was substantially inspired by an approach suggested by Gemini Pro. The proof of \Cref{thm:poisson-maximum-entropy} was generated by GPT-5.6 Pro in response to a statement of the desired result. The technical arguments in \Cref{subsec:ulc-completely-log-concavity,sec:ulc-closure-properties} were completed with the assistance of GPT-5.6 Sol, primarily for routine technical details. All remaining proofs were developed by the authors. The authors independently verified the correctness, originality, and references of all content in the paper.

\bibliographystyle{alpha}
\bibliography{refs}

@article{GJMPPS26,
  author        = {G{\"o}bel, Andreas and Jenssen, Matthew and Michelen, Marcus and Pappik, Marcus and Perkins, Will and Schiller, Leon},
  title         = {A simple proof of rapid mixing on random regular graphs beyond uniqueness},
  journal       = {arXiv preprint arXiv:2606.27545},
  year          = {2026},
  eprint        = {2606.27545},
  archiveprefix = {arXiv}
}

@article{KKO13,
  author        = {Kondratiev, Yuri and Kuna, Tobias and Ohlerich, Nataliya},
  title         = {Spectral gap for {Glauber} type dynamics for a special class of potentials},
  journal       = {Electronic Journal of Probability},
  volume        = {18},
  number        = {42},
  pages         = {1--18},
  year          = {2013},
  doi           = {10.1214/EJP.v18-2260},
  eprint        = {1103.5079},
  archiveprefix = {arXiv}
}

@inproceedings{CCCYZ25,
  author        = {Chen, Xiaoyu and Chen, Zejia and Chen, Zongchen and Yin, Yitong and Zhang, Xinyuan},
  title         = {Rapid mixing on random regular graphs beyond uniqueness},
  booktitle     = {2025 IEEE 66th Annual Symposium on Foundations of Computer Science (FOCS)},
  pages         = {2170--2193},
  year          = {2025},
  doi           = {10.1109/FOCS63196.2025.00115},
  eprint        = {2504.03406},
  archiveprefix = {arXiv}
}

@article{Joh17,
  author        = {Johnson, Oliver},
  title         = {A discrete log-{Sobolev} inequality under a {Bakry-\'Emery} type condition},
  journal       = {Annales de l'Institut Henri Poincar{\'e}, Probabilit{\'e}s et Statistiques},
  volume        = {53},
  number        = {4},
  pages         = {1952--1970},
  year          = {2017},
  doi           = {10.1214/16-AIHP778},
  eprint        = {1507.06268},
  archiveprefix = {arXiv}
}

@article{BCDPP06,
  title         = {Spectral gap estimates for interacting particle systems via a {B}ochner-type identity},
  author        = {Boudou, Anne-Severine and Caputo, Pietro and Dai Pra, Paolo and Posta, Gustavo},
  journal       = {Journal of Functional Analysis},
  volume        = {232},
  number        = {1},
  pages         = {222--258},
  year          = {2006},
  doi           = {10.1016/j.jfa.2005.07.012},
  eprint        = {math/0505533},
  archiveprefix = {arXiv}
}

@article{CE25,
  author        = {Chen, Yuansi and Eldan, Ronen},
  title         = {Localization schemes: A framework for proving mixing bounds for {Markov} chains},
  journal       = {Duke Mathematical Journal},
  volume        = {174},
  number        = {8},
  pages         = {1431--1510},
  year          = {2025},
  doi           = {10.1215/00127094-2024-0063},
  eprint        = {2203.04163},
  archiveprefix = {arXiv}
}

@article{STZ25,
  author        = {Shi, Bobby and Tian, Kevin and Zhang, Matthew S.},
  title         = {Perspectives on stochastic localization},
  journal       = {arXiv preprint arXiv:2510.04460},
  year          = {2025},
  eprint        = {2510.04460},
  archiveprefix = {arXiv}
}

@article{EAM22,
  title         = {An information-theoretic view of stochastic localization},
  author        = {El Alaoui, Ahmed and Montanari, Andrea},
  journal       = {IEEE Transactions on Information Theory},
  volume        = {68},
  number        = {11},
  pages         = {7423--7426},
  year          = {2022},
  doi           = {10.1109/TIT.2022.3180298},
  eprint        = {2109.00709},
  archiveprefix = {arXiv}
}

@article{AMM21,
  title         = {Concentration inequalities for ultra log-concave distributions},
  author        = {Aravinda, Heshan and Marsiglietti, Arnaud and Melbourne, James},
  journal       = {Studia Mathematica},
  volume        = {265},
  number        = {1},
  pages         = {111--120},
  year          = {2022},
  doi           = {10.4064/sm210605-2-10},
  eprint        = {2104.05054},
  archiveprefix = {arXiv}
}

@article{DJ13,
  title         = {{Bounds on the Poincar{\'e} constant under negative dependence}},
  author        = {Daly, Fraser and Johnson, Oliver},
  journal       = {Statistics \& Probability Letters},
  volume        = {83},
  number        = {2},
  pages         = {511--518},
  year          = {2013},
  doi           = {10.1016/j.spl.2012.11.001},
  eprint        = {0801.2112},
  archiveprefix = {arXiv}
}

@article{Wu00,
  title   = {{A new modified logarithmic Sobolev inequality for Poisson point processes and several applications}},
  author  = {Wu, Liming},
  journal = {Probability Theory and Related Fields},
  volume  = {118},
  number  = {3},
  pages   = {427--438},
  year    = {2000},
  doi     = {10.1007/PL00008749}
}

@article{LRS25,
  title         = {{The Poisson transport map}},
  author        = {{L{\'o}pez-Rivera}, Pablo and Shenfeld, Yair},
  journal       = {Journal of Functional Analysis},
  volume        = {288},
  number        = {10},
  pages         = {110864},
  year          = {2025},
  doi           = {10.1016/j.jfa.2025.110864},
  eprint        = {2407.02359},
  archiveprefix = {arXiv}
}

@article{WJ08,
  title   = {Graphical models, exponential families, and variational inference},
  author  = {Wainwright, Martin J. and Jordan, Michael I.},
  journal = {Foundations and Trends in Machine Learning},
  volume  = {1},
  number  = {1--2},
  pages   = {1--305},
  year    = {2008},
  doi     = {10.1561/2200000001}
}

@inproceedings{AKV24,
  title         = {Trickle-down in localization schemes and applications},
  author        = {Anari, Nima and Koehler, Frederic and Vuong, Thuy-Duong},
  booktitle     = {Proceedings of the 56th Annual ACM Symposium on Theory of Computing},
  pages         = {1094--1105},
  year          = {2024},
  doi           = {10.1145/3618260.3649622},
  eprint        = {2407.16104},
  archiveprefix = {arXiv}
}

@article{ABJ18,
  title         = {{Intertwinings and generalized Brascamp--Lieb inequalities}},
  author        = {Arnaudon, Marc and Bonnefont, Michel and Joulin, Ald{\'e}ric},
  journal       = {Revista Matem{\'a}tica Iberoamericana},
  volume        = {34},
  number        = {3},
  pages         = {1021--1054},
  year          = {2018},
  doi           = {10.4171/RMI/1014},
  eprint        = {1602.03836},
  archiveprefix = {arXiv}
}

@article{Kah01,
  author  = {Kahn, Jeff},
  title   = {An entropy approach to the hard-core model on bipartite graphs},
  journal = {Combinatorics, Probability and Computing},
  volume  = {10},
  number  = {3},
  pages   = {219--237},
  year    = {2001},
  doi     = {10.1017/S0963548301004631}
}

@incollection{GT04,
  author    = {Galvin, David and Tetali, Prasad},
  title     = {On weighted graph homomorphisms},
  booktitle = {Graphs, Morphisms and Statistical Physics},
  editor    = {Ne\v{s}et\v{r}il, Jaroslav and Winkler, Peter},
  series    = {DIMACS Series in Discrete Mathematics and Theoretical Computer Science},
  volume    = {63},
  pages     = {97--104},
  year      = {2004},
  publisher = {American Mathematical Society},
  doi       = {10.1090/DIMACS/063/07}
}

@article{Zha10,
  author  = {Zhao, Yufei},
  title   = {The number of independent sets in a regular graph},
  journal = {Combinatorics, Probability and Computing},
  volume  = {19},
  number  = {2},
  pages   = {315--320},
  year    = {2010},
  doi     = {10.1017/S0963548309990538}
}

@article{CR11,
  author  = {Cutler, Jonathan and Radcliffe, A. J.},
  title   = {Extremal problems for independent set enumeration},
  journal = {The Electronic Journal of Combinatorics},
  volume  = {18},
  number  = {1},
  pages   = {P169},
  year    = {2011},
  doi     = {10.37236/656}
}

@article{DJPR17,
  author        = {Davies, Ewan and Jenssen, Matthew and Perkins, Will and Roberts, Barnaby},
  title         = {Independent sets, matchings, and occupancy fractions},
  journal       = {Journal of the London Mathematical Society},
  volume        = {96},
  number        = {1},
  pages         = {47--66},
  year          = {2017},
  doi           = {10.1112/jlms.12056},
  eprint        = {1508.04675},
  archiveprefix = {arXiv}
}

@article{SSSZ19,
  author  = {Sah, Ashwin and Sawhney, Mehtaab and Stoner, David and Zhao, Yufei},
  title   = {The number of independent sets in an irregular graph},
  journal = {Journal of Combinatorial Theory, Series B},
  volume  = {138},
  pages   = {172--195},
  year    = {2019},
  doi     = {10.1016/j.jctb.2019.01.007}
}

@article{DK25,
  author        = {Davies, Ewan and Kang, Ross J.},
  title         = {The hard-core model in graph theory},
  journal       = {arXiv preprint arXiv:2501.03379},
  year          = {2025},
  eprint        = {2501.03379},
  archiveprefix = {arXiv}
}

@article{DST25,
  author        = {Davies, Ewan and Sandhu, Juspreet Singh and Tan, Brian},
  title         = {On expectations and variances in the hard-core model on bounded degree graphs},
  journal       = {arXiv preprint arXiv:2505.13396},
  year          = {2025},
  eprint        = {2505.13396},
  archiveprefix = {arXiv}
}

@article{ZX26,
  author        = {Zhang, Weiyuan and Xu, Kexiang},
  title         = {On the variance fraction of the hard-core model on graphs with bounded maximum degree},
  journal       = {arXiv preprint arXiv:2604.01717},
  year          = {2026},
  eprint        = {2604.01717},
  archiveprefix = {arXiv}
}

@article{Joh07,
  title         = {Log-concavity and the maximum entropy property of the {Poisson} distribution},
  author        = {Johnson, Oliver},
  journal       = {Stochastic Processes and their Applications},
  volume        = {117},
  number        = {6},
  pages         = {791--802},
  year          = {2007},
  doi           = {10.1016/j.spa.2006.10.006},
  eprint        = {math/0603647},
  archiveprefix = {arXiv}
}

@article{JKM13,
  title         = {Log-concavity, ultra-log-concavity, and a maximum entropy property of discrete compound {Poisson} measures},
  author        = {Johnson, Oliver and Kontoyiannis, Ioannis and Madiman, Mokshay},
  journal       = {Discrete Applied Mathematics},
  volume        = {161},
  number        = {9},
  pages         = {1232--1250},
  year          = {2013},
  doi           = {10.1016/j.dam.2011.08.025},
  eprint        = {0912.0581},
  archiveprefix = {arXiv}
}

@article{Yu10,
  author  = {Yu, Yaming},
  title   = {Relative log-concavity and a pair of triangle inequalities},
  journal = {Bernoulli},
  volume  = {16},
  number  = {2},
  pages   = {459--470},
  year    = {2010},
  doi     = {10.3150/09-BEJ216}
}

@book{LPW17,
  author    = {Levin, David A. and Peres, Yuval and Wilmer, Elizabeth L.},
  title     = {{Markov} chains and mixing times},
  edition   = {2nd},
  year      = {2017},
  publisher = {American Mathematical Society},
  address   = {Providence, RI},
  isbn      = {978-1-4704-2962-1}
}

@inproceedings{BT03,
  author    = {Bobkov, Sergey G. and Tetali, Prasad},
  title     = {Modified log-{Sobolev} inequalities, mixing and hypercontractivity},
  booktitle = {Proceedings of the Thirty-Fifth Annual ACM Symposium on Theory of Computing},
  pages     = {287--296},
  year      = {2003},
  doi       = {10.1145/780542.780586}
}

@book{BGL14,
  author    = {Bakry, Dominique and Gentil, Ivan and Ledoux, Michel},
  title     = {Analysis and geometry of {Markov} diffusion operators},
  series    = {Grundlehren der mathematischen Wissenschaften},
  volume    = {348},
  year      = {2014},
  publisher = {Springer},
  doi       = {10.1007/978-3-319-00227-9}
}

@incollection{BE85,
  author    = {Bakry, Dominique and {\'E}mery, Michel},
  title     = {Diffusions hypercontractives},
  booktitle = {S{\'e}minaire de Probabilit{\'e}s XIX, 1983/84},
  editor    = {Az{\'e}ma, Jacques and Yor, Marc},
  series    = {Lecture Notes in Mathematics},
  volume    = {1123},
  pages     = {177--206},
  year      = {1985},
  publisher = {Springer},
  address   = {Berlin},
  doi       = {10.1007/BFb0075847}
}

@article{BH20,
  author  = {Br{\"a}nd{\'e}n, Petter and Huh, June},
  title   = {Lorentzian polynomials},
  journal = {Annals of Mathematics},
  volume  = {192},
  number  = {3},
  pages   = {821--891},
  year    = {2020},
  doi     = {10.4007/annals.2020.192.3.4}
}

@book{Fri08,
  author        = {Friedman, Joel},
  title         = {A proof of {Alon}'s second eigenvalue conjecture and related problems},
  series        = {Memoirs of the American Mathematical Society},
  volume        = {195},
  number        = {910},
  year          = {2008},
  publisher     = {American Mathematical Society},
  doi           = {10.1090/memo/0910},
  eprint        = {cs/0405020},
  archiveprefix = {arXiv}
}

@article{Bor20,
  author        = {Bordenave, Charles},
  title         = {A new proof of {Friedman}'s second eigenvalue theorem and its extension to random lifts},
  journal       = {Annales scientifiques de l'{\'E}cole normale sup{\'e}rieure},
  volume        = {53},
  number        = {6},
  pages         = {1393--1439},
  year          = {2020},
  doi           = {10.24033/asens.2450},
  eprint        = {1502.04482},
  archiveprefix = {arXiv}
}

@article{GGY18,
  title         = {Uniqueness for the 3-state antiferromagnetic {Potts} model on the tree},
  author        = {Galanis, Andreas and Goldberg, Leslie Ann and Yang, Kuan},
  journal       = {Electronic Journal of Probability},
  volume        = {23},
  number        = {82},
  pages         = {1--43},
  year          = {2018},
  doi           = {10.1214/18-EJP211},
  eprint        = {1804.03514},
  archiveprefix = {arXiv}
}

@article{BBR23,
  title         = {Uniqueness of the {Gibbs} measure for the 4-state anti-ferromagnetic {Potts} model on the regular tree},
  author        = {de Boer, David and Buys, Pjotr and Regts, Guus},
  journal       = {Combinatorics, Probability and Computing},
  volume        = {32},
  number        = {1},
  pages         = {158--182},
  year          = {2023},
  doi           = {10.1017/S0963548322000207},
  eprint        = {2011.05638},
  archiveprefix = {arXiv}
}

@article{BBBR23,
  title         = {Uniqueness of the {Gibbs} measure for the anti-ferromagnetic {Potts} model on the infinite {$\Delta$}-regular tree for large {$\Delta$}},
  author        = {Bencs, Ferenc and de Boer, David and Buys, Pjotr and Regts, Guus},
  journal       = {Journal of Statistical Physics},
  volume        = {190},
  number        = {8},
  pages         = {140},
  year          = {2023},
  doi           = {10.1007/s10955-023-03145-z},
  eprint        = {2203.15457},
  archiveprefix = {arXiv}
}

@article{KL18,
  author        = {Klartag, Bo'az and Lehec, Joseph},
  title         = {Poisson processes and a log-concave {Bernstein} theorem},
  journal       = {Studia Mathematica},
  volume        = {247},
  number        = {1},
  pages         = {85--107},
  year          = {2019},
  doi           = {10.4064/sm180212-30-7},
  eprint        = {1802.04176},
  archiveprefix = {arXiv}
}

@article{SBP26,
  author        = {Shenfeld, Yair and Baptista, Ricardo and Peluchetti, Stefano},
  title         = {Binomial flows: Denoising and flow matching for discrete ordinal data},
  journal       = {arXiv preprint arXiv:2605.00360},
  year          = {2026},
  eprint        = {2605.00360},
  archiveprefix = {arXiv}
}

@article{Walk76,
  author  = {Walkup, David W.},
  title   = {P{\'o}lya sequences, binomial convolution and the union of random sets},
  journal = {Journal of Applied Probability},
  volume  = {13},
  number  = {1},
  pages   = {76--85},
  year    = {1976},
  doi     = {10.2307/3212667}
}

@article{Lig97,
  author  = {Liggett, Thomas M.},
  title   = {Ultra logconcave sequences and negative dependence},
  journal = {Journal of Combinatorial Theory, Series A},
  volume  = {79},
  number  = {2},
  pages   = {315--325},
  year    = {1997},
  doi     = {10.1006/jcta.1997.2790}
}

@article{Pem00,
  author  = {Pemantle, Robin},
  title   = {Towards a theory of negative dependence},
  journal = {Journal of Mathematical Physics},
  volume  = {41},
  number  = {3},
  pages   = {1371--1390},
  year    = {2000},
  doi     = {10.1063/1.533200}
}

@article{SW14,
  author  = {Saumard, Adrien and Wellner, Jon A.},
  title   = {Log-concavity and strong log-concavity: a review},
  journal = {Statistics Surveys},
  volume  = {8},
  pages   = {45--114},
  year    = {2014},
  doi     = {10.1214/14-SS107}
}

@article{BL76,
  author  = {Brascamp, Herm Jan and Lieb, Elliott H.},
  title   = {On extensions of the {Brunn--Minkowski} and {Pr{\'e}kopa--Leindler} theorems, including inequalities for log concave functions, and with an application to the diffusion equation},
  journal = {Journal of Functional Analysis},
  volume  = {22},
  number  = {4},
  pages   = {366--389},
  year    = {1976},
  doi     = {10.1016/0022-1236(76)90004-5}
}

@article{ALOVii,
  author        = {Anari, Nima and Liu, Kuikui and {Oveis Gharan}, Shayan and Vinzant, Cynthia},
  title         = {Log-concave polynomials {II}: High-dimensional walks and an {FPRAS} for counting bases of a matroid},
  journal       = {Annals of Mathematics},
  volume        = {199},
  number        = {1},
  pages         = {259--299},
  year          = {2024},
  doi           = {10.4007/annals.2024.199.1.4},
  eprint        = {1811.01816},
  archiveprefix = {arXiv}
}

@article{CGM21,
  author        = {Cryan, Mary and Guo, Heng and Mousa, Giorgos},
  title         = {Modified log-{Sobolev} inequalities for strongly log-concave distributions},
  journal       = {The Annals of Probability},
  volume        = {49},
  number        = {1},
  pages         = {506--525},
  year          = {2021},
  doi           = {10.1214/20-AOP1453},
  eprint        = {1903.06081},
  archiveprefix = {arXiv}
}

@article{ALGV24,
  author        = {Anari, Nima and Liu, Kuikui and {Oveis Gharan}, Shayan and Vinzant, Cynthia},
  title         = {Log-concave polynomials {III}: {Mason}'s ultra-log-concavity conjecture for independent sets of matroids},
  journal       = {Proceedings of the American Mathematical Society},
  volume        = {152},
  number        = {5},
  pages         = {1969--1981},
  year          = {2024},
  doi           = {10.1090/proc/16724},
  eprint        = {1811.01600},
  archiveprefix = {arXiv}
}

@article{CDPP09,
  author        = {Caputo, Pietro and Dai Pra, Paolo and Posta, Gustavo},
  title         = {Convex entropy decay via the {Bochner--Bakry--Emery} approach},
  journal       = {Annales de l'Institut Henri Poincar{\'e}, Probabilit{\'e}s et Statistiques},
  volume        = {45},
  number        = {3},
  pages         = {734--753},
  year          = {2009},
  doi           = {10.1214/08-AIHP183},
  eprint        = {0712.2578},
  archiveprefix = {arXiv}
}

@article{HH02,
  author  = {Herzog, J{\"u}rgen and Hibi, Takayuki},
  title   = {Discrete polymatroids},
  journal = {Journal of Algebraic Combinatorics},
  volume  = {16},
  number  = {3},
  pages   = {239--268},
  year    = {2002},
  doi     = {10.1023/A:1021852421716}
}

@article{Eld13,
  author        = {Eldan, Ronen},
  title         = {Thin shell implies spectral gap up to polylog via a stochastic localization scheme},
  journal       = {Geometric and Functional Analysis},
  volume        = {23},
  number        = {2},
  pages         = {532--569},
  year          = {2013},
  doi           = {10.1007/s00039-013-0214-y},
  eprint        = {1203.0893},
  archiveprefix = {arXiv}
}

@inproceedings{RSJ19,
  author        = {Robinson, Joshua and Sra, Suvrit and Jegelka, Stefanie},
  title         = {Flexible modeling of diversity with strongly log-concave distributions},
  booktitle     = {Advances in Neural Information Processing Systems 32},
  pages         = {15199--15209},
  year          = {2019},
  eprint        = {1906.05413},
  archiveprefix = {arXiv},
  url           = {https://papers.nips.cc/paper_files/paper/2019/hash/d55eaf8506f9046a88b8730781830194-Abstract.html}
}

@inproceedings{CGZZ24,
  author        = {Chen, Xiaoyu and Guo, Heng and Zhang, Xinyuan and Zou, Zongrui},
  title         = {Near-linear time samplers for matroid independent sets with applications},
  booktitle     = {Approximation, Randomization, and Combinatorial Optimization. Algorithms and Techniques ({APPROX/RANDOM} 2024)},
  series        = {Leibniz International Proceedings in Informatics ({LIPIcs})},
  volume        = {317},
  pages         = {32:1--32:12},
  year          = {2024},
  publisher     = {Schloss Dagstuhl -- Leibniz-Zentrum f{\"u}r Informatik},
  doi           = {10.4230/LIPIcs.APPROX/RANDOM.2024.32},
  url           = {https://drops.dagstuhl.de/entities/document/10.4230/LIPIcs.APPROX/RANDOM.2024.32},
  eprint        = {2308.09683},
  archiveprefix = {arXiv}
}

@misc{Che26,
  author = {Chewi, Sinho},
  title  = {Log-concave sampling},
  year   = {2026},
  note   = {\url{https://chewisinho.github.io/main.pdf}. Version: 2026}
}

@article{Opp18,
  author        = {Oppenheim, Izhar},
  title         = {Local spectral expansion approach to high dimensional expanders part {I}: Descent of spectral gaps},
  journal       = {Discrete \& Computational Geometry},
  volume        = {59},
  number        = {2},
  pages         = {293--330},
  year          = {2018},
  doi           = {10.1007/s00454-017-9948-x},
  eprint        = {1709.04431},
  archiveprefix = {arXiv},
  primaryclass  = {math.CO}
}

@article{KO20,
  author        = {Kaufman, Tali and Oppenheim, Izhar},
  title         = {High order random walks: Beyond spectral gap},
  journal       = {Combinatorica},
  volume        = {40},
  number        = {2},
  pages         = {245--281},
  year          = {2020},
  doi           = {10.1007/s00493-019-3847-0},
  eprint        = {1707.02799},
  archiveprefix = {arXiv},
  primaryclass  = {math.CO}
}

@inproceedings{AL20,
  author        = {Alev, Vedat Levi and Lau, Lap Chi},
  title         = {Improved analysis of higher order random walks and applications},
  booktitle     = {Proceedings of the 52nd Annual ACM SIGACT Symposium on Theory of Computing},
  pages         = {1198--1211},
  year          = {2020},
  publisher     = {Association for Computing Machinery},
  doi           = {10.1145/3357713.3384317},
  eprint        = {2001.02827},
  archiveprefix = {arXiv},
  primaryclass  = {cs.DS}
}

@article{ALS25,
  title         = {Improvement of {Wu}'s logarithmic {Sobolev} inequality via the {Poisson-F{\"o}llmer} process},
  author        = {Aryan, Shrey and {L{\'o}pez-Rivera}, Pablo and Shenfeld, Yair},
  journal       = {Electronic Communications in Probability},
  volume        = {30},
  number        = {91},
  pages         = {1--14},
  year          = {2025},
  doi           = {10.1214/25-ECP740},
  eprint        = {2410.06117},
  archiveprefix = {arXiv}
}

@article{Sta89,
  author  = {Stanley, Richard P.},
  title   = {Log-concave and unimodal sequences in algebra, combinatorics, and geometry},
  journal = {Annals of the New York Academy of Sciences},
  volume  = {576},
  number  = {1},
  pages   = {500--535},
  year    = {1989},
  doi     = {10.1111/j.1749-6632.1989.tb16434.x}
}

@book{Bre89,
  author    = {Brenti, Francesco},
  title     = {Unimodal, log-concave, and {P}{\'o}lya frequency sequences in combinatorics},
  series    = {Memoirs of the American Mathematical Society},
  volume    = {81},
  number    = {413},
  publisher = {American Mathematical Society},
  year      = {1989},
  doi       = {10.1090/memo/0413}
}

@article{BBL09,
  author  = {Borcea, Julius and Br{\"a}nd{\'e}n, Petter and Liggett, Thomas M.},
  title   = {Negative dependence and the geometry of polynomials},
  journal = {Journal of the American Mathematical Society},
  volume  = {22},
  number  = {2},
  pages   = {521--567},
  year    = {2009},
  doi     = {10.1090/S0894-0347-08-00618-8}
}

@article{HL72,
  author  = {Heilmann, Ole J. and Lieb, Elliott H.},
  title   = {Theory of monomer-dimer systems},
  journal = {Communications in Mathematical Physics},
  volume  = {25},
  number  = {3},
  pages   = {190--232},
  year    = {1972},
  doi     = {10.1007/BF01877590}
}

@book{Asm03,
  author    = {Asmussen, S{\o}ren},
  title     = {Applied probability and queues},
  edition   = {2nd},
  series    = {Stochastic Modelling and Applied Probability},
  volume    = {51},
  year      = {2003},
  publisher = {Springer},
  address   = {New York, NY},
  doi       = {10.1007/b97236}
}

@article{BHKKL25,
  author        = {Baker, Matthew and Huh, June and Kim, Donggyu and Kummer, Mario and Lorscheid, Oliver},
  title         = {Representation theory for polymatroids},
  journal       = {arXiv preprint arXiv:2507.14718},
  year          = {2025},
  eprint        = {2507.14718},
  archiveprefix = {arXiv}
}

@article{CFYZ24,
  title         = {Rapid mixing of {Glauber} dynamics via spectral independence for all degrees},
  author        = {Chen, Xiaoyu and Feng, Weiming and Yin, Yitong and Zhang, Xinyuan},
  journal       = {SIAM Journal on Computing},
  pages         = {FOCS21{\char45}224--FOCS21{\char45}298},
  year          = {2024},
  doi           = {10.1137/22M1474734},
  eprint        = {2105.15005},
  archiveprefix = {arXiv}
}

@inproceedings{CCYZ25,
  title         = {Rapid mixing at the uniqueness threshold},
  author        = {Chen, Xiaoyu and Chen, Zongchen and Yin, Yitong and Zhang, Xinyuan},
  booktitle     = {Proceedings of the 57th Annual ACM Symposium on Theory of Computing},
  pages         = {879--890},
  year          = {2025},
  doi           = {10.1145/3717823.3718260},
  eprint        = {2411.03413},
  archiveprefix = {arXiv}
}

@inproceedings{AKJVP22,
  author    = {Anari, Nima and Jain, Vishesh and Koehler, Frederic and Pham, Huy Tuan and Vuong, Thuy-Duong},
  title     = {Entropic independence: optimal mixing of down-up random walks},
  booktitle = {Proceedings of the 54th Annual ACM SIGACT Symposium on Theory of Computing},
  pages     = {1418--1430},
  year      = {2022},
  doi       = {10.1145/3519935.3520048}
}

@inproceedings{ALOVV21,
  author        = {Anari, Nima and Liu, Kuikui and {Oveis Gharan}, Shayan and Vinzant, Cynthia and Vuong, Thuy-Duong},
  title         = {Log-concave polynomials {IV}: Approximate exchange, tight mixing times, and near-optimal sampling of forests},
  booktitle     = {Proceedings of the 53rd Annual ACM SIGACT Symposium on Theory of Computing},
  pages         = {408--420},
  year          = {2021},
  doi           = {10.1145/3406325.3451091},
  eprint        = {2004.07220},
  archiveprefix = {arXiv}
}

@article{Wan26,
  author        = {Wang, Sihan},
  title         = {Optimal mixing of {Glauber} dynamics for the {Sherrington--Kirkpatrick} model at {$\beta<1/2$}},
  journal       = {arXiv preprint arXiv:2608.22159},
  year          = {2026},
  eprint        = {2608.22159},
  archiveprefix = {arXiv}
}

@article{GZ26,
  author        = {Guo, Heng and Zhang, Xinyuan},
  title         = {Approximating spin systems on planar graphs},
  journal       = {arXiv preprint arXiv:2608.06172},
  year          = {2026},
  eprint        = {2608.06172},
  archiveprefix = {arXiv}
}

@article{CL26,
  author        = {Chen, Xiaoyu and Liu, Kuikui},
  title         = {A spectral local-to-global principle for spin systems on graphs with girth at least five},
  journal       = {arXiv preprint arXiv:2608.25491},
  year          = {2026},
  eprint        = {2608.25491},
  archiveprefix = {arXiv}
}

@article{LOG25,
  author        = {Leake, Jonathan and {Oveis Gharan}, Shayan},
  title         = {Trickle-down theorems via {$\mathcal{C}$-Lorentzian} polynomials {II}: Pairwise spectral influence and improved {Dobrushin}'s condition},
  journal       = {arXiv preprint arXiv:2510.06549},
  year          = {2025},
  eprint        = {2510.06549},
  archiveprefix = {arXiv}
}

@inproceedings{Gur09a,
  title        = {On multivariate {Newton}-like inequalities},
  author       = {Gurvits, Leonid},
  booktitle    = {Advances in Combinatorial Mathematics: Proceedings of the Waterloo Workshop in Computer Algebra 2008},
  pages        = {61--78},
  year         = {2009},
  organization = {Springer}
}

@article{Pre73,
  title   = {On logarithmic concave measures and functions},
  author  = {Pr{\'e}kopa, Andr{\'a}s},
  journal = {Acta Sci. Math.},
  volume  = {34},
  pages   = {335--343},
  year    = {1973}
}

@article{Gur09b,
  author  = {Gurvits, Leonid},
  title   = {A Short Proof, Based on Mixed Volumes, of {Liggett}'s Theorem on the Convolution of Ultra-Logconcave Sequences},
  journal = {The Electronic Journal of Combinatorics},
  volume  = {16},
  number  = {1},
  pages   = {N5},
  year    = {2009},
  doi     = {10.37236/243}
}

@article{KN11,
  title     = {A strong log-concavity property for measures on {Boolean} algebras},
  author    = {Kahn, Jeff and Neiman, Michael},
  journal   = {Journal of Combinatorial Theory, Series A},
  volume    = {118},
  number    = {6},
  pages     = {1749--1760},
  year      = {2011},
  publisher = {Elsevier}
}

@article{WY07,
  title     = {Log-concavity and {LC}-positivity},
  author    = {Wang, Yi and Yeh, Yeong-Nan},
  journal   = {Journal of Combinatorial Theory, Series A},
  volume    = {114},
  number    = {2},
  pages     = {195--210},
  year      = {2007},
  publisher = {Elsevier}
}

@article{Yu09,
  title   = {Monotonic Convergence in an Information-Theoretic Law of Small Numbers},
  author  = {Yu, Yaming},
  journal = {IEEE Transactions on Information Theory},
  volume  = {55},
  number  = {12},
  pages   = {5412--5422},
  year    = {2009},
  doi     = {10.1109/TIT.2009.2032727}
}

@inproceedings{YJ09,
  title     = {Concavity of Entropy under Thinning},
  author    = {Yu, Yaming and Johnson, Oliver},
  booktitle = {2009 IEEE International Symposium on Information Theory},
  pages     = {144--148},
  year      = {2009},
  doi       = {10.1109/ISIT.2009.5205880},
  publisher = {IEEE}
}

@article{AOGV21,
  title   = {Log-Concave Polynomials {I}: Entropy and a Deterministic Approximation Algorithm for Counting Bases of Matroids},
  author  = {Anari, Nima and {Oveis Gharan}, Shayan and Vinzant, Cynthia},
  journal = {Duke Mathematical Journal},
  volume  = {170},
  number  = {16},
  pages   = {3459--3504},
  year    = {2021},
  doi     = {10.1215/00127094-2020-0091}
}

\appendix

\section{Deferred Proofs}

\subsection{\Ito Calculus for Poisson Stochastic Localization}
\label{subsec:proof-ito-calculus}

\begin{proof}[Proof of \Cref{thm:poisson-linear-tilt}]
    Write
    \[
    \*m_t \defeq \mean{\nu_t}, \qquad \*a_t \defeq \*m_t - \*X_t.
    \]
    The conditional-expectation formula for the birth rates (\Cref{cor:denoising-birth-rates-conditional-expectation}) gives
    \[
    q_{t, i}\tp{\*X_{t-}} = \frac{a_{t-, i}}{1 - t}.
    \]
    Fix the current state \(\*X_{t-} = \*x\). As long as \(\*X_t\) does not jump, differentiating the explicit posterior formula gives
    \[
    \partial_t \nu_{t, \*x}\tp{\*n} = -\frac{\nu_{t, \*x}\tp{\*n}}{1 - t} \sum_{i = 1}^d \tp{n_i - m_{t, i}} = -\nu_{t, \*x}\tp{\*n} \sum_{i = 1}^d \frac{n_i - m_{t, i}}{a_{t-, i}} q_{t, i}\tp{\*X_{t-}}.
    \]
    If \(\*X_t\) jumps from \(\*x\) to \(\*x + \*e_i\), then Bayes' rule gives
    \[
    \nu_{t, \*x + \*e_i}\tp{\*n} = \nu_{t-, \*x}\tp{\*n} \frac{n_i - x_i}{m_{t-, i} - x_i},
    \]
    and therefore
    \[
    \nu_{t, \*x + \*e_i}\tp{\*n} - \nu_{t-, \*x}\tp{\*n} = \nu_{t-, \*x}\tp{\*n} \frac{n_i - m_{t-, i}}{a_{t-, i}}.
    \]
    Combining the continuous evolution and the jump updates yields
    \[
    \dd \nu_t\tp{\*n} = \nu_{t-}\tp{\*n} \sum_{i = 1}^d \frac{n_i - m_{t-, i}}{a_{t-, i}} \tp{\-d X_{t, i} - q_{t, i}\tp{\*X_{t-}} \dd t}.
    \]
    Here and below, a summand with \(a_{t-, i} = 0\) is set equal to zero. Indeed, \(a_{t-, i} = 0\) implies that \(\nu_{t-}\) is supported on \(\set{\*n \in \bb N^d \cmid n_i = X_{t-, i}}\), so the \(i\)-th birth rate and the corresponding numerator both vanish. This proves the asserted linear-tilt equation.

    Finally, for \(T < 1\),
    \[
    \int_0^T \frac{\*1_{\set{a_{t-, i} > 0}}}{a_{t-, i}} q_{t, i}\tp{\*X_{t-}} \dd t = \int_0^T \frac{\*1_{\set{a_{t-, i} > 0}}}{1 - t} \dd t \le -\log\tp{1 - T}.
    \]
    Thus the stochastic integral defining \(Z_{t, i}\) is integrable and is a martingale on \(\stp{0, T}\).
\end{proof}

\begin{proof}[Proof of \Cref{thm:phi-entropy-evolution}]
    Write
    \[
    \*m_t \defeq \mean{\nu_t}, \qquad \*a_t \defeq \*m_t - \*X_t, \qquad F_t \defeq \*E_{\nu_t}\stp{f}, \qquad G_t \defeq \*E_{\nu_t}\stp{\Phi\tp{f}},
    \]
    and define
    \[
    c_{t, i} \defeq
    \begin{cases}
        \displaystyle \frac{\*{Cov}_{\*n \sim \nu_t}\stp{f\tp{\*n}, n_i}}{a_{t, i}}, & a_{t, i} > 0, \\
        0, & a_{t, i} = 0.
    \end{cases}
    \]
    Both \(F_t\) and \(G_t\) are martingales, and \Cref{thm:poisson-linear-tilt} gives
    \[
    \dd F_t = \sum_{i = 1}^d c_{t-, i} \tp{\-d X_{t, i} - q_{t, i}\tp{\*X_{t-}} \dd t}.
    \]
    When \(\*X_t\) jumps in coordinate \(i\), the corresponding jump of \(F_t\) is \(c_{t-, i}\). The jump-process \Ito formula therefore yields
    \[
    \dd \Phi\tp{F_t} = \Phi'\tp{F_{t-}} \dd F_t + \sum_{i = 1}^d \Psi\tp{F_{t-}, c_{t-, i}} \dd X_{t, i}.
    \]
    Taking expectations and using the compensators of the coordinate birth processes gives
    \[
    \dd \*E\stp{\Phi\tp{F_t}} = \*E\stp{\sum_{i = 1}^d q_{t, i}\tp{\*X_t} \Psi\tp{F_t, c_{t, i}}} \dd t = \frac{1}{1 - t} \*E\stp{\sum_{i = 1}^d a_{t, i} \Psi\tp{F_t, c_{t, i}}} \dd t,
    \]
    where the second equality uses the conditional-expectation formula \(q_{t, i}\tp{\*X_t} = a_{t, i} / \tp{1 - t}\) for the birth rates (\Cref{cor:denoising-birth-rates-conditional-expectation}). Since \(\*{Ent}_{\nu_t}^{\Phi}\stp{f} = G_t - \Phi\tp{F_t}\) and \(\*E\stp{G_t}\) is constant, this proves the general \(\Phi\)-entropy identity.

    For \(\Phi\tp{x} = x^2\), we have \(\Psi\tp{u, v} = v^2\); substituting it gives the variance identity. For \(\Phi\tp{x} = x \log x\), the homogeneity identity
    \[
    a \, \Psi\tp{u, \frac{v}{a}} = u \, \Psi\tp{a, \frac{v}{u}}
    \]
    holds for \(a, u > 0\). Applying it with \(a \gets a_{t, i}\), \(u \gets F_t\), and \(v \gets \*{Cov}_{\*n \sim \nu_t}\stp{f\tp{\*n}, n_i}\) gives the entropy identity.
\end{proof}

\subsection{Spectral Independence, Variance, and Entropy}
\label{subsec:proof-sufficient-condition-approximate-conservation}

\begin{lemma}
    For \(t \in \left[0, 1\right)\) and \(\*x \in \supp\tp{\mu_t}\), define \(\rho_t^{\*x} \defeq \-{Law}\tp{\*X_1 - \*X_t \mid \*X_t = \*x}\). Then
    \[
    \rho_t^{\*x} = \tp{1 - t} * \mu^{\*x}.
    \]
    \label{lem:tilted-pinning-conditional-law}
\end{lemma}

\begin{proof}[Proof of \Cref{lem:tilted-pinning-conditional-law}]
    For \(t \in \tp{0, 1}\), by the posterior formula \eqref{eq:posterior-formula}, for every \(\*k \in \bb N^d\) such that \(\*x + \*k \in \supp\tp{\mu}\),
    \[
    \rho_t^{\*x}\tp{\*k} = \nu_{t, \*x}\tp{\*x + \*k} = \frac{t^{\norm{\*x}_1}}{\mu_t\tp{\*x}} \tp{1 - t}^{\norm{\*k}_1} \mu\tp{\*x + \*k} \prod_{i = 1}^d \binom{x_i + k_i}{x_i}.
    \]
    The factor \(t^{\norm{\*x}_1} / \mu_t\tp{\*x}\) does not depend on \(\*k\). Comparing the remaining expression with \Cref{def:weighted-pinning} therefore gives \(\rho_t^{\*x} = \tp{1 - t} * \mu^{\*x}\). The case \(t = 0\) follows directly from \(\*X_0 = \*0\).
\end{proof}

We first prove approximate conservation of variance. The only additional ingredient beyond the variance evolution identity in \Cref{thm:phi-entropy-evolution} is a covariance form of the Cauchy--Schwarz inequality.

\begin{proof}[Proof of \Cref{prop:sufficient-condition-approximate-conservation} (Variance)]
    Fix \(t \in \tp{0, 1}\) and condition on \(\*X_t = \*x\). Write \(\rho \defeq \rho_t^{\*x}\) and \(g\tp{\*k} \defeq f\tp{\*x + \*k}\), and set
    \[
    \*m \defeq \mean{\rho}, \qquad \*c \defeq \*{Cov}_{\*k \sim \rho}\stp{g\tp{\*k}, \*k}, \qquad \*C \defeq \cov{\rho}.
    \]
    If \(m_i = 0\), then the coordinate map \(\*k \mapsto k_i\) vanishes \(\rho\)-almost surely and hence \(c_i = 0\). We may therefore restrict to the coordinates on which \(m_i > 0\) and set \(\*z \defeq \diag\tp{\*m}^{-1} \*c\). The Cauchy--Schwarz inequality and spectral independence yield
    \[
    \tp{\*z^{\top} \*c}^2 = \*{Cov}_{\*k \sim \rho}\stp{g\tp{\*k}, \*z^{\top} \*k}^2 \le \*{Var}_{\rho}\stp{g} \*z^{\top} \*C \*z.
    \]
    By \Cref{lem:tilted-pinning-conditional-law}, \(\rho = \tp{1 - t} * \mu^{\*x}\), so the spectral independence hypothesis gives \(\*C \preceq \eta_t \diag\tp{\*m}\). Hence,
    \[
    \tp{\*z^{\top} \*c}^2 \le \eta_t \, \*{Var}_{\rho}\stp{g} \*z^{\top} \diag\tp{\*m} \*z = \eta_t \, \*{Var}_{\rho}\stp{g} \*z^{\top} \*c.
    \]
    Dividing by \(\*z^{\top} \*c\) when it is positive, with the zero case being immediate, gives
    \[
    \sum_{i = 1}^d \frac{\*{Cov}_{\*n \sim \nu_{t, \*x}}\stp{f\tp{\*n}, n_i}^2}{\mean{\nu_{t, \*x}}_i - x_i} = \sum_{i = 1}^d \frac{c_i^2}{m_i} = \*z^{\top} \*c \le \eta_t \, \*{Var}_{\nu_{t, \*x}}\stp{f}.
    \]
    Applying the variance identity in \Cref{thm:phi-entropy-evolution} and averaging over \(\*x \sim \mu_t\) now gives
    \[
    \partial_t \*E\stp{\*{Var}_{\nu_t}\stp{f}} \ge -\frac{\eta_t}{1 - t} \*E\stp{\*{Var}_{\nu_t}\stp{f}}.
    \]
    The same argument at \(t = 0\), where \(\*X_0 = \*0\) and \(\rho_0^{\*0} = \mu\), gives the asserted right-derivative inequality at the endpoint.
\end{proof}

We next prove approximate conservation of entropy. We use two standard facts about exponential families, followed by the notion of entropic stability.

\begin{theorem}[Information Projection Principle; see, e.g., \cite{WJ08}]
    Let \(\rho\) be a probability measure on \(\Omega \subseteq \bb R^d\), let \(\boldsymbol{\varphi}: \Omega \to \bb R^d\) be measurable, and define
    \[
    \chi_{\rho, \boldsymbol{\varphi}}\tp{\boldsymbol{\theta}} \defeq \log \*E_{\rho}\stp{e^{\inner{\boldsymbol{\theta}}{\boldsymbol{\varphi}}}}.
    \]
    For \(\boldsymbol{\theta}\) in the interior of the effective domain of \(\chi_{\rho, \boldsymbol{\varphi}}\), let \(\rho_{\boldsymbol{\theta}}\) be the exponential tilt
    \[
    \frac{\dd \rho_{\boldsymbol{\theta}}}{\dd \rho}\tp{\*x} = \exp\tp{\inner{\boldsymbol{\theta}}{\boldsymbol{\varphi}\tp{\*x}} - \chi_{\rho, \boldsymbol{\varphi}}\tp{\boldsymbol{\theta}}}.
    \]
    If \(\*m = \*E_{\rho_{\boldsymbol{\theta}}}\stp{\boldsymbol{\varphi}}\), then every \(\wt{\rho} \ll \rho\) satisfying \(\*E_{\wt{\rho}}\stp{\boldsymbol{\varphi}} = \*m\) obeys
    \[
    \-{KL}\tp{\wt{\rho} \| \rho} = \-{KL}\tp{\wt{\rho} \| \rho_{\boldsymbol{\theta}}} + \-{KL}\tp{\rho_{\boldsymbol{\theta}} \| \rho} \ge \-{KL}\tp{\rho_{\boldsymbol{\theta}} \| \rho}.
    \]
    Thus \(\rho_{\boldsymbol{\theta}}\) is the unique minimizer of relative entropy among probability measures with mean statistic \(\*m\). Means on the relative boundary of the mean domain are obtained by limits of exponential tilts.
    \label{thm:maximum-entropy-principle}
\end{theorem}

\begin{theorem}[Cumulants and Entropy; see, e.g., \cite{WJ08}]
    In the setting of \Cref{thm:maximum-entropy-principle}, the log-Laplace transform is smooth and convex on the interior of its effective domain and satisfies
    \[
    \nabla \chi_{\rho, \boldsymbol{\varphi}}\tp{\boldsymbol{\theta}} = \*E_{\rho_{\boldsymbol{\theta}}}\stp{\boldsymbol{\varphi}}, \qquad \nabla^2 \chi_{\rho, \boldsymbol{\varphi}}\tp{\boldsymbol{\theta}} = \*{Cov}_{\rho_{\boldsymbol{\theta}}}\stp{\boldsymbol{\varphi}}.
    \]
    If the exponential family is minimal, then \(\chi_{\rho, \boldsymbol{\varphi}}\) is strictly convex. For every mean \(\*m\) in the interior of the mean domain, let \(\boldsymbol{\theta}\tp{\*m}\) be the unique parameter satisfying \(\nabla \chi_{\rho, \boldsymbol{\varphi}}\tp{\boldsymbol{\theta}\tp{\*m}} = \*m\). Let \(\chi_{\rho, \boldsymbol{\varphi}}^*\) denote the convex conjugate of \(\chi_{\rho, \boldsymbol{\varphi}}\). Then
    \[
    \chi_{\rho, \boldsymbol{\varphi}}^*\tp{\*m} = \-{KL}\tp{\rho_{\boldsymbol{\theta}\tp{\*m}} \| \rho}, \qquad \nabla \chi_{\rho, \boldsymbol{\varphi}}^*\tp{\*m} = \boldsymbol{\theta}\tp{\*m},
    \]
    and, whenever the covariance is nonsingular,
    \[
    \nabla^2 \chi_{\rho, \boldsymbol{\varphi}}^*\tp{\*m} = \*{Cov}_{\rho_{\boldsymbol{\theta}\tp{\*m}}}\stp{\boldsymbol{\varphi}}^{-1}.
    \]
    In the nonminimal case, the same identities hold after restriction to the affine hull of \(\boldsymbol{\varphi}\tp{\supp\tp{\rho}}\).
    \label{thm:log-Laplace-transform}
\end{theorem}

\begin{remark}
    When \(\boldsymbol{\varphi}\tp{\*x} = \*x\), 
    \[
    \dom \boldsymbol{\theta} = \operatorname{relint}\tp{\dom \chi_{\rho, \boldsymbol{\varphi}}^*} \subseteq \operatorname{relint}\tp{\conv\tp{\Omega}}.
    \]
    Boundary means are handled by approximation with interior means.
\end{remark}

We now introduce the form of entropic stability needed for the entropy evolution argument.

\begin{definition}[{Entropic Stability \cite[Definition~29]{CE25}}]
    Let \(\rho\) be a probability measure on \(\Omega\), let \(\psi: \conv \Omega \times \conv \Omega \to \bb R_{\ge 0}\), and let \(\eta > 0\). We say that \(\rho\) is \(\eta\)-entropically stable with respect to \(\psi\) if
    \[
    \psi\tp{\mean{\boldsymbol{\lambda} * \rho}, \mean{\rho}} \le \eta \, \-{KL}\tp{\boldsymbol{\lambda} * \rho \| \rho}, \quad \forall \boldsymbol{\lambda} \in \bb R_{> 0}^d.
    \]
    Unless otherwise specified, we take
    \[
    \psi\tp{\*x, \*y} \defeq \sum_{i = 1}^d \Psi\tp{y_i, x_i - y_i},
    \]
    where \(\Psi\tp{u, v} \defeq \tp{u + v} \log\tp{u + v} - u \log u - v \tp{\log u + 1}\) and \(\Psi\tp{0, 0} \defeq 0\).
\end{definition}

\begin{lemma}
    Suppose that \(\rho\) is \(\eta\)-entropically stable with respect to \(\psi\). Then every probability measure \(\wt{\rho} \ll \rho\) satisfies
    \[
    \psi\tp{\mean{\wt{\rho}}, \mean{\rho}} \le \eta \, \-{KL}\tp{\wt{\rho} \| \rho}.
    \]
    \label{lem:entropic-stability-tilted}
\end{lemma}

\begin{proof}[Proof of \Cref{lem:entropic-stability-tilted}]
    Set \(\*m_0 \defeq \mean{\rho}\) and \(\*m \defeq \mean{\wt{\rho}}\). We may assume that \(\-{KL}\tp{\wt{\rho} \| \rho} < \infty\), since otherwise the claim is immediate. First suppose that \(\*m\) lies in the interior of the mean domain. Let \(\rho_{\boldsymbol{\theta}\tp{\*m}}\) be the exponential tilt of \(\rho\) with mean \(\*m\). By \Cref{thm:maximum-entropy-principle} and entropic stability,
    \[
    \psi\tp{\*m, \*m_0} \le \eta \, \-{KL}\tp{\rho_{\boldsymbol{\theta}\tp{\*m}} \| \rho} \le \eta \, \-{KL}\tp{\wt{\rho} \| \rho}.
    \]
    For a general mean, fix \(\eps \in \tp{0, 1}\) and define \(\wt{\rho}_{\eps} \defeq \tp{1 - \eps} \wt{\rho} + \eps \rho\). Its mean \(\*m_{\eps} = \tp{1 - \eps} \*m + \eps \*m_0\) lies in the relative interior of the mean domain. Applying the preceding argument and using convexity of relative entropy gives
    \[
    \psi\tp{\*m_{\eps}, \*m_0} \le \eta \, \-{KL}\tp{\wt{\rho}_{\eps} \| \rho} \le \eta \tp{1 - \eps} \-{KL}\tp{\wt{\rho} \| \rho}.
    \]
    Letting \(\eps \to 0^+\) and using lower semicontinuity of \(\psi\) proves the claim.
\end{proof}

Entropic stability controls exactly the nonlinear covariance term in the entropy evolution identity.

\begin{lemma}
    Fix \(t \in \tp{0, 1}\). Suppose that \(\rho_t^{\*x}\) is \(\eta\)-entropically stable with respect to \(\psi\) for every \(\*x \in \supp\tp{\mu_t}\). Then every nonnegative \(f \in \+C_{\-c}\tp{\supp\tp{\mu}}\) satisfies
    \[
    \partial_t \*E\stp{\*{Ent}_{\nu_t}\stp{f}} \ge -\frac{\eta}{1 - t} \*E\stp{\*{Ent}_{\nu_t}\stp{f}}.
    \]
    \label{lem:entropic-stability-implies-entropy-conservation}
\end{lemma}

\begin{proof}[Proof of \Cref{lem:entropic-stability-implies-entropy-conservation}]
    Condition on \(\*X_t = \*x\), and write \(\rho \defeq \rho_t^{\*x}\), \(g\tp{\*k} \defeq f\tp{\*x + \*k}\). If \(\*E_{\rho}\stp{g} = 0\), then \(g = 0\) \(\rho\)-almost surely, so both the conditional entropy and the corresponding term in the entropy evolution identity vanish. Suppose that \(\*E_{\rho}\stp{g} > 0\), and define the \(g\)-tilt \(\wt{\rho}\) by
    \[
    \frac{\dd \wt{\rho}}{\dd \rho}\tp{\*k} \defeq \frac{g\tp{\*k}}{\*E_{\rho}\stp{g}}.
    \]
    Then
    \[
    \mean{\wt{\rho}} - \mean{\rho} = \frac{\*{Cov}_{\*k \sim \rho}\stp{g\tp{\*k}, \*k}}{\*E_{\rho}\stp{g}},
    \]
    while
    \[
    \*E_{\rho}\stp{g} \, \-{KL}\tp{\wt{\rho} \| \rho} = \*{Ent}_{\rho}\stp{g} = \*{Ent}_{\nu_{t, \*x}}\stp{f}.
    \]
    Therefore, \Cref{lem:entropic-stability-tilted} gives
    \[
    \begin{aligned}
        \*E_{\rho}\stp{g} \sum_{i = 1}^d \Psi\tp{\mean{\rho}_i, \frac{\*{Cov}_{\*k \sim \rho}\stp{g\tp{\*k}, k_i}}{\*E_{\rho}\stp{g}}} &= \*E_{\rho}\stp{g} \, \psi\tp{\mean{\wt{\rho}}, \mean{\rho}} \\
        &\le \eta \, \*E_{\rho}\stp{g} \, \-{KL}\tp{\wt{\rho} \| \rho} \\
        &= \eta \, \*{Ent}_{\nu_{t, \*x}}\stp{f}.
    \end{aligned}
    \]
    Since \(\mean{\rho}_i = \mean{\nu_{t, \*x}}_i - x_i\), substituting this estimate into the entropy identity in \Cref{thm:phi-entropy-evolution} and averaging over \(\*x \sim \mu_t\) yields
    \[
    \partial_t \*E\stp{\*{Ent}_{\nu_t}\stp{f}} \ge -\frac{\eta}{1 - t} \*E\stp{\*{Ent}_{\nu_t}\stp{f}}.
    \]
\end{proof}

Spectral independence of every exponential tilt provides a convenient criterion for entropic stability.

\begin{lemma}
    Suppose that \(\boldsymbol{\lambda} * \rho\) is \(\eta\)-spectrally independent for every \(\boldsymbol{\lambda} \in \bb R_{> 0}^d\). Then \(\rho\) is \(\eta\)-entropically stable with respect to \(\psi\).
    \label{lem:spectral-independence-implies-entropic-stability}
\end{lemma}

\begin{proof}[Proof of \Cref{lem:spectral-independence-implies-entropic-stability}]
    We may discard coordinates that vanish \(\rho\)-almost surely and work on the affine hull of \(\supp\tp{\rho}\). For \(\boldsymbol{\theta} \in \bb R^d\), define
    \[
    \chi\tp{\boldsymbol{\theta}} \defeq \log \*E_{\*k \sim \rho}\stp{e^{\inner{\boldsymbol{\theta}}{\*k}}}, \qquad \rho_{\boldsymbol{\theta}} \defeq e^{\boldsymbol{\theta}} * \rho,
    \]
    and write \(\*m_0 \defeq \mean{\rho}\). For \(\*m\) in the interior of the mean domain, let \(\boldsymbol{\theta}\tp{\*m}\) be the parameter satisfying \(\mean{\rho_{\boldsymbol{\theta}\tp{\*m}}} = \*m\), and set
    \[
    H\tp{\*m} \defeq \-{KL}\tp{\rho_{\boldsymbol{\theta}\tp{\*m}} \| \rho}, \qquad G\tp{\*m} \defeq \psi\tp{\*m, \*m_0}.
    \]
    By \Cref{thm:log-Laplace-transform},
    \[
    \nabla^2 H\tp{\*m} = \cov{\rho_{\boldsymbol{\theta}\tp{\*m}}}^{-1}.
    \]
    On the other hand, 
    \[
    G\tp{\*m} = \sum_{i = 1}^d \stp{m_i \log\tp{\frac{m_i}{m_{0, i}}} - m_i + m_{0, i}}, \qquad \nabla^2 G\tp{\*m} = \diag\tp{\*m}^{-1}.
    \]
    Spectral independence of \(\rho_{\boldsymbol{\theta}\tp{\*m}}\) gives
    \[
    \cov{\rho_{\boldsymbol{\theta}\tp{\*m}}} \preceq \eta \diag\tp{\mean{\rho_{\boldsymbol{\theta}\tp{\*m}}}} = \eta \diag\tp{\*m}.
    \]
    Hence,
    \[
    \eta \nabla^2 H\tp{\*m} - \nabla^2 G\tp{\*m} \succeq \*O.
    \]
    Here and above, inverse matrices are interpreted on the affine hull; equivalently, the Hessian comparison follows from the corresponding generalized-inverse inequality on its tangent space. Thus \(\eta H - G\) is convex. Note that \(\mean{\rho_{\boldsymbol{\theta}\tp{\*m_0}}} = \*m_0 = \mean{\rho}\), so \(\boldsymbol{\theta}\tp{\*m_0} = 0\). Together with \Cref{thm:log-Laplace-transform}, we have
    \[
    H\tp{\*m_0} = \-{KL}\tp{\rho_{\boldsymbol{\theta}\tp{\*m_0}} \| \rho} = 0 = G\tp{\*m_0},
    \]
    \[
    \nabla H\tp{\*m_0} = \boldsymbol{\theta}\tp{\*m_0} = 0 = \nabla G\tp{\*m_0}.
    \]
    It follows that \(G\tp{\*m} \le \eta H\tp{\*m}\) throughout the interior of the mean domain. In particular, this holds for \(\*m = \mean{\boldsymbol{\lambda} * \rho}\) for every \(\boldsymbol{\lambda} \in \bb R_{> 0}^d\), which is precisely \(\eta\)-entropic stability.
\end{proof}

We can now combine the preceding lemmas to obtain the entropy assertion of \Cref{prop:sufficient-condition-approximate-conservation}.

\begin{proof}[Proof of \Cref{prop:sufficient-condition-approximate-conservation} (Entropy)]
    Fix \(t \in \tp{0, 1}\) and \(\*x \in \supp\tp{\mu_t}\). By \Cref{lem:tilted-pinning-conditional-law},
    \[
    \rho_t^{\*x} = \tp{1 - t} * \mu^{\*x}.
    \]
    Consequently, for every \(\boldsymbol{\lambda} \in \bb R_{> 0}^d\),
    \[
    \boldsymbol{\lambda} * \rho_t^{\*x} = \tp{\tp{1 - t} \boldsymbol{\lambda}} * \mu^{\*x}.
    \]
    Note that
    \[
    \*x \in \supp\tp{\mu_t} = \set{\*x \in \bb N^d \cmid \exists \*x' \in \supp\tp{\mu}, \, \*x \le \*x'} = \Omega.
    \]
    The hypothesis of the proposition therefore implies that every exponential tilt of \(\rho_t^{\*x}\) is \(\eta\)-spectrally independent. By \Cref{lem:spectral-independence-implies-entropic-stability}, \(\rho_t^{\*x}\) is \(\eta\)-entropically stable. Applying \Cref{lem:entropic-stability-implies-entropy-conservation} gives
    \[
    \partial_t \*E\stp{\*{Ent}_{\nu_t}\stp{f}} \ge -\frac{\eta}{1 - t} \*E\stp{\*{Ent}_{\nu_t}\stp{f}}, \quad \forall t \in \tp{0, 1}.
    \]
    The same argument gives the asserted right-derivative inequality at \(t = 0\).
\end{proof}

\subsection{Homogenization for ULC Measures and Lorentzian Polynomials}
\label{subsec:ulc-completely-log-concavity}

In this subsection, we introduce a homogenization procedure that connects ULC measures with Lorentzian polynomials. We then use this construction to establish the implication \(1 \Rightarrow 5\) in \Cref{thm:ulc-generating-function-characterization}.

\begin{definition}
    Let \(\mu\) be a finitely supported probability measure on \(\bb N^d\), and let \(N \in \bb N\) satisfy
    \[
    N \ge \max_{\*n \in \supp\tp{\mu}} \norm{\*n}_1.
    \]
    The homogenized generating polynomial of \(\mu\) is defined by
    \[
    \+H_N g_{\mu}\tp{\theta_0, \boldsymbol{\theta}} \defeq \sum_{\*n \in \supp\tp{\mu}} \frac{\mu\tp{\*n}}{\tp{N - \norm{\*n}_1}!} \theta_0^{N - \norm{\*n}_1} \boldsymbol{\theta}^{\*n}, \quad \tp{\theta_0, \boldsymbol{\theta}} \in \bb R^{d + 1}.
    \]
    \label{def:homogenized-generating-polynomial}
\end{definition}

\begin{lemma}
    Let \(\mu\) be supported on a finite downward closed subset of \(\bb N^d\). Then \(\mu\) is ULC if and only if, for every \(N \in \bb N\) satisfying \(N \ge \max_{\*n \in \supp\tp{\mu}} \norm{\*n}_1\), the homogenized generating polynomial \(\+H_N g_{\mu}\) is Lorentzian.
    \label{lem:ulc-homogenization-lorentzian}
\end{lemma}

\begin{proof}[Proof of \Cref{lem:ulc-homogenization-lorentzian}]
    Write \(a\tp{\*n} \defeq \*n! \mu\tp{\*n}\). Suppose first that \(\mu\) is ULC. By \Cref{thm:ulc-generating-function-characterization}, \(g_{\mu}\) is strongly log-concave. Fix a nonzero derivative \(h = \partial^{\*x} g_{\mu}\). By log-concavity, pointwise on \(\bb R_{> 0}^d\),
    \[
    \*A \defeq \nabla h \nabla h^{\top} - h \nabla^2 h \succeq \*O.
    \]
    In particular,
    \[
    A_{ii} = \tp{\partial_i h}^2 - h \partial_{ii} h \ge 0, \quad \forall i \in \stp{d}.
    \]
    Since \(h\) has nonnegative coefficients, \(h, \partial_i h,\partial_{ii} h \ge 0\), and hence \(0 \le \sqrt{A_{ii}} \le \partial_i h\). Moreover, positive semidefiniteness of \(\*A\) gives
    \[
    \abs{A_{ij}} \le \sqrt{A_{ii} A_{jj}} \le \partial_i h \, \partial_j h, \quad \forall i, j \in \stp{d}, \, i \ne j.
    \]
    Thus, for \(i \ne j\),
    \[
    \partial_i h\,\partial_j h \ge \abs{A_{ij}} \ge -A_{ij} = h \partial_{ij} h - \partial_i h \,\partial_j h,
    \]
    and therefore \(h \partial_{ij} h \le 2 \partial_i h \, \partial_j h\). For \(i = j\), the stronger inequality
    \[
    h \partial_{ii} h \le \tp{\partial_i h}^2
    \]
    follows directly from \(A_{ii} \ge 0\). Hence, \(g_{\mu}\) is \(2\)-Rayleigh \cite[Definition~2.18]{BH20}. Since \(\supp\tp{\mu}\) is downward closed, \(g_{\mu}\tp{\*0} = \mu\tp{\*0} > 0\). Hence, \cite[Lemma~2.22]{BH20} implies that \(\supp\tp{\mu}\) is \(M^{\natural}\)-convex. Consequently,
    \[
    \+J_N \defeq \set{\tp{N - \norm{\*n}_1, \*n} \cmid \*n \in \supp\tp{\mu}}
    \]
    is \(M\)-convex.

    For \(\*n \in \supp\tp{\mu}\) with \(\norm{\*n}_1 \le N - 2\), the corresponding quadratic derivative is
    \[
    \partial_0^{N - \norm{\*n}_1 - 2} \partial^{\*n} \+H_N g_{\mu}\tp{\theta_0, \boldsymbol{\theta}} = \frac{a\tp{\*n}}{2} \theta_0^2 + \sum_{i = 1}^d a\tp{\*n + \*e_i} \theta_0 \theta_i + \frac{1}{2} \sum_{i, j = 1}^d a\tp{\*n + \*e_i + \*e_j} \theta_i \theta_j.
    \]
    Its Hessian is
    \[
    \begin{pmatrix}
        a\tp{\*n} & \tp{a\tp{\*n + \*e_i}}_{i \in \stp{d}}^{\top} \\
        \tp{a\tp{\*n + \*e_i}}_{i \in \stp{d}} & \tp{a\tp{\*n + \*e_i + \*e_j}}_{i, j \in \stp{d}}
    \end{pmatrix}.
    \]
    Since \(a\tp{\*n} > 0\), the Schur complement shows that this Hessian has at most one positive eigenvalue if and only if
    \[
    \tp{\frac{a\tp{\*n + \*e_i} a\tp{\*n + \*e_j}}{a\tp{\*n}} - a\tp{\*n + \*e_i + \*e_j}}_{i, j \in \+I_{\*n}} \succeq \*O,
    \]
    where \(\+I_{\*n} \defeq \set{i \in \stp{d} \cmid \*n + \*e_i \in \supp\tp{\mu}}\). This follows directly from the ULC matrix condition in \eqref{eq:ulc-condition} by a congruence transformation. Hence every quadratic derivative has at most one positive eigenvalue. Together with the \(M\)-convexity of \(\+J_N\), \cite[Theorem~2.25]{BH20} implies that \(\+H_N g_{\mu}\) is Lorentzian.

    Conversely, suppose that \(\+H_N g_{\mu}\) is Lorentzian. Its quadratic derivatives have at most one positive eigenvalue, so the same Schur-complement calculation gives the ULC matrix condition at every \(\*n\) with \(\norm{\*n}_1 \le N - 2\). The remaining boundary cases are automatic because no second increment lies in the support. Thus \(\mu\) is ULC.
\end{proof}

\begin{proof}[Proof of \Cref{thm:ulc-generating-function-characterization} (\(1 \Rightarrow 5\))]
    First suppose that \(\mu\) has finite support. For every sufficiently large \(N\), define
    \[
    F_{\mu, N}\tp{\boldsymbol{\theta}} \defeq N! \, \+H_N g_{\mu}\tp{1, \boldsymbol{\theta} / N} = \sum_{\*n \in \supp\tp{\mu}} \mu\tp{\*n} \frac{N!}{\tp{N - \norm{\*n}_1}! N^{\norm{\*n}_1}} \boldsymbol{\theta}^{\*n}.
    \]
    By the equivalence between Lorentzian and completely log-concave polynomials \cite[Theorem~2.30]{BH20}, \(\+H_N g_{\mu}\) is completely log-concave. Since complete log-concavity is preserved under restriction to \(\theta_0 = 1\), positive diagonal rescaling of the remaining variables, and multiplication by a positive constant, it follows that \(F_{\mu, N}\) is completely log-concave. As \(N \to \infty\), \(F_{\mu, N}\) and all their directional derivatives converge pointwise to those of \(g_{\mu}\). Log-concavity of every derivative passes to the limit, so \(g_{\mu}\) is completely log-concave.

    For general \(\mu\), let \(\mu_N\) denote its normalized restriction to \(\set{\*n \in \bb N^d \cmid \norm{\*n}_1 \le N}\). It is straightforward to check that \(\mu_N\) is ULC. By the preceding argument, \(g_{\mu_N}\) is completely log-concave. Note that \(g_{\mu_N}\) and all its directional derivatives converge pointwise to those of \(g_{\mu}\). Log-concavity of every derivative passes to the limit, so \(g_{\mu}\) is completely log-concave.
\end{proof}

\subsection{Rapid Mixing for the Antiferromagnetic \texorpdfstring{\(q\)}{q}-Potts Model}
\label{subsec:proofs-antiferromagnetic-potts}

\begin{proof}[Proof of \Cref{cor:q-potts-rapid-mixing}]
    For \(v \in V\) and \(c \in \stp{q}\), define the local heat-bath probabilities by
    \[
    p_c^{\sigma}\tp{v} \defeq \frac{B^{d_c^{\sigma}\tp{v}}}{\sum_{a = 1}^q B^{d_a^{\sigma}\tp{v}}}, \quad v \in V, \, c \in \stp{q}.
    \]
    The Glauber Dirichlet form is
    \[
    \+E_{P_{\-{GD}}}\tp{f, f} = \frac{1}{2 n} \*E_{\sigma \sim \mu_{G, B}}\stp{\sum_{v \in V} \sum_{c = 1}^q p_c^{\sigma}\tp{v} \tp{f\tp{\sigma^{v \gets c}} - f\tp{\sigma}}^2}.
    \]
    By detailed balance, the two orientations of every update between color \(q\) and a color \(c \in \stp{q - 1}\) contribute equally. Pairing these orientations and discarding updates between two nonreference colors gives
    \[
    \+E_{P_{\-{GD}}}\tp{f, f} \ge \frac{1}{n} \*E_{\sigma \sim \mu_{G, B}}\stp{\sum_{\substack{v \in V \\ \sigma_v = q}} \sum_{c = 1}^{q - 1} p_c^{\sigma}\tp{v} \tp{f\tp{\sigma^{v \gets c}} - f\tp{\sigma}}^2}.
    \]
    Under the downward closed embedding, \eqref{eq:birth-death-generator} gives
    \[
    \+E\tp{f, f} = \*E_{\sigma \sim \mu_{G, B}}\stp{\sum_{\substack{v \in V \\ \sigma_v = q}} \sum_{c = 1}^{q - 1} B^{d_c^{\sigma}\tp{v} - d_q^{\sigma}\tp{v}} \tp{f\tp{\sigma^{v \gets c}} - f\tp{\sigma}}^2},
    \]
    where \(\+E\) is the Dirichlet form of the embedded linear-death generator. Note that
    \[
    \frac{B^{d_c^{\sigma}\tp{v} - d_q^{\sigma}\tp{v}}}{p_c^{\sigma}\tp{v}} = \sum_{a = 1}^q B^{d_a^{\sigma}\tp{v} - d_q^{\sigma}\tp{v}} \le q B^{-\Delta},
    \]
    since \(d_a^{\sigma}\tp{v} - d_q^{\sigma}\tp{v} \ge -\Delta\) for every \(a \in \stp{q}\). Hence,
    \[
    \+E_{P_{\-{GD}}}\tp{f, f} \ge \frac{B^{\Delta}}{q n} \+E\tp{f, f}.
    \]
    Combining this comparison with \Cref{thm:ulc-poincare-inequality} proves the spectral-gap bound.
\end{proof}

\begin{proof}[Proof of \Cref{cor:q-potts-random-regular-mixing}]
    Set \(\alpha \defeq 1 / 2 - \zeta\). Since \(1 - B \le \alpha \log \Delta / \Delta = o_{\Delta}\tp{1}\), for all sufficiently large \(\Delta\),
    \[
    \log \frac{1}{B} \le \frac{1 - B}{B} \le \tp{\frac{1}{2} - \frac{\zeta}{2}} \frac{\log \Delta}{\Delta}.
    \]
    Consequently,
    \[
    B^{1 - \Delta} \le B^{-\Delta} \le \Delta^{1 / 2 - \zeta / 2}.
    \]
    On the event in the eigenvalue estimate preceding the corollary, \(\lambda^{\star} \le 3 \sqrt{\Delta}\) for all sufficiently large \(n\). Hence,
    \[
    \lambda^{\star} \tp{1 - B} B^{1 - \Delta} \le 3 \sqrt{\Delta} \frac{\alpha \log \Delta}{\Delta} \Delta^{1 / 2 - \zeta / 2} = 3 \alpha \Delta^{-\zeta / 2} \log \Delta = o_{\Delta}\tp{1}.
    \]
    Thus, for all sufficiently large \(\Delta\), this quantity is at most \(1 / 2\) with probability \(1 - o_n\tp{1}\). By \Cref{thm:antiferromagnetic-potts-ulc}, the embedded Potts measure is then \(\delta\)-ULC with \(\delta \ge 1 / 2\). Therefore, \Cref{cor:q-potts-rapid-mixing} implies
    \[
    \gamma_{\-{GD}}^{-1} \le \frac{q n}{\delta} B^{-\Delta} \le 2 q n \Delta^{1 / 2 - \zeta / 2}.
    \]

    Let \(\mu_{\min} \defeq \min_{\sigma \in \stp{q}^V} \mu_{G, B}\tp{\sigma}\). Since \(B^{\abs{E}} \le B^{m_G\tp{\sigma}} \le 1\), the normalizing constant of \(\mu_{G, B}\) is at most \(q^n\), and hence
    \[
    \mu_{\min} \ge \frac{B^{\abs{E}}}{q^n}.
    \]
    Using \(\abs{E} = n \Delta / 2\) and the bound on \(-\log B\), we obtain
    \[
    \log \frac{1}{\mu_{\min}} \le n \log q + \frac{n \Delta}{2} \log \frac{1}{B} \le n \log q + \frac{n \Delta}{2} \tp{\frac{1}{2} - \frac{\zeta}{2}} \frac{\log \Delta}{\Delta} = O_q\tp{n \log \Delta}.
    \]
    The kernel \(P_{\-{GD}}\) is ergodic and reversible. Moreover, it is the average of the single-site conditional-expectation operators, each of which is an orthogonal projection on \(L^2\tp{\mu_{G, B}}\); hence \(P_{\-{GD}}\) is positive semidefinite. Applying \Cref{thm:functional-inequality-mixing} now gives
    \[
    \begin{aligned}
        T_{\-{mix}}\tp{\eps; P_{\-{GD}}} &\le \ceil{\frac{1}{\gamma_{\-{GD}}} \tp{\log \frac{1}{2 \eps} + \frac{1}{2} \log \frac{1}{\mu_{\min}}}} \\
        &= O_q\tp{n \Delta^{1 / 2 - \zeta / 2} \tp{\log \frac{1}{\eps} + n \log \Delta}} \\
        &= O_q\tp{\sqrt{\Delta} n^2 \log \frac{1}{\eps}}.
    \end{aligned}
    \]
\end{proof}

\section{Point Processes and Probability Measures on \texorpdfstring{\(\bb N^d\)}{N d}}
\label{sec:point-processes}

In this section, we explain the relation between the birth-death chains studied in \Cref{sec:bakry-emery-birth-death} and Glauber dynamics for point processes. This gives a point-process interpretation of the factorial weight \(a\tp{\*n} = \*n! \mu\tp{\*n}\) and relates the \(\delta\)-ULC condition to the curvature criterion of \cite[Theorem~3.4]{KKO13}.

Let \(X\) be a locally compact Polish space equipped with a \(\sigma\)-finite Radon measure \(\nu\). The configuration space \(\+N\tp{X}\) consists of locally finite counting measures
\[
\eta = \sum_{k \in I} \delta_{x_k}.
\]
A point process is a random configuration \(\Xi \sim \bb P\), where \(\bb P \in \+P\tp{\+N\tp{X}}\). For a function \(F: \+N\tp{X} \to \bb R\), define the add-one and remove-one differences by
\[
\nabla_x^+ F\tp{\eta} \defeq F\tp{\eta + \delta_x} - F\tp{\eta}, \qquad \nabla_x^- F\tp{\eta} \defeq F\tp{\eta - \delta_x} - F\tp{\eta},
\]
where the second expression is used only when \(x \in \eta\).

\begin{definition}[Papangelou Intensity]
    A measurable function \(r: X \times \+N\tp{X} \to \bb R_{\ge 0}\) is a Papangelou intensity of \(\bb P\) with respect to \(\nu\) if the Georgii--Nguyen--Zessin identity
    \begin{equation}
        \*E_{\bb P}\stp{\int_X H\tp{x, \eta - \delta_x} \eta\tp{\-d x}} = \*E_{\bb P}\stp{\int_X H\tp{x, \eta} r_x\tp{\eta} \nu\tp{\-d x}}
        \label{eq:papangelou-intensity}
    \end{equation}
    holds for every nonnegative measurable \(H\), where \(r_x\tp{\eta} \defeq r\tp{x, \eta}\).
\end{definition}

The Glauber generator associated with \(r\) is
\begin{equation}
    \+G F\tp{\eta} \defeq \int_X r_x\tp{\eta} \nabla_x^+ F\tp{\eta} \nu\tp{\-d x} + \int_X \nabla_x^- F\tp{\eta} \eta\tp{\-d x}.
    \label{eq:point-process-glauber-generator}
\end{equation}
The identity \eqref{eq:papangelou-intensity} implies that \(\+G\) is reversible with respect to \(\bb P\), with Dirichlet form
\[
\+E_{\+G}\tp{F, F} = \*E_{\bb P}\stp{\int_X r_x\tp{\eta} \tp{\nabla_x^+ F\tp{\eta}}^2 \nu\tp{\-d x}} = \*E_{\bb P}\stp{\int_X \tp{\nabla_x^- F\tp{\eta}}^2 \eta\tp{\-d x}}.
\]

The point-process counterpart of the \(\delta\)-ULC matrix is the following quadratic form. For a test function \(h: X \to \bb R\), set
\begin{equation}
    \begin{aligned}
        \+K_{\eta}^{\tp{\delta}}\tp{h} &\defeq \int_X \int_X h\tp{x} h\tp{y} r_x\tp{\eta} \tp{r_y\tp{\eta} - r_y\tp{\eta + \delta_x}} \nu\tp{\-d y} \nu\tp{\-d x} \\
        &\quad + \tp{1 - \delta} \int_X h\tp{x}^2 r_x\tp{\eta} \nu\tp{\-d x}.
    \end{aligned}
    \label{eq:point-process-curvature-form}
\end{equation}
The Papangelou cocycle identity makes the first kernel symmetric almost everywhere. The coercivity inequality in \cite[Theorem~3.4]{KKO13} gives the following criterion in the self-adjoint ergodic setting.

\begin{theorem}[Point-Process Curvature Criterion]
    Suppose that the cylinder generator in \eqref{eq:point-process-glauber-generator} is essentially self-adjoint on \(L^2\tp{\bb P}\), its closure has only constants in its kernel, and \(\+K_{\eta}^{\tp{\delta}}\tp{h} \ge 0\) for \(\bb P\)-almost every \(\eta\) and every compactly supported test function \(h\). Then \(\+G\) satisfies the \Poincare inequality with constant \(\delta\):
    \[
    \delta \, \*{Var}_{\bb P}\stp{F} \le \+E_{\+G}\tp{F, F}, \quad \forall F \in \+D\tp{\+E_{\+G}}.
    \]
    \label{thm:point-process-curvature-criterion}
\end{theorem}

\paragraph{Simple point processes.} 

When \(X\) is finite, \(\nu\) is the counting measure, and \(\bb P\) has downward closed support consisting of simple configurations, we may identify \(\eta\) with its indicator in \(\set{0, 1}^X\). If \(p\tp{\eta} \defeq \bb P\tp{\eta}\), then
\[
r_x\tp{\eta} = \frac{p\tp{\eta + \delta_x}}{p\tp{\eta}}
\]
on the active set \(\+I_{\eta} \defeq \set{x \in X \cmid \eta + \delta_x \in \supp\tp{\bb P}}\). A diagonal congruence transformation shows that nonnegativity of \eqref{eq:point-process-curvature-form} is equivalent to
\[
\tp{1 - \frac{p\tp{\eta} p\tp{\eta + \delta_x + \delta_y}}{p\tp{\eta + \delta_x} p\tp{\eta + \delta_y}}}_{x, y \in \+I_{\eta}} + \tp{1 - \delta} \diag\tp{\frac{p\tp{\eta}}{p\tp{\eta + \delta_x}}}_{x \in \+I_{\eta}} \succeq \*O.
\]
Thus \eqref{eq:point-process-curvature-form} restricts to the \(\delta\)-ULC condition on a finite-dimensional hypercube and provides its point-process counterpart on a general configuration space.

\paragraph{Marked point-process lift.} 

Let \(\mu\) be a probability measure on \(\bb N^d\) with downward closed support, and set \(a\tp{\*n} \defeq \*n! \mu\tp{\*n}\). Fix a diffuse probability measure \(\nu\) on \(X\), let \(\bar X \defeq X \times \stp{d}\), and equip \(\bar X\) with the measure
\[
\bar\nu\tp{A \times \set{i}} \defeq \frac{1}{d} \nu\tp{A}, \quad A \subseteq X, \quad i \in \stp{d}.
\]
Construct a point process as follows: sample \(\*N \sim \mu\), then, conditionally on \(\*N = \*n\), sample \(\xi_{i, 1}, \dots, \xi_{i, n_i}\) independently from \(\nu\) for every \(i \in \stp{d}\), and set
\begin{equation}
    \Xi \defeq \sum_{i = 1}^d \sum_{j = 1}^{N_i} \delta_{\tp{\xi_{i, j}, i}}.
    \label{eq:marked-point-process-lift}
\end{equation}
Let \(\bb P_{\mu}\) denote the law of \(\Xi\), and let \(\Pi_{\bar\nu}\) denote the Poisson point process with intensity \(\bar\nu\). For \(\eta \in \+N\tp{\bar X}\), write
\[
N_i\tp{\eta} \defeq \eta\tp{X \times \set{i}}, \qquad \*N\tp{\eta} \defeq \tp{N_1\tp{\eta}, \dots, N_d\tp{\eta}}.
\]

\begin{proposition}
    The lifted law \(\bb P_{\mu}\) has density
    \begin{equation}
        \frac{\dd \bb P_{\mu}}{\dd \Pi_{\bar\nu}}\tp{\eta} = e \, d^{\norm{\*N\tp{\eta}}_1} a\tp{\*N\tp{\eta}},
        \label{eq:point-process-lift-density}
    \end{equation}
    and Papangelou intensity
    \begin{equation}
        r_{\tp{x, i}}\tp{\eta} = d \frac{a\tp{\*n + \*e_i}}{a\tp{\*n}}, \qquad \*n = \*N\tp{\eta}.
        \label{eq:point-process-lift-papangelou}
    \end{equation}
    Moreover, if \(F = f \circ \*N\), then the point-process generator \(\+G\) in \eqref{eq:point-process-glauber-generator} satisfies
    \[
    \+G F\tp{\eta} = \sum_{i = 1}^d \frac{a\tp{\*n + \*e_i}}{a\tp{\*n}} \nabla_i^+ f\tp{\*n} + \sum_{i = 1}^d n_i \nabla_i^- f\tp{\*n} = \+L f\tp{\*n},
    \]
    where \(\+L\) is the linear-death generator in \eqref{eq:birth-death-generator}.
    \label{prop:point-process-lift}
\end{proposition}

\begin{proof}[Proof of \Cref{prop:point-process-lift}]
    Under \(\Pi_{\bar\nu}\), the marked counts are independent with law \(\-{Pois}\tp{1 / d}\), and, conditional on these counts, the point locations in each fiber are independent with law \(\nu\). Hence
    \[
    \Pi_{\bar\nu}\tp{\*N = \*n} = e^{-1} \frac{d^{-\norm{\*n}_1}}{\*n!}.
    \]
    Comparing this mass with \(\mu\tp{\*n}\) proves \eqref{eq:point-process-lift-density}. The Papangelou intensity is the add-one ratio of this density, which gives \eqref{eq:point-process-lift-papangelou}.

    If \(F = f \circ \*N\), then \(\nabla_{\tp{x, i}}^+ F\tp{\eta} = \nabla_i^+ f\tp{\*n}\), independently of \(x\). Similarly, removing any point in the \(i\)-th fiber produces \(\nabla_i^- f\tp{\*n}\). Since \(\bar\nu\tp{X \times \set{i}} = 1 / d\), substituting \eqref{eq:point-process-lift-papangelou} into \eqref{eq:point-process-glauber-generator} yields the claimed generator identity.
\end{proof}

The factor \(\*n!\) in \eqref{eq:point-process-lift-density} is the symmetry factor for unlabeled points: conditional on \(\*N = \*n\), permuting the \(n_i\) points in any fiber leaves the configuration unchanged. Thus \(a\tp{\*n}\) is the natural weight relative to the Poisson reference law \(\Pi_{\bar{\nu}}\).

Assume now that \(\mu\) is \(\delta\)-ULC for some \(0 < \delta \le 1\). Fix a configuration \(\eta\) with count vector \(\*n\), and define
\[
b_i\tp{\*n} \defeq \frac{a\tp{\*n + \*e_i}}{a\tp{\*n}}, \qquad b_{ij}\tp{\*n} \defeq \frac{a\tp{\*n + \*e_i + \*e_j}}{a\tp{\*n}},
\]
and
\[
\*B_{\*n} \defeq \tp{b_i\tp{\*n} b_j\tp{\*n} - b_{ij}\tp{\*n}}_{i, j \in \+I_{\*n}}, \qquad \*A_{\*n} \defeq \*B_{\*n} + \tp{1 - \delta} \diag\tp{b_i\tp{\*n}}_{i \in \+I_{\*n}}.
\]
Congruence by \(\diag\tp{b_i\tp{\*n}^{-1}}_{i \in \+I_{\*n}}\) transforms \(\*A_{\*n}\) into the matrix in \eqref{eq:ulc-condition}. Hence \(\*A_{\*n} \succeq \*O\), and every fiber-constant test function \(h\tp{x, i} = v_i\) satisfies
\[
\+K_{\eta}^{\tp{\delta}}\tp{h} = \*v^{\top} \*A_{\*n} \*v \ge 0.
\]

For an arbitrary test function \(h: \bar X \to \bb R\), set
\[
\bar{h}_i \defeq \int_X h\tp{x, i} \nu\tp{\-d x}, \qquad \bar{\*h} \defeq \tp{\bar{h}_i}_{i \in \+I_{\*n}}.
\]
Substituting \eqref{eq:point-process-lift-papangelou} and the definition of \(\bar\nu\) into \eqref{eq:point-process-curvature-form} gives
\[
\+K_{\eta}^{\tp{\delta}}\tp{h} = \bar{\*h}^{\top} \*B_{\*n} \bar{\*h} + \tp{1 - \delta} \sum_{i \in \+I_{\*n}} b_i\tp{\*n} \int_X h\tp{x, i}^2 \nu\tp{\-d x}, \qquad \*n = \*N\tp{\eta}.
\]
Since \(\nu\) is a probability measure, Jensen's inequality gives
\[
\int_X h\tp{x, i}^2 \nu\tp{\-d x} \ge \tp{\int_X h\tp{x, i} \nu\tp{\-d x}}^2 = \bar{h}_i^2.
\]
Because \(1 - \delta \ge 0\), it follows that
\[
\+K_{\eta}^{\tp{\delta}}\tp{h} \ge \bar{\*h}^{\top} \*B_{\*n} \bar{\*h} + \tp{1 - \delta} \sum_{i \in \+I_{\*n}} b_i\tp{\*n} \bar{h}_i^2 = \bar{\*h}^{\top} \*A_{\*n} \bar{\*h} \ge 0.
\]
Thus \(\bb P_{\mu}\) satisfies the curvature hypothesis in \Cref{thm:point-process-curvature-criterion}. Under the analytic assumptions of that theorem, its Glauber generator satisfies the \Poincare inequality with constant \(\delta\).

Finally, restrict this point-process \Poincare inequality to \(F = f \circ \*N\). Since \(\*N\tp{\Xi} \sim \mu\),
\[
\*{Var}_{\bb P_{\mu}}\stp{F} = \*{Var}_{\mu}\stp{f},
\]
while \Cref{prop:point-process-lift} and the point-process Dirichlet form give
\[
\+E_{\+G}\tp{F, F} = -\inner{F}{\+G F}_{L^2\tp{\bb P_{\mu}}} = -\inner{f}{\+L f}_{L^2\tp{\mu}} = \+E_{\+L}\tp{f, f}.
\]
Consequently,
\[
\delta \, \*{Var}_{\mu}\stp{f} \le \+E_{\+L}\tp{f, f} = \*E_{\*n \sim \mu}\stp{\sum_{i = 1}^d \frac{\tp{n_i + 1} \mu\tp{\*n + \*e_i}}{\mu\tp{\*n}} \tp{\nabla_i^+ f\tp{\*n}}^2}, \quad \forall f \in \+D\tp{\+E},
\]
which is the weighted \Poincare inequality in \Cref{thm:ulc-poincare-inequality}.

\section{Poisson Stochastic Localization and Poisson Bridge}
\label{sec:poisson-bridge}

Poisson stochastic localization also admits a \Schrodinger-bridge formulation. This parallels the bridge interpretation of standard stochastic localization in \cite[Section~6]{STZ25}, with Wiener measure replaced by a Poisson path measure. Such a formulation is also discussed in \cite[Section~2.4]{SBP26}; for completeness, we give the explicit computation in this section.

Let \(\Omega\) be a countable state space, let \(R\) be a Markov law on the c\`adl\`ag path space \(\+D\tp{\stp{0, 1}, \Omega}\), and write \(R_t\) for its time-\(t\) marginal. Given endpoint laws \(\alpha\) and \(\beta\), the static \Schrodinger bridge problem is
\begin{equation}
    \inf_{\substack{P \in \+P\tp{\+D\tp{\stp{0, 1}, \Omega}} \\ P_0 = \alpha, \, P_1 = \beta}} \quad \-{KL}\tp{P \| R}.
    \label{eq:static-schrodinger-bridge}
\end{equation}
We use the following Markovian description of its solution; see \cite[Lemmas~13 to~15]{STZ25}.

\begin{theorem}
    Suppose that the infimum in \eqref{eq:static-schrodinger-bridge} is finite and attained. Then there are nonnegative functions \(f_0, g_1: \Omega \to \bb R_{\ge 0}\) such that
    \[
    \frac{\dd P^{\star}}{\dd R}\tp{\*p_{\stp{0, 1}}} = f_0\tp{\*p_0} g_1\tp{\*p_1}.
    \]
    Define
    \[
    f_t\tp{\*z} \defeq \*E_{\*p_{\stp{0, 1}} \sim R}\stp{f_0\tp{\*p_0} \mid \*p_t = \*z}, \qquad g_t\tp{\*z} \defeq \*E_{\*p_{\stp{0, 1}} \sim R}\stp{g_1\tp{\*p_1} \mid \*p_t = \*z}.
    \]
    Then \(P^{\star}\) is Markov, its time-\(t\) marginal satisfies
    \[
    \frac{\dd P_t^{\star}}{\dd R_t}\tp{\*p_t} = f_t\tp{\*p_t} g_t\tp{\*p_t},
    \]
    and, for \(0 \le s \le t \le 1\), its forward and backward transition kernels are
    \[
    P_{s \to t}^{\star}\tp{\*p_s, \*p_t} = \frac{g_t\tp{\*p_t}}{g_s\tp{\*p_s}} R_{s \to t}\tp{\*p_s, \*p_t}, \qquad P_{t \to s}^{\star}\tp{\*p_t, \*p_s} = \frac{f_s\tp{\*p_s}}{f_t\tp{\*p_t}} R_{t \to s}\tp{\*p_t, \*p_s}.
    \]
    If \(\+L_t^R\) and \(\Gamma_t^R\) are the generator and carr\'e du champ of the reference process, then
    \begin{equation}
        \+L_t^{P^{\star}} u = \+L_t^R u + \frac{2 \Gamma_t^R\tp{g_t, u}}{g_t}.
        \label{eq:schrodinger-bridge-doob-transform}
    \end{equation}
    \label{thm:markovian-schrodinger-bridge}
\end{theorem}

We now take \(R\) to be the law of \(d\) independent rate-one Poisson processes \(\tp{\*N_t}_{t \in \stp{0, 1}}\), all initialized at \(\*0\), and set the endpoint laws to \(\alpha = \delta_{\*0}\) and \(\beta = \mu\). Then
\[
R_0 = \delta_{\*0}, \qquad R_t = \pi_t \defeq \bigotimes_{i = 1}^d \-{Pois}\tp{t}, \qquad \+L^R u = \sum_{i = 1}^d \nabla_i^+ u, \qquad \Gamma^R\tp{u, v} = \frac{1}{2} \sum_{i = 1}^d \nabla_i^+ u \, \nabla_i^+ v.
\]
Write \(\mu_t \defeq P_t^{\star}\). By \Cref{thm:markovian-schrodinger-bridge}, we may choose, for every \(t \in \stp{0, 1}\),
\[
f_t \equiv 1, \qquad g_t = \frac{\dd \mu_t}{\dd \pi_t}.
\]
Then
\[
g_t\tp{\*n} = \frac{\mu_t\tp{\*n}}{\prod_{i = 1}^d \tp{e^{-t} \frac{t^{n_i}}{n_i!}}} = e^{d t} \frac{\*n! \mu_t\tp{\*n}}{t^{\norm{\*n}_1}}.
\]
Hence, by \eqref{eq:schrodinger-bridge-doob-transform},
\[
\begin{aligned}
    \+L_t^{P^{\star}} u\tp{\*n} &= \sum_{i = 1}^d \nabla_i^+ u\tp{\*n} + g_t^{-1}\tp{\*n} \sum_{i = 1}^d \nabla_i^+ u\tp{\*n} \nabla_i^+ g_t\tp{\*n} \\
    &= \sum_{i = 1}^d \frac{g_t\tp{\*n + \*e_i}}{g_t\tp{\*n}} \nabla_i^+ u\tp{\*n} \\
    &= \sum_{i = 1}^d \frac{\tp{n_i + 1} \mu_t\tp{\*n + \*e_i}}{t \mu_t\tp{\*n}} \nabla_i^+ u\tp{\*n}.
\end{aligned}
\]
This matches the denoising generator \eqref{eq:denoising-generator} of Poisson stochastic localization.

Replacing the Poisson reference law by Wiener measure gives standard stochastic localization after the usual time change \cite[Section~6]{STZ25}. In that setting, \eqref{eq:schrodinger-bridge-doob-transform} transforms \(\+L^R = \Delta / 2\) into
\[
\+L_t^{P^{\star}} = \frac{1}{2} \Delta + \inner{\nabla \log g_t}{\nabla},
\]
the generator of the \Follmer process. The Poisson construction replaces the Brownian score \(\nabla \log g_t\) by the multiplicative discrete increment \(g_t\tp{\*n + \*e_i} / g_t\tp{\*n}\), and the solution \(P^{\star}\) becomes the Poisson--\Follmer process.

\section{Weighted \Poincare Inequality via Trickle-Down Theorem}
\label{sec:trickle-down}

In this section, we give an alternative proof of \Cref{thm:ulc-poincare-inequality} using the Poisson stochastic localization developed in \Cref{sec:poisson-stochastic-localization}. The argument remains entirely within the localization framework: the trickle-down equation of \cite{AKV24} propagates the local \(\delta\)-ULC condition from the terminal regime backward along the localization process. In particular, the proof is independent of the Bakry--\Emery theory developed in \Cref{sec:bakry-emery-birth-death} and extends the argument of \cite[Section~4.1]{CCCYZ25} to measures on \(\bb N^d\).

\begin{theorem}[Trickle-Down Equation (Informal) \cite{AKV24}]
    Let \(\tp{\nu_t}\) denote the linear-tilt localization scheme defined by
    \[
    \dd \nu_t\tp{\*n} = \nu_t\tp{\*n} \inner{\*n - \mean{\nu_t}}{\dd \*Z_t},
    \]
    where \(\tp{\*Z_t}\) is a martingale satisfying
    \[
    \*E\stp{\dd \*Z_t \mid \+F_t} = 0.
    \]
    Then the SDE for \(\tp{\cov{\nu_t}}\) is
    \[
    \dd \, \cov{\nu_t} = -\cov{\nu_t} \dd \ang{\*Z}_t  \cov{\nu_t} + \dd \, \-{Martingale}_t.
    \]
    \label{thm:trickle-down-equation}
\end{theorem}

\begin{corollary}
    Let \(\mu\) be a probability measure on \(\bb N^d\) with downward closed support. Let \(\tp{\nu_t}_{t \in \stp{0, 1}}\) denote the corresponding localization process given by Poisson stochastic localization. Define \(\rho_t^{\*x}\) as in \Cref{lem:tilted-pinning-conditional-law}:
    \[
    \rho_t^{\*x} \defeq \-{Law}\tp{\*X_1 - \*X_t \mid \*X_t = \*x} = \tp{1 - t} * \mu^{\*x}, \quad t \in \tp{0, 1}, \, \*x \in \supp\tp{\mu}.
    \]
    Then the SDE for \(\tp{\rho_t^{\*X_t}}_{t \in \tp{0, 1}}\) is
    \[
    \dd \, \cov{\rho_t^{\*X_t}} = -\frac{1}{1 - t} \cov{\rho_t^{\*X_t}} \diag\tp{\mean{\rho_t^{\*X_t}}}^{\dagger} \cov{\rho_t^{\*X_t}} \dd t + \dd \, \-{Martingale}_t, \quad t \in \tp{0, 1}.
    \]
    \label{cor:trickle-down-tilted-pinning}
\end{corollary}

\begin{proof}[Proof of \Cref{cor:trickle-down-tilted-pinning}]
    Translation by \(\*X_t\) does not change covariance, so
    \[
    \cov{\rho_t^{\*X_t}} = \cov{\nu_t}.
    \]
    Write \(\*a_t \defeq \mean{\nu_t} - \*X_t = \mean{\rho_t^{\*X_t}}\). By \Cref{thm:poisson-linear-tilt}, the driving martingale satisfies \(\dd Z_{t, i} = \dd M_{t, i} / a_{t-, i}\) on the coordinates for which \(a_{t-, i} > 0\). The coordinate counting processes have no simultaneous jumps, and hence
    \[
    \dd \ang{\*M}_t = \diag\tp{\*q_t\tp{\*X_{t-}}} \dd t.
    \]
    By \Cref{cor:denoising-birth-rates-conditional-expectation}, \(\*q_t\tp{\*X_{t-}} = \*a_{t-} / \tp{1 - t}\). Therefore,
    \[
    \dd \ang{\*Z}_t = \frac{1}{1 - t} \diag\tp{\*a_{t-}}^{-1} \dd t,
    \]
    where the inverse is taken on the coordinates on which \(a_{t-, i} > 0\), and both sides vanish on the remaining coordinates. Substituting this bracket into \Cref{thm:trickle-down-equation} gives the claimed SDE.
\end{proof}

We next control the optimal spectral-independence constant of each tilted pinning. For \(\*x \in \supp\tp{\mu}\), define
\begin{equation}
    \eta_{\*x}\tp{t} \defeq \inf \set{\eta \ge 0 \cmid \cov{\rho_t^{\*x}} \preceq \eta \diag\tp{\mean{\rho_t^{\*x}}}}, \qquad t \in \tp{0, 1}.
    \label{eq:eta-x-definition}
\end{equation}
We then set
\begin{equation}
    \eta\tp{t} \defeq \sup_{\*x \in \supp\tp{\mu}} \eta_{\*x}\tp{t}.
    \label{eq:eta-definition}
\end{equation}
Thus \(\rho_t^{\*x}\) is \(\eta_{\*x}\tp{t}\)-spectrally independent, and an upper bound on \(\eta\) verifies the hypothesis of \Cref{prop:sufficient-condition-approximate-conservation}. If \(\mean{\rho_t^{\*x}} = \*0\), then \(\rho_t^{\*x} = \delta_{\*0}\), and the convention in \eqref{eq:eta-x-definition} gives \(\eta_{\*x}\tp{t} = 0\).

\begin{lemma}
    For every \(\*x \in \supp\tp{\mu}\), the function \(\eta_{\*x}\) is locally Lipschitz on \(\tp{0, 1}\). If \(\mu\) is finitely supported, then \(\eta\) is also locally Lipschitz on \(\tp{0, 1}\).
    \label{lem:regularity-of-eta}
\end{lemma}

\begin{proof}[Proof of \Cref{lem:regularity-of-eta}]
    Fix \(\*x \in \supp\tp{\mu}\), and let \(\+I_{\*x} \defeq \set{i \in \stp{d} \cmid \*x + \*e_i \in \supp\tp{\mu}}\). Downward closure implies that the coordinates outside \(\+I_{\*x}\) vanish under \(\rho_t^{\*x}\). On \(\+I_{\*x}\), set
    \[
    \*C_{\*x}\tp{t} \defeq \cov{\rho_t^{\*x}}, \qquad \*D_{\*x}\tp{t} \defeq \diag\tp{\mean{\rho_t^{\*x}}}, \qquad \*A_{\*x}\tp{t} \defeq \*D_{\*x}\tp{t}^{-1/2} \*C_{\*x}\tp{t} \*D_{\*x}\tp{t}^{-1/2}.
    \]
    Every diagonal entry of \(\*D_{\*x}\tp{t}\) is positive on \(\+I_{\*x}\). Writing \(a\tp{\*n} \defeq \*n! \mu\tp{\*n}\), the law \(\rho_t^{\*x}\) has mass function proportional to
    \[
    \frac{a\tp{\*x + \*k}}{\*k!} \tp{1 - t}^{\norm{\*k}_1}.
    \]
    Its moments are therefore analytic in \(t\) on \(\tp{0, 1}\). It follows that \(\*A_{\*x}\) is analytic there. Since
    \[
    \eta_{\*x}\tp{t} = \lambda_{\max}\tp{\*A_{\*x}\tp{t}},
    \]
    the standard eigenvalue perturbation bound shows that \(\eta_{\*x}\) is locally Lipschitz. If \(\+I_{\*x} = \varnothing\), then \(\eta_{\*x} \equiv 0\), so the same conclusion holds. When \(\mu\) has finite support, \(\eta\) is the maximum of finitely many locally Lipschitz functions and is therefore locally Lipschitz.
\end{proof}

\begin{lemma}
    For every \(\*x \in \supp\tp{\mu}\) with \(\+I_{\*x} \ne \varnothing\),
    \[
    \eta_{\*x}\tp{1^-} = 1.
    \]
    If \(\+I_{\*x} = \varnothing\), then \(\eta_{\*x} \equiv 0\). Moreover, if the \(\delta\)-ULC condition \eqref{eq:ulc-condition} holds at \(\*n = \*x\), then
    \[
    \eta_{\*x}\tp{1 - \eps} \le 1 + \tp{1 - \delta} \eps + O\tp{\eps^2} \quad \text{as } \eps \to 0^+.
    \]
    In particular, if \(\mu\) is finitely supported and \(\delta\)-ULC, then
    \[
    \eta\tp{1 - \eps} \le 1 + \tp{1 - \delta} \eps + O\tp{\eps^2} \quad \text{as } \eps \to 0^+.
    \]
    If, in addition, \(\mu \ne \delta_{\*0}\), then \(\eta\tp{1^-} = 1\); for \(\mu = \delta_{\*0}\), one has \(\eta \equiv 0\).
    \label{lem:eta-t-approaches-1}
\end{lemma}

\begin{proof}[Proof of \Cref{lem:eta-t-approaches-1}]
    Fix \(\*x\) with \(\+I_{\*x} \ne \varnothing\), write \(a\tp{\*n} \defeq \*n! \mu\tp{\*n}\), and set
    \[
    r_i \defeq \frac{a\tp{\*x + \*e_i}}{a\tp{\*x}}, \qquad \*R_{\*x} \defeq \diag\tp{r_i}_{i \in \+I_{\*x}}, \qquad \*K_{\*x} \defeq \tp{1 - \frac{a\tp{\*x} a\tp{\*x + \*e_i + \*e_j}}{a\tp{\*x + \*e_i} a\tp{\*x + \*e_j}}}_{i, j \in \+I_{\*x}}.
    \]
    Let
    \[
    h_{\*x}\tp{\boldsymbol{\theta}} \defeq \sum_{\*k \in \bb N^d} \frac{a\tp{\*x + \*k}}{\*k!} \boldsymbol{\theta}^{\*k}.
    \]
    As in \eqref{eq:tilted-pinning-log-hessian}, with \(\boldsymbol{\theta} = \eps \*1\),
    \[
    \cov{\rho_{1 - \eps}^{\*x}} = \diag\tp{\mean{\rho_{1 - \eps}^{\*x}}} + \eps^2 \nabla^2 \log h_{\*x}\tp{\eps \*1}.
    \]
    Moreover,
    \[
    \mean{\rho_{1 - \eps}^{\*x}} = \eps \tp{r_i}_{i \in \+I_{\*x}} + O\tp{\eps^2}, \qquad \nabla^2 \log h_{\*x}\tp{\eps \*1} = \tp{\frac{a\tp{\*x + \*e_i + \*e_j}}{a\tp{\*x}} - r_i r_j}_{i, j \in \+I_{\*x}} + O\tp{\eps}.
    \]
    Consequently, the normalized covariance from the proof of \Cref{lem:regularity-of-eta} satisfies
    \[
    \*A_{\*x}\tp{1 - \eps} = \*I - \eps \*R_{\*x}^{1 / 2} \*K_{\*x} \*R_{\*x}^{1 / 2} + O\tp{\eps^2}.
    \]
    This already gives \(\eta_{\*x}\tp{1^-} = 1\). If the \(\delta\)-ULC condition holds at \(\*x\), then
    \[
    \*K_{\*x} + \tp{1 - \delta} \*R_{\*x}^{-1} \succeq \*O,
    \]
    and hence
    \[
    -\*R_{\*x}^{1 / 2} \*K_{\*x} \*R_{\*x}^{1 / 2} \preceq \tp{1 - \delta} \*I.
    \]
    Taking the largest eigenvalue in the preceding expansion proves
    \[
    \eta_{\*x}\tp{1 - \eps} \le 1 + \tp{1 - \delta} \eps + O\tp{\eps^2}.
    \]
    If \(\mu\) is finitely supported, the support contains only finitely many \(\*x\), so the error term is uniform in \(\*x\). Taking the maximum proves the final assertions.
\end{proof}

\begin{lemma}
    Suppose that \(\mu\) is finitely supported. Then, for every \(t \in \tp{0, 1}\),
    \[
    \eta\tp{t} \le \eta\tp{t + \eps} + \frac{\eta\tp{t}^2 - \eta\tp{t}}{1 - t} \eps + O\tp{\eps^2} \quad \text{as } \eps \to 0^+.
    \]
    \label{lem:eta-t-0-to-1}
\end{lemma}

\begin{proof}[Proof of \Cref{lem:eta-t-0-to-1}]
    Along the localization process, write
    \[
    \*C_s \defeq \cov{\rho_s^{\*X_s}}, \qquad \*D_s \defeq \diag\tp{\mean{\rho_s^{\*X_s}}}.
    \]
    All inverses below are generalized inverses, or equivalently inverses restricted to the coordinates on which the corresponding mean is positive. Integrating \Cref{cor:trickle-down-tilted-pinning} from \(t\) to \(t + \eps\) and conditioning on \(\+F_t\) gives
    \[
    \*C_t = \*E\stp{\*C_{t + \eps} \mid \+F_t} + \*E\stp{\int_t^{t + \eps} \frac{1}{1 - s} \*C_s \*D_s^{\dagger} \*C_s \dd s \mathrel{\big|} \+F_t}.
    \]
    By the definition of \(\eta\),
    \[
    \*C_s \preceq \eta\tp{s} \*D_s, \qquad \*C_s \*D_s^{\dagger} \*C_s \preceq \eta\tp{s}^2 \*D_s.
    \]
    Hence,
    \[
    \*C_t \preceq \eta\tp{t + \eps} \*E\stp{\*D_{t + \eps} \mid \+F_t} + \*E\stp{\int_t^{t + \eps} \frac{\eta\tp{s}^2}{1 - s} \*D_s \dd s \mathrel{\big|} \+F_t}.
    \]
    Set \(\*a_s \defeq \mean{\rho_s^{\*X_s}}\). The tower property and the martingale property of the denoising birth rates yield, for \(s \ge t\),
    \[
    \*E\stp{\*a_s \mid \+F_t} = \*a_t - \*E\stp{\*X_s - \*X_t \mid \+F_t} = \*a_t - \int_t^s \*E\stp{\*q_u\tp{\*X_u} \mid \+F_t} \dd u = \tp{1 - \frac{s - t}{1 - t}} \*a_t.
    \]
    Therefore,
    \[
    \*E\stp{\*D_s \mid \+F_t} = \frac{1 - s}{1 - t} \*D_t.
    \]
    Substituting this estimate into the preceding inequality gives
    \[
    \*C_t \preceq \tp{\tp{1 - \frac{\eps}{1 - t}} \eta\tp{t + \eps} + \frac{1}{1 - t} \int_t^{t + \eps} \eta\tp{s}^2 \dd s} \*D_t.
    \]
    Taking the maximum over the finitely many possible values of \(\*X_t\) yields
    \[
    \eta\tp{t} \le \tp{1 - \frac{\eps}{1 - t}} \eta\tp{t + \eps} + \frac{1}{1 - t} \int_t^{t + \eps} \eta\tp{s}^2 \dd s.
    \]
    By \Cref{lem:regularity-of-eta}, \(\eta\) is locally Lipschitz, so
    \[
    \int_t^{t + \eps} \eta\tp{s}^2 \dd s = \eta\tp{t}^2 \eps + O\tp{\eps^2}, \qquad \eta\tp{t + \eps} = \eta\tp{t} + O\tp{\eps}.
    \]
    It follows that
    \[
    \eta\tp{t} \le \eta\tp{t + \eps} + \frac{\eta\tp{t}^2 - \eta\tp{t}}{1 - t} \eps + O\tp{\eps^2} \quad \text{as } \eps \to 0^+,
    \]
    as claimed.
\end{proof}

The terminal expansion in \Cref{lem:eta-t-approaches-1} supplies the boundary condition for the backward comparison, while \Cref{lem:eta-t-0-to-1} supplies the differential inequality in the interior. Together they yield the sharp spectral-independence profile and, through approximate conservation of variance, the desired weighted \Poincare inequality.

\begin{theorem}
    Let \(\mu\) be a \(\delta\)-ULC probability measure on \(\bb N^d\). Then, for \(t \in \tp{0, 1}\), \(\tp{1 - t} * \mu^{\*x}\) is \(1 / \tp{1 - \tp{1 - \delta} \tp{1 - t}}\)-spectrally independent for every \(\*x \in \supp\tp{\mu}\). Moreover, \(\mu\) satisfies the weighted \Poincare inequality with constant \(\delta\):
    \[
    \delta \, \*{Var}_{\mu}\stp{f} \le \*E_{\mu}\stp{\sum_{i = 1}^d \frac{\tp{n_i + 1} \mu\tp{\*n + \*e_i}}{\mu\tp{\*n}} \tp{\nabla_i^+ f\tp{\*n}}^2}, \quad \forall f \in \+D\tp{\+E}.
    \]
    Equivalently, the birth-death generator \(\+L\) defined in \eqref{eq:birth-death-generator} has spectral gap at least \(\delta\).
    \label{thm:ulc-poincare-trickle-down}
\end{theorem}

\begin{remark}
    The \(\delta\)-ULC matrix appears here as the terminal boundary condition for the normalized covariance flow. Thus the trickle-down argument does more than recover the weighted \Poincare inequality: it identifies \(\delta\)-ULC as precisely the local second-order condition that initiates the trickle-down analysis along Poisson stochastic localization.
\end{remark}

\begin{proof}[Proof of \Cref{thm:ulc-poincare-trickle-down}]
    We first assume that \(\mu\) has finite support. If \(\mu = \delta_{\*0}\), both conclusions are immediate, so suppose otherwise. By \Cref{lem:regularity-of-eta}, \Cref{lem:eta-t-approaches-1}, and \Cref{lem:eta-t-0-to-1}, we have
    \begin{equation}
        \eta\tp{t} \le \frac{1}{1 - \tp{1 - \delta} \tp{1 - t}}.
        \label{eq:eta-trickle-down-bound}
    \end{equation}
    Hence, for \(t \in \tp{0, 1}\), \(\tp{1 - t} * \mu^{\*x}\) is \(1 / \tp{1 - \tp{1 - \delta} \tp{1 - t}}\)-spectrally independent for every \(\*x \in \supp\tp{\mu}\). By \Cref{prop:sufficient-condition-approximate-conservation},
    \[
    \partial_t \*E\stp{\*{Var}_{\nu_t}\stp{f}} \ge -\frac{\*E\stp{\*{Var}_{\nu_t}\stp{f}}}{\tp{1 - t} \tp{1 - \tp{1 - \delta} \tp{1 - t}}}, \quad \forall t \in \tp{0, 1}.
    \]
    It follows that
    \[
    \gamma_t = \inf_{\substack{f \in \+C_{\-c}\tp{\supp\tp{\mu}} \\ \*{Var}_{\mu}\stp{f} > 0}} \frac{\*E\stp{\*{Var}_{\nu_t}\stp{f}}}{\*{Var}_{\mu}\stp{f}} \ge \exp\tp{-\int_0^t \frac{1}{\tp{1 - s} \tp{1 - \tp{1 - \delta} \tp{1 - s}}} \dd s} = \frac{\delta \tp{1 - t}}{t + \delta \tp{1 - t}}.
    \]
    Consequently,
    \[
    \liminf_{\eps \to 0^+} \frac{\gamma_{1 - \eps}}{\eps} \ge \delta,
    \]
    and \Cref{thm:limiting-process-functional-inequalities} proves the weighted \Poincare inequality with constant \(\delta\) for finitely supported \(\delta\)-ULC \(\mu\).

    We now remove the finite-support assumption. For \(N \in \bb N\), apply the coordinate-box restrictions \(S_N\) and conditional measures \(\mu_N\) introduced in the proof of \Cref{thm:ulc-poincare-inequality}. For every fixed \(t \in \tp{0, 1}\) and \(\*x \in \supp\tp{\mu}\), the tilted pinnings \(\tp{1 - t} * \mu_N^{\*x}\) converge to \(\tp{1 - t} * \mu^{\*x}\) together with their first two moments. Thus, the finite-support covariance bounds pass to the limit and prove the first conclusion. The finite-support result applied to \(\mu_N\) also gives
    \[
    \delta \, \*{Var}_{\mu_N}\stp{f} \le \+E_N\tp{f, f}.
    \]
    For every \(f \in \+D\tp{\+E}\), one has \(\*{Var}_{\mu_N}\stp{f} \to \*{Var}_{\mu}\stp{f}\) and \(\+E_N\tp{f, f} \to \+E\tp{f, f}\) as \(N \to \infty\). Passing to the limit proves the claimed weighted \Poincare inequality on \(\+D\tp{\+E}\).
\end{proof}

\section{Closure Properties of ULC Measures}
\label{sec:ulc-closure-properties}

In this section, we establish closure properties of high-dimensional ULC measures.

We first relate ULC to log-concavity relative to a product Poisson measure. We say that a function \(f: \bb N^d \to \bb R_{\ge 0}\) is log-concave if \(\supp\tp{f}\) is downward closed and
\begin{equation}
    \tp{1 - \frac{f\tp{\*n} f\tp{\*n + \*e_i + \*e_j}}{f\tp{\*n + \*e_i} f\tp{\*n + \*e_j}}}_{i, j \in \+I_{\*n}} \succeq \*O, \quad \forall \*n \in \supp\tp{f},
    \label{eq:discrete-function-log-concavity}
\end{equation}
where \(\+I_{\*n} \defeq \set{i \in \stp{d} \cmid \*n + \*e_i \in \supp\tp{f}}\).

\begin{proposition}
    Let \(\mu\) be a probability measure supported on \(\bb N^d\) with downward closed support. For every \(\boldsymbol{\lambda} \in \bb R_{> 0}^d\), the measure \(\mu\) is ULC if and only if its relative density with respect to \(\pi_{\boldsymbol{\lambda}} \defeq \bigotimes_{i = 1}^d \-{Pois}\tp{\lambda_i}\) is log-concave in the sense of \eqref{eq:discrete-function-log-concavity}.
    \label{prop:ulc-poisson-relative-log-concavity}
\end{proposition}

\begin{proof}
    Write \(a\tp{\*n} \defeq \*n! \mu\tp{\*n}\). Since
    \[
    \frac{\dd \mu}{\dd \pi_{\boldsymbol{\lambda}}}\tp{\*n} = e^{\norm{\boldsymbol{\lambda}}_1} \frac{a\tp{\*n}}{\boldsymbol{\lambda}^{\*n}},
    \]
    the constant and geometric factors cancel from every ratio in \eqref{eq:discrete-function-log-concavity}. The resulting matrix is exactly the ULC matrix of \(\mu\).
\end{proof}

\subsection{Basic Closure Properties}

The following closure properties follow directly from the local-matrix condition in \eqref{eq:ulc-condition}.

\begin{proposition}
    The class of ULC measures is closed under the following operations:
    \begin{itemize}
        \item \textbf{Exponential tilt:} If \(\mu\) is ULC, then for every \(\boldsymbol{\theta} \in \bb R_{> 0}^d\), the tilted measure \(\boldsymbol{\theta} * \mu\) is ULC.
        \item \textbf{Weighted pinning:} If \(\mu\) is ULC, then for every \(\*x \in \supp\tp{\mu}\), the measure \(\mu^{\*x}\), as defined in \Cref{def:weighted-pinning}, is ULC.
        \item \textbf{Product:} If \(\mu^{\tp{i}}\) is ULC for \(i = 1, 2, \dots, k\), then the product measure \(\bigotimes_{i = 1}^k \mu^{\tp{i}}\) is ULC.
        \item \textbf{Box truncation:} If \(\mu\) is ULC on \(\bb N^d\), then for every \(\*N \in \bb N^d\), the measure \(\mu\) restricted to the box \(\set{\*n \in \bb N^d \cmid \*n \le \*N}\) is ULC.
        \item \textbf{Coordinate conditioning:} If \(\mu\) is ULC on \(\bb N^d\) and \(\*{Pr}_{\*X \sim \mu}\stp{\*X_S = \*x} > 0\) for some \(S \subseteq \stp{d}\) and \(\*x \in \bb N^S\), then \(\mu_{\stp{d} \setminus S}^{S \gets \*x}\) is ULC.
    \end{itemize}
    \label{prop:ulc-basic-closure-properties}
\end{proposition}

\begin{remark}
    Weighted pinning, product, box truncation, and coordinate conditioning preserve not only ULC, but also \(\delta\)-ULC for the same \(\delta\).
\end{remark}

\subsection{Marginalization and Coordinatewise Binomial Thinning}

Closure under marginalization and coordinatewise binomial thinning follows directly from the strongly log-concave probability-generating function characterization of ULC measures in \Cref{thm:ulc-generating-function-characterization}.

\begin{proposition}
    If \(\mu\) is ULC on \(\bb N^d\), then for every \(S \subseteq \stp{d}\), the marginal measure \(\mu_S\) is ULC.
    \label{prop:ulc-marginalization-closure}
\end{proposition}

\begin{proof}[Proof of \Cref{prop:ulc-marginalization-closure}]
    The support of \(\mu_S\) is downward closed, and
    \[
    g_{\mu_S}\tp{\boldsymbol{\theta}_S} = g_{\mu}\tp{\boldsymbol{\theta}_S, \*1_{\stp{d} \setminus S}}, \quad \boldsymbol{\theta}_S \in \bb R_{> 0}^S.
    \]
    Every partial derivative of \(g_{\mu_S}\) is therefore the restriction of the corresponding partial derivative of \(g_{\mu}\) to this affine subspace. Since restrictions of log-concave functions to affine subspaces are log-concave, \(g_{\mu_S}\) is strongly log-concave. The result follows from \Cref{thm:ulc-generating-function-characterization}.
\end{proof}

\begin{definition}
    Let \(\mu\) be a probability measure on \(\bb N^d\). For \(\*p \in \left(0, 1\right]^d\), define the coordinatewise binomial thinning operator \(\+T_{\*p}\) by
    \begin{equation}
        \+T_{\*p} \mu \defeq \-{Law}\tp{\*X'}, \qquad \*{Pr}\stp{\*X' = \*n' \mid \*X = \*n} = \prod_{i = 1}^d \-{Bin}\tp{n_i, p_i; n_i'}, \qquad \*X \sim \mu.
        \label{eq:binomial-thinning-operator}
    \end{equation}
    In particular, for \(p \in \left(0, 1\right]\), we abbreviate \(\+T_{p \*1}\) as \(\+T_p\).
    \label{def:binomial-thinning-operator}
\end{definition}

\begin{proposition}
    If \(\mu\) is ULC on \(\bb N^d\), then for every \(\*p \in \left(0, 1\right]^d\), the thinned measure \(\+T_{\*p} \mu\) is ULC.
    \label{prop:ulc-thinning-closure}
\end{proposition}

\begin{proof}[Proof of \Cref{prop:ulc-thinning-closure}]
    The probability-generating function of \(\+T_{\*p} \mu\) is
    \[
    g_{\+T_{\*p} \mu}\tp{\boldsymbol{\theta}} = g_{\mu}\tp{\tp{1 - p_i + p_i \theta_i}_{i \in \stp{d}}}.
    \]
    Hence, for every \(\*x \in \bb N^d\),
    \[
    \partial^{\*x} g_{\+T_{\*p} \mu}\tp{\boldsymbol{\theta}} = \*p^{\*x} \tp{\partial^{\*x} g_{\mu}}\tp{\tp{1 - p_i + p_i \theta_i}_{i \in \stp{d}}}.
    \]
    By \Cref{thm:ulc-generating-function-characterization}, each nonzero \(\partial^{\*x} g_{\mu}\) is log-concave on \(\bb R_{> 0}^d\). Its composition with the displayed affine map is log-concave, so \(g_{\+T_{\*p} \mu}\) is strongly log-concave. The support of \(\+T_{\*p} \mu\) is downward closed, and applying \Cref{thm:ulc-generating-function-characterization} completes the proof.
\end{proof}

\Cref{prop:ulc-thinning-closure} admits an equivalent test-function formulation: the Poisson semigroup preserves log-concavity. The one-dimensional special case is well established (e.g., \cite[Proposition~2.2]{LRS25}) and is the discrete analogue of preservation of log-concavity by the heat semigroup \cite{Pre73}.

\begin{corollary}
    Let \(\tp{P_t}_{t \ge 0}\) denote the semigroup of the pure birth process on \(\bb N^d\) with birth rates \(\*b = \tp{b_1, \dots, b_d}\). If \(f: \bb N^d \to \bb R_{\ge 0}\) is log-concave, then \(P_t f\) is log-concave whenever it is finite. Here log-concavity is understood in the sense of \eqref{eq:discrete-function-log-concavity}.
    \label{cor:poisson-semigroup-preserves-log-concavity}
\end{corollary}

\begin{remark}
    The equivalence between \Cref{prop:ulc-thinning-closure} and \Cref{cor:poisson-semigroup-preserves-log-concavity} can also be established more intrinsically, using the fact that binomial thinning is the time reversal of the Poisson semigroup.
\end{remark}

\begin{proof}[Proof of \Cref{cor:poisson-semigroup-preserves-log-concavity}]
    The result is immediate if \(f\) is identically zero, so assume otherwise. For \(N \in \bb N\), let
    \[
    f_N\tp{\*n} \defeq f\tp{\*n} \*1_{\set{\*n \in \bb N^d \cmid \norm{\*n}_{\infty} \le N}}.
    \]
    Directly from \eqref{eq:discrete-function-log-concavity}, each \(f_N\) is log-concave. Fix an arbitrary \(\*q \in \bb R_{> 0}^d\), set \(\boldsymbol{\lambda} \defeq \*q + t \*b\), and let \(\*p \defeq \tp{q_i / \lambda_i}_{i \in \stp{d}}\). Define
    \[
    \pi_{\boldsymbol{\lambda}} \defeq \bigotimes_{i = 1}^d \-{Pois}\tp{\lambda_i}, \qquad Z_N \defeq \*E_{\pi_{\boldsymbol{\lambda}}}\stp{f_N}, \qquad \mu_N\tp{\*n} \defeq \frac{f_N\tp{\*n} \pi_{\boldsymbol{\lambda}}\tp{\*n}}{Z_N}.
    \]
    By \Cref{prop:ulc-poisson-relative-log-concavity}, \(\mu_N\) is ULC. Since \(\*p \in \left(0, 1\right]^d\), \Cref{prop:ulc-thinning-closure} shows that \(\+T_{\*p} \mu_N\) is ULC. A direct calculation using \(\lambda_i p_i = q_i\) and \(\lambda_i \tp{1 - p_i} = b_i t\) gives
    \[
    \tp{\+T_{\*p} \mu_N}\tp{\*n} = \frac{\pi_{\*q}\tp{\*n}}{Z_N} P_t f_N\tp{\*n}, \qquad \pi_{\*q} \defeq \bigotimes_{i = 1}^d \-{Pois}\tp{q_i}.
    \]
    Thus \Cref{prop:ulc-poisson-relative-log-concavity} implies that \(P_t f_N\) is log-concave. As \(N \to \infty\), monotone convergence gives \(P_t f_N \to P_t f\) pointwise. The cone of positive semidefinite matrices is closed, so the inequalities in \eqref{eq:discrete-function-log-concavity} pass to the limit whenever \(P_t f\) is finite. Moreover, \(\supp\tp{P_t f} = \supp\tp{f}\), which is downward closed. Hence \(P_t f\) is log-concave.
\end{proof}

\subsection{Convolution}

Closure of one-dimensional ULC measures under convolution is a classical result and admits several proofs \cite{Walk76,Lig97,Joh07,WY07,Gur09b,KN11,BH20}. We prove the multivariate analogue via the Lorentzian-polynomial framework of Br\"and\'en and Huh \cite{BH20}. The key tool is the homogenization operator introduced earlier in \Cref{subsec:ulc-completely-log-concavity}.

\begin{theorem}
    The class of ULC measures is closed under convolution. That is, if \(\mu\) and \(\nu\) are ULC measures on \(\bb N^d\), then so is their convolution \(\mu * \nu\).
    \label{thm:ulc-convolution-closure}
\end{theorem}

\begin{proof}[Proof of \Cref{thm:ulc-convolution-closure}]
    First suppose that \(\mu\) and \(\nu\) have finite support. Let \(\+H_N\) denote the homogenization operator introduced in \Cref{def:homogenized-generating-polynomial}. For every sufficiently large \(N\), define
    \[
    F_{\mu, N}\tp{\boldsymbol{\theta}} \defeq N! \, \+H_N g_{\mu}\tp{1, \boldsymbol{\theta} / N} = \sum_{\*n \in \supp\tp{\mu}} \mu\tp{\*n} \frac{N!}{\tp{N - \norm{\*n}_1}! N^{\norm{\*n}_1}} \boldsymbol{\theta}^{\*n},
    \]
    and define \(F_{\nu, N}\) analogously. By \Cref{lem:ulc-homogenization-lorentzian}, \(\+H_N g_{\mu}\) and \(\+H_N g_{\nu}\) are Lorentzian. Their product is Lorentzian by \cite[Corollary~2.32]{BH20}. By the equivalence between Lorentzian and homogeneous strongly log-concave polynomials \cite[Theorem~2.30]{BH20}, the product \(\+H_N g_{\mu} \+H_N g_{\nu}\) is strongly log-concave. Since strong log-concavity is preserved under restriction to \(\theta_0 = 1\), positive diagonal rescaling of the remaining variables, and multiplication by a positive constant, it follows that \(F_{\mu, N} F_{\nu, N}\) is strongly log-concave.

    As \(N \to \infty\), these polynomials and all their partial derivatives converge pointwise to those of \(g_{\mu} g_{\nu}\). Log-concavity of every derivative passes to the limit, so \(g_{\mu} g_{\nu}\) is strongly log-concave. Since \(g_{\mu * \nu} = g_{\mu} g_{\nu}\) and \(\supp\tp{\mu * \nu}\) is downward closed, \Cref{thm:ulc-generating-function-characterization} proves the claim for finite support.

    For general \(\mu\) and \(\nu\), let \(\mu_N\) and \(\nu_N\) be their normalized restrictions to \(\set{\*n \in \bb N^d \cmid \norm{\*n}_{\infty} \le N}\). These measures are ULC by \Cref{prop:ulc-basic-closure-properties}, so \(\rho_N \defeq \mu_N * \nu_N\) is ULC by the finite-support case. Let \(\Omega_N^{\mu}\) and \(\Omega_N^{\nu}\) denote the two restricted supports, and write \(\rho \defeq \mu * \nu\). For each fixed \(\*n\), every \(\*m\) entering its ULC matrix satisfies, for sufficiently large \(N\),
    \[
    \rho_N\tp{\*m} = \frac{\rho\tp{\*m}}{\mu\tp{\Omega_N^{\mu}} \nu\tp{\Omega_N^{\nu}}}, \quad \forall \*m \in \set{\*n} \cup \set{\*n + \*e_i \cmid i \in \stp{d}} \cup \set{\*n + \*e_i + \*e_j \cmid i, j \in \stp{d}}.
    \]
    The common factor cancels from the local ratios, so the ULC matrices of \(\rho_N\) and \(\rho\) at \(\*n\) agree for all sufficiently large \(N\). Hence \(\rho\) is ULC.
\end{proof}

The convolution theorem has the following test-function form. Its one-dimensional special case is a classical result due to Walkup \cite{Walk76}.

\begin{definition}
    For \(f, g: \bb N^d \to \bb R_{\ge 0}\), their binomial convolution \(h\) is defined by
    \[
    h\tp{\*n} \defeq \sum_{\*0 \le \*k \le \*n} f\tp{\*k} g\tp{\*n - \*k} \prod_{i = 1}^d \binom{n_i}{k_i}, \quad \*n \in \bb N^d.
    \]
    \label{def:binomial-convolution}
\end{definition}

\begin{corollary}
    Log-concavity is preserved under binomial convolution. That is, if \(f\) and \(g\) are log-concave in the sense of \eqref{eq:discrete-function-log-concavity}, then so is their binomial convolution \(h\).
    \label{cor:binomial-convolution-preserves-log-concavity}
\end{corollary}

\begin{proof}[Proof of \Cref{cor:binomial-convolution-preserves-log-concavity}]
    The claim is immediate if \(f\) or \(g\) is identically zero, so assume otherwise. Let \(f_N\) and \(g_N\) be their restrictions to \(\set{\*n \in \bb N^d \cmid \norm{\*n}_{\infty} \le N}\), and let \(h_N\) be their binomial convolution. Define
    \[
    A_N \defeq \sum_{\*n \in \bb N^d} \frac{f_N\tp{\*n}}{\*n!}, \qquad B_N \defeq \sum_{\*n \in \bb N^d} \frac{g_N\tp{\*n}}{\*n!}, \qquad \mu_N\tp{\*n} \defeq \frac{f_N\tp{\*n}}{A_N \*n!}, \qquad \nu_N\tp{\*n} \defeq \frac{g_N\tp{\*n}}{B_N \*n!}.
    \]
    The truncations \(f_N\) and \(g_N\) are log-concave, so \(\mu_N\) and \(\nu_N\) are ULC. By \Cref{thm:ulc-convolution-closure}, \(\rho_N \defeq \mu_N * \nu_N\) is ULC, and
    \[
    \*n! \rho_N\tp{\*n} = \frac{h_N\tp{\*n}}{A_N B_N}.
    \]
    Thus \(h_N\) is log-concave. For every fixed \(\*n\), one has \(h_N\tp{\*m} = h\tp{\*m}\) for all sufficiently large \(N\) and every \(\*m\) entering the local matrix at \(\*n\). Hence the local matrices of \(h_N\) and \(h\) eventually agree. Since \(\supp\tp{h} = \supp\tp{f} + \supp\tp{g}\) is downward closed, \(h\) is log-concave.
\end{proof}

\subsection{Particlewise Stochastic Projection}

The characterization of ULC measures via complete log-concavity of the probability-generating function in \Cref{thm:ulc-generating-function-characterization} gives closure under particlewise stochastic projection.

\begin{definition}
    Let \(\mu\) be a probability measure on \(\bb N^d\), and let \(\*P \in \bb R_{\ge 0}^{d \times m}\) satisfy \(\sum_{j = 1}^m P_{ij} \le 1\) for all \(i \in \stp{d}\). The particlewise stochastic projection channel associated with \(\*P\) is defined as follows. Let \(\*X \sim \mu\) be the input. Conditional on \(\*X = \*n\), regard \(\*n\) as a configuration of \(d\) types of particles, with \(n_i\) particles of type \(i\). There are \(m\) possible output types. Independently, each particle of input type \(i\) is assigned output type \(j\) with probability \(P_{ij}\), and is discarded with probability \(1 - \sum_{j = 1}^m P_{ij}\). The output \(\*X' \in \bb N^m\) is the vector of resulting particle counts of the \(m\) output types.
\end{definition}

The particlewise stochastic projection channel subsumes several natural operations. For example, if \(\*P = \diag\tp{\*p}\) for some \(\*p \in \left(0, 1\right]^d\), then the channel reduces to coordinatewise binomial thinning \(\+T_{\*p}\), as defined in \eqref{eq:binomial-thinning-operator}. If \(\*P \in \bb R_{\ge 0}^{\stp{d} \times S}\) is the zero-extension of the identity matrix \(\*I_S\) for some \(S \subseteq \stp{d}\), then the channel reduces to coordinate marginalization onto \(S\). If \(\*P = \*1_{d \times 1}\), then the channel maps \(\*X\) to the sum of its coordinates.

\begin{lemma}
    Particlewise stochastic projection preserves downward closed support. More precisely, let \(\mu\) be a probability measure on \(\bb N^d\) with downward closed support, and let \(\nu\) be the output law of the particlewise stochastic projection channel associated with any feasible \(\*P\). Then \(\supp\tp{\nu}\) is downward closed.
    \label{lem:stochastic-projection-downward-closed-support}
\end{lemma}

\begin{proof}[Proof of \Cref{lem:stochastic-projection-downward-closed-support}]
    Let \(\*n' \in \supp\tp{\nu}\). By definition, there exists \(\*n \in \supp\tp{\mu}\) together with a particlewise projection plan of positive probability that maps \(\*n\) to \(\*n'\). Fix any \(\*m' \le \*n'\). Starting from the same projection plan, remove precisely those input particles whose projected output particles are present in \(\*n'\) but not in \(\*m'\). Let \(\*m \le \*n\) be the resulting input configuration. Since \(\supp\tp{\mu}\) is downward closed, \(\*m \in \supp\tp{\mu}\). Applying the same projection choices to \(\*m\) produces exactly \(\*m'\), and this event still has positive probability. Therefore, \(\supp\tp{\nu}\) is downward closed.
\end{proof}

\begin{theorem}
    Let \(\mu\) be a ULC measure on \(\bb N^d\), and let \(\*P \in \bb R_{\ge 0}^{d \times m}\) satisfy \(\sum_{j = 1}^m P_{ij} \le 1\) for all \(i \in \stp{d}\). Suppose that \(\*X \sim \mu\), and let \(\*X'\) be the output of the particlewise stochastic projection channel associated with \(\*P\). Then \(\nu \defeq \-{Law}\tp{\*X'}\) is ULC on \(\bb N^m\).
    \label{thm:ulc-stochastic-projection-closure}
\end{theorem}

\begin{proof}[Proof of \Cref{thm:ulc-stochastic-projection-closure}]
    We prove a slightly more general statement. Let \(\*A \in \bb R_{\ge 0}^{d \times m}\) and \(\*b \in \bb R_{\ge 0}^d\) satisfy \(\sum_{j = 1}^m A_{ij} + b_i > 0\) for all \(i \in \stp{d}\). We claim that
    \[
    h\tp{\boldsymbol{\theta}} \defeq \frac{g_{\mu}\tp{\*A \boldsymbol{\theta} + \*b}}{g_{\mu}\tp{\*A \*1 + \*b}}, \quad \boldsymbol{\theta} \in \bb R^m
    \]
    is the probability-generating function of a ULC measure on \(\bb N^m\). The particlewise stochastic projection channel is the special case
    \[
    \*A = \*P, \qquad \*b = \tp{1 - \sum_{j = 1}^m P_{ij}}_{i \in \stp{d}}.
    \]

    Since \(g_{\mu}\) has nonnegative coefficients, so does \(h\), and \(h\tp{\*1} = 1\). Hence \(h\) is the probability-generating function of some probability measure \(\nu\) on \(\bb N^m\). Set
    \[
    \*r \defeq \tp{r_i}_{i \in \stp{d}}, \qquad r_i \defeq \sum_{j = 1}^m A_{ij} + b_i, \quad i \in \stp{d},
    \]
    and define
    \[
    \wt{\*P}_{ij} \defeq \frac{A_{ij}}{r_i}, \quad i \in \stp{d}, \, j \in \stp{m}.
    \]
    Since \(\sum_{j = 1}^m \wt{\*P}_{ij} = 1 - b_i / r_i\) for all \(i \in \stp{d}\), \(\wt{\*P}\) defines a particlewise stochastic projection channel. Let
    \[
    \wt{h}\tp{\boldsymbol{\theta}} \defeq g_{\mu}\tp{\wt{\*P} \boldsymbol{\theta} + \tp{1 - \sum_{j = 1}^m \wt{\*P}_{ij}}_{i \in \stp{d}}} = g_{\mu}\tp{\diag\tp{\*r}^{-1} \tp{\*A \boldsymbol{\theta} + \*b}}, \quad \boldsymbol{\theta} \in \bb R^m.
    \]
    Then \(\wt{h}\) is the probability-generating function of the output law \(\wt{\nu}\) of the particlewise stochastic projection channel associated with \(\wt{\*P}\) and input law \(\mu\). Since \(\*r \in \bb R_{> 0}^d\), \(h\) and \(\wt{h}\) have the same coefficient support, so \(\supp\tp{\nu} = \supp\tp{\wt{\nu}}\). By \Cref{lem:stochastic-projection-downward-closed-support}, \(\supp\tp{\wt{\nu}}\) is downward closed, and thus so is \(\supp\tp{\nu}\).

    It remains, by \Cref{thm:ulc-generating-function-characterization}, to show that \(h\) is strongly log-concave. Let \(\*a_j\) denote the \(j\)-th column of \(\*A\). For every \(\*x \in \bb N^m\), the chain rule gives
    \[
    \partial^{\*x} h\tp{\boldsymbol{\theta}} = \frac{1}{g_{\mu}\tp{\*A \*1 + \*b}} \tp{\partial_{\*a_1}^{x_1} \dots \partial_{\*a_m}^{x_m} g_{\mu}}\tp{\*A \boldsymbol{\theta} + \*b}.
    \]
    By the characterization of ULC measures via complete log-concavity in \Cref{thm:ulc-generating-function-characterization}, \(g_{\mu}\) is completely log-concave. Since \(\*a_1, \dots, \*a_m \in \bb R_{\ge 0}^d\), \(\partial_{\*a_1}^{x_1} \dots \partial_{\*a_m}^{x_m} g_{\mu}\) is either identically zero or finite, positive, and log-concave on \(\bb R_{> 0}^d\). Moreover, \(\*A \boldsymbol{\theta} + \*b \in \bb R_{> 0}^d\) for every \(\boldsymbol{\theta} \in \bb R_{> 0}^m\), and composition with an affine map preserves log-concavity. Hence \(\partial^{\*x} h\) is either identically zero or finite, positive, and log-concave on \(\bb R_{> 0}^m\) for every \(\*x \in \bb N^m\). Thus \(h\) is strongly log-concave. Applying \Cref{thm:ulc-generating-function-characterization} once more shows that \(\nu\) is ULC, completing the proof.
\end{proof}

\end{document}